\documentclass[11pt,leqno]{amsart}
\usepackage{amsmath,amsfonts,amsthm,amscd,amssymb,graphicx}
\usepackage{epstopdf}
\numberwithin{equation}{section}

\newtheorem{theorem}{Theorem}[section]
\newtheorem{definition}[theorem]{Definition}
\newtheorem{corollary}[theorem]{Corollary}
\newtheorem{example}[theorem]{Example}

\newtheorem{lemma}[theorem]{Lemma}
\newtheorem{proposition}[theorem]{Proposition}
\newtheorem{remark}[theorem]{Remark}
\newtheorem{remarks}[theorem]{Remarks}
\newcommand{\f}{\theta,Y}
\newcommand{\s}{x,y}

\usepackage[margin=1in]{geometry}
\usepackage{lscape}
\usepackage{indentfirst}

\usepackage{caption}
\usepackage{subcaption}
\usepackage{float}
\usepackage{cite}

\def\g{\gamma}

\def\eps{\varepsilon }

\newcommand\R{\mathbb R}

\def\g{\gamma}

\def\eps{\varepsilon}

\newcommand\br{\begin{remark}}
\newcommand\er{\end{remark}}
\newcommand\bp{\begin{pmatrix}}
\newcommand\ep{\end{pmatrix}}
\newcommand{\be}{\begin{equation}}
\newcommand{\ee}{\end{equation}}
\newcommand\ba{\begin{equation}\begin{aligned}}
\newcommand\ea{\end{aligned}\end{equation}}

\newcommand{\bap}{\begin{app}}
\newcommand{\eap}{\end{app}}
\newcommand{\begs}{\begin{exams}}
\newcommand{\eegs}{\end{exams}}
\newcommand{\beg}{\begin{example}}
\newcommand{\eeg}{\end{exaplem}}
\newcommand{\bpr}{\begin{proposition}}
\newcommand{\epr}{\end{proposition}}
\newcommand{\bt}{\begin{theorem}}
\newcommand{\et}{\end{theorem}}
\newcommand{\bc}{\begin{corollary}}
\newcommand{\ec}{\end{corollary}}
\newcommand{\bl}{\begin{lemma}}
\newcommand{\el}{\end{lemma}}
\newcommand{\bd}{\begin{definition}}
\newcommand{\ed}{\end{definition}}
\newcommand{\brs}{\begin{remarks}}
\newcommand{\ers}{\end{remarks}}

\newcommand{\CalF}{\mathcal{F}}

\newcommand{\CalG}{\mathcal{G}}

\newcommand{\sgn}{\text{\rm sgn}}
\newcommand{\bbN}{{\mathbb{N}}}
\newcommand{\bbR}{{\mathbb{R}}}

\newcommand{\bbZ}{{\mathbb{Z}}}

\newcommand{\bbT}{{\mathbb{T}}}

\newcommand{\beq}{\begin{equation}}
\newcommand{\eeq}{\end{equation}}
\newtheorem{thm}{Theorem}[section]
\newtheorem{prop}[thm]{Proposition}

\newtheorem{lem}[thm]{Lemma}
\newtheorem{defn}[thm]{Definition}
\newtheorem{ass}[thm]{Assumption}

\newtheorem{rem}[thm]{Remark}

\newtheorem{exams}[thm]{Examples}
\newtheorem{notation}[thm]{Notation}
\numberwithin{equation}{section}

\def\({\left(\begin{array}{cccccc}}
\def\){\end{array}\right)}

\newcommand{\pf}{\begin{proof}}
\newcommand{\foorp}{\end{proof}}
\title
{Geometric optics for Rayleigh wavetrains in d-dimensional nonlinear elasticity}

\author{Aric Wheeler and Mark Williams}
\begin{document}

\begin{abstract}

A Rayleigh wave is a type of surface wave that propagates in the boundary of an elastic solid with traction (or Neumann) boundary conditions.  
Since the 1980s much work has been done on the problem of constructing a leading term in an \emph{approximate} solution to the rather complicated second-order quasilinear hyperbolic boundary value problem with fully nonlinear Neumann boundary conditions that governs  the propagation of Rayleigh waves.  The question has remained open whether or not this leading term approximate solution is really close in a precise sense to the exact solution of the governing equations.   We prove a positive answer to this question for  the case of Rayleigh wavetrains in any space dimension $d\geq 2$.  The case of Rayleigh pulses in dimension $d=2$ has already been treated by Coulombel and Williams.  For highly oscillatory Rayleigh wavetrains we are able to construct high-order approximate solutions consisting of the leading term plus an arbitrary number of correctors.  Using those high-order solutions we then perform an error analysis which shows (among other things) that for small wavelengths the leading term is close to the exact solution in $L^\infty$ on a fixed time interval independent of the wavelength.  The error analysis is carried out in a somewhat general setting and is applicable to other types of waves for which high order approximate solutions can be constructed.

\end{abstract}

\date{\today}
\maketitle

%2
%\frame{
\tableofcontents
\newpage

\part{General introduction and main results}\label{sec:Intro}

The problem of constructing oscillatory, approximate solutions to the traction problem in nonlinear elasticity \eqref{j0} has been considered by a number of authors including \cite{Lardner,Hunter06,Benzoni-Gavage,BCproc,Williams}.  These papers construct 2-scale, WKB-type, approximate solutions consisting of a leading term and, in the case of \cite{Williams}, a first corrector as well.    This paper provides a ``rigorous justification" (explained below) of such approximate solutions for wavetrains in any number of space dimensions $d\geq 2$.

The traction problem is a Neumann-type boundary value problem for a quasilinear second-order hyperbolic system, and like most Neumann-type problems the boundary conditions exhibit a certain degeneracy - the uniform Lopatinski condition fails.  The specific nature of the failure in this case is manifested by the presence of surface waves, or Rayleigh waves, that propagate in the boundary along characteristics of the Lopatinski determinant.    Approximate surface wave solutions  of both pulse-type   and wavetrain-type   (defined below) have  been constructed.

The task of rigorously justifying these approximate solutions, that is, showing that they are close in  a precise sense to true exact solutions of the traction problem was begun in \cite{Williams}, which provided a justification in the case of pulses in \emph{two} space dimensions; see Remark \ref{pvw}.   
In this paper we treat the case of wavetrains in any space dimension $d\geq 2$ by entirely different methods.  

%The construction of approximate solutions is valid in any number of space dimensions; the restriction to space dimension two  in \cite{Williams} or to dimension three in this paper (CHECK restriction for wavetrains) applies only to the rigorous error analysis. 

%\emph{\quad}Our main focus in the remaining chapters  is to provide rigorous answers to the basic questions of geometric optics for surface \emph{pulses} in isotropic hyperelastic materials.    We will describe how the results extend to wavetrains in section \ref{wavetrains}.   Although many pieces of the argument  work just as well in higher dimensions,  there is one piece that requires us to assume $d=2$ in our main results, Theorems \ref{uniformexistence} and \ref{approxthm}. 
%At several points our estimates rely on the use of Kreiss symmetrizers and, as we explain in section \ref{higherD}, there is a serious difficulty with the construction of Kreiss symmetrizers for linearized elasticity in $d\geq 3$.

We now describe the form of the traction problem to be considered here.
   Let the unknown  $\phi=(\phi_1,\dots,\phi_d)(t,x)$ represent the deformation of  an isotropic, hyperelastic,  Saint Venant-Kirchhoff (SVK) material  whose reference (or undeformed) configuration is $\omega=\{x=(x_1,\dots,x_d):x_d>0\}$.    We will study the case where the  deformation results from the application of  a surface force $g=g(t,x)$, but our methods extend easily to the case where forcing is applied in the interior as well.   Here $\phi(t,\cdot):\omega\to \mathbb{R}^d$ and $g(t,\cdot):\partial\omega\to \mathbb{R}^d$.
%\footnote{At the moment I don't see why we can't work in any dimension $d\geq 2$. Did Sable-Tougeron take $d=2$ just because she needed strict hyperbolicity to get Kreiss symmetrizers in 1985-6, or was there a reason connected with the elasticity equations in $d\geq 3$? Check this.}   
The equations are  a second-order, nonlinear $d\times d$ system
%\footnote{The interior equation is quasilinear, while the boundary equation is fully nonlinear.}
\begin{align}\label{j0}
\begin{split}
&\partial_t^2\phi -\mathrm{Div}(\nabla\phi\;\sigma(\nabla\phi))=0\text{ in }x_d>0\\
& \nabla\phi\;\sigma(\nabla\phi)n=g\text{ on }x_d=0,\; \;\\
&\phi(t,x)=x \text{ and }g=0 \text{ in }t\leq 0,
\end{split}
\end{align}
where  $n=(0,\dots,0,-1)$  is  the outer unit normal to the boundary of $\omega$, $\nabla\phi=(\partial_{x_j}\phi_i)_{i,j=1,\dots,d}$ 
%(\nabla\phi)_{i,j=1}$, $i,j=1,2$ i
is the spatial gradient matrix,  $\sigma$ is the stress $\sigma(\nabla\phi)=\lambda \mathrm{tr} E\cdot I+2\mu E$ with Lam\'e constants $\lambda$ and $\mu$  satisfying $\mu>0$, $\lambda+\mu>0$,   and $E$ is the strain $E(\nabla\phi)=\frac{1}{2}({}^t\nabla\phi\cdot\nabla \phi-I)$.  Here ${}^t\nabla\phi$ denotes the transpose of $\nabla\phi$, $\mathrm{tr} E$ is the trace of the matrix $E$, and 
\begin{align}
\mathrm{Div}M=(\sum^d_{j=1}\partial_jm_{i,j})_{i=1,\dots,d}\text{ for a matrix }  M=(m_{i,j})_{i,j=1,\dots,d}.
\end{align}

%We recall that the stored energy function $W(E)$ for an SVK material  \eqref{SVK} is the leading part, quadratic in $E$, of the general isotropic hyperelastic energy given by \eqref{generalenergy}.\footnote{When rewritten in terms of $\nabla\phi$ or $\nabla U=\nabla \phi -I$, the stored energy function is fourth order in those arguments.}  In section \ref{generalisotropic} we explain how our results extend to this more general case.  It will simplify the exposition to work initially with the SVK problem, which contains all the main difficulties.

The system \eqref{j0} has the form of a second-order quasilinear system with fully nonlinear Neumann boundary conditions:
\begin{align}\label{j1}
\begin{split}
&\partial_t^2 \phi -  \sum_{i=1}^d \partial_{x_i}(\mathcal{A}_i(\nabla\phi))=0\text{ in }x_d>0\\
&\sum^d_{i=1}n_i\mathcal{A}_i(\nabla\phi)=g \text{ on }x_d=0\\
&\phi(t,x)=x \text{ and }g=0 \text{ in }t\leq 0,
\end{split}
\end{align}
where  the real functions $\mathcal{A}_i(\cdot)$ are $C^\infty$ (in fact, polynomial) in their arguments.  
%The second form of the boundary condition is obtained  by an application of the implicit function theorem to the equation $B(\nabla \phi)=g$,  and is valid for $|\nabla\phi-I_{2\times 2}|_{L^\infty}$ small.    Set $x'=(t,x_1)$. 

Defining the displacement $U(t,x)=\phi(t,x)-x$, we rewrite \eqref{j1} as 
\begin{align}\label{j2}
\begin{split}
&\partial_t^2 U -  \sum_{i=1}^d \partial_{x_i}{A}_i(\nabla U))=0\text{ in }x_d>0\\
&\sum^d_{i=1}n_i{A}_i(\nabla U)=g \text{ on }x_d=0\\
&U(t,x)=0 \text{ and }g=0 \text{ in }t\leq 0,
\end{split}
\end{align}
where the functions $A_i$ are related to $\mathcal{A}_i$ in the obvious way.   
%\begin{split}
%&\partial_t^2 U+\sum_{|\alpha|=2} A_\alpha(\nabla U)\partial_x^\alpha U=0\text{ in }x_2>0\\
%&h(\nabla U)=g\text{ or }\partial _{x_2} U=H(\partial_{x_1} U,g)\text{ on }x_2=0\\
%&U(t,x)=0 \text{ and }g(t,x_1)=0 \text{ in }t\leq 0,
%\end{split}
Writing $x'=(x_1,\dots,x_{d-1})$ for the tangential spatial variables, we  take highly oscillatory wavetrain boundary data with the \emph{weakly nonlinear} scaling (defined below): 
\begin{align}\label{j3}
g=g^\eps(t,x')=\eps^2 G\left(t,x',\frac{\beta\cdot (t,x')}{\eps}\right),\;\;\eps\in (0,1],
\end{align}
where $G(t,x',\theta)\in H^\infty(\mathbb{R}_t\times \mathbb{R}^{d-1}_{x'}\times \mathbb{T}_\theta)$.  Here  $\beta\in \mathbb{R}^d\setminus 0$ is a frequency in the elliptic region of  (the linearization at $\nabla U=0$ of) \eqref{j2}, chosen so that the uniform Lopatinskii condition fails at $\beta$.\footnote{See section \ref{introp3} for definitions and more detail on the choice of $\beta$.   The existence of such special frequencies $\beta$ was predicted by Lord Rayleigh \cite{Ray}.}   
%We may take $\beta=(\pm\tau_r(\mathbf{\eta}),\mathbf{\eta})$, for $\tau_r(\mathbf{\eta})$ as in (H3) of chapter \ref{chapter2} and $\bf{\eta}=1$.}
We will refer to $\beta$ as a \emph{Rayleigh frequency}.  We expect the response $U^\eps(t,x)$ to be a Rayleigh wavetrain
 propagating in the boundary.

We note that for a fixed $\eps$ the existence of an exact solution $U^\eps$  of \eqref{j2} on a time interval $(-\infty,T_\eps]$  follows from the main result of \cite{SN}. Since Sobolev norms of $g^\eps(t,x)$ clearly blow up as $\eps\to 0$, the times of existence $T_\eps$ provided by \cite{SN} converge to zero as $\eps\to 0$.   
One of the goals of this paper is to show that solutions actually exist on a fixed time interval independent of $\eps$.  The strategy, which is an example of a general method going back to \cite{Gues},  is to first  construct high-order approximate solutions on a fixed time interval independent of $\eps$, and then to construct exact solutions (which are known to be unique) that are ``close" to the approximate solutions on a time interval that is possibly shorter but still independent of $\eps$.   Carrying this out will at the same time achieve the second main goal of the paper, which is to show that the (first-order) approximate solutions constructed by earlier authors are indeed close to the exact solutions for $\eps$ small.

In part \ref{p3} (Theorem \ref{theo:Error}) we construct  high-order approximate solutions $U^\eps_a$ of \eqref{j2}  on a time interval $(\-\infty,T]$,  with  $T>0$  independent of $\eps$, of the form
\begin{align}\label{j4}
U_{a}^\varepsilon(t,x)=\sum_{k=2}^{N}\varepsilon^k U_k(t,x,\frac{\beta\cdot (t,x')}{\varepsilon},\frac{x_d}{\varepsilon}),
\end{align}
where the ``profiles" $U_k$ are \emph{bounded} and can be written 
\begin{align}
U_k(t,x,\theta,Y)=\underline{U_k}(t,x)+U_k^*(t,x',\theta,Y),
\end{align}
with $U^*_k$ periodic in $\theta$ and exponentially decaying in $Y$.\footnote{More precisely, $U_k$ lies in the space $S$ of Definition \ref{d1a}.   The sense in which $U^\varepsilon_a$ is an \emph{approximate} solution is specified in Theorem \ref{theo:Error}.}

 We stated above that the scaling in the choice of boundary data \eqref{j3} is the \emph{weakly nonlinear} scaling.  For the wavelength $\eps$, the weakly nonlinear scaling of $g^\eps(t,x')$  is by definition the smallest amplitude, in this case $\eps^2$, for which the equations for the leading order profile $U_2$  are nonlinear.  A higher power of $\eps$ would lead to a linear problem for $U_2$.  The weakly nonlinear scaling was identified by Lardner \cite{Lardner}, who made a first attempt to formulate  the profile equation for the leading term $U_2$ in \eqref{j4}.  When the Fourier mean of $G$, $\underline{G}(t,x')$, is zero, it turns out that $U_2=U_2^*$ and is completely determined by its trace on the boundary.  The equation for that trace, which is usually referred to as \emph{the amplitude equation},  is a nonlinear, Burgers-type equation  \eqref{eq:BilFourMult} that is \emph{nonlocal} in $\theta$. 

The following theorem summarizes our main results, Theorems \ref{main} and \ref{maincor}, as applied to nonlinear elasticity in the case where the boundary forcing $G\in H^\infty$. 
   Those results are more precise than  Theorem \ref{mainintro}; for example, they allow boundary data of finite smoothness and give precise information on the regularity and ``size" of  both the exact solution $U^\eps(t,x)$ and the profiles $U_k(t,x,\theta,Y)$.
   
   \begin{thm}\label{mainintro}
   Consider the SVK problem \eqref{j2} with boundary data \eqref{j3}, where  $G(t,x',\theta)\in H^\infty(\mathbb{R}_t\times \mathbb{R}^{d-1}_{x'}\times \mathbb{T}_\theta)$.  There exist constants $\eps_0>0$ and $T>0$ such that the exact solution $U^\eps(t,x)$ of \ref{j2} exists and is $C^\infty$ on $\Omega:=(-\infty, T]\times \mathbb{R}^{d-1}_{x'}\times \mathbb{R}^+_{x_d}$ and satisfies for any $p\geq 2$
   \begin{align}\label{j5}
  \left|U^\eps - (\eps^2U_2+\dots+\eps^p U_p)|_{\theta=\frac{\beta\cdot (t,x')}{\eps},Y=\frac{x_d}{\eps}}\right|_{L^\infty(\Omega)} = o(\eps^p), \;\;\eps\in (0,\eps_0],
\end{align}
 where $\eps^2U_2+\dots+\eps^p U_p$ is the approximate solution constructed in Theorem \ref{theo:Error}.
\end{thm}

For any given choice of $p$ in \eqref{j5}, the proof of the theorem depends on being able to construct a number of bounded profiles beyond $U_p$.  
This theorem implies that the explicit qualitative information contained in the approximate solution really does apply to the exact solution.  This includes information about amplitude, group velocity (Remark \ref{h5}), and internal rectification (Remark \ref{internal}).  Moreover, the estimate shows \eqref{j5} shows that even higher profiles contain information about the exact solution.

\begin{rem}[Pulses versus wavetrains]\label{pvw}
When the function $G(t,x',\theta)$ in \eqref{j3} belongs to  $H^\infty(\mathbb{R}_t\times \mathbb{R}^{d-1}_{x'}\times \mathbb{R}_\theta)$ instead of  $H^\infty(\mathbb{R}_t\times \mathbb{R}^{d-1}_{x'}\times \mathbb{T}_\theta)$, we obtain Rayleigh pulses instead of Rayleigh wavetrains.  The Fourier spectrum of $G$, that is, the $k$-support of  $\hat{G}(t,x',k)$, is now a subset of $\mathbb{R}$ rather than $\mathbb{Z}$ and may include a full neighborhood of $k=0$.  A consequence of this is that the decay of pulses in the interior (the $Y$ decay) is much weaker than for wavetrains, generally only  polynomial rather than exponential decay.  Since the construction of successive correctors $U_k$, $k\geq 3$ requires successive integrations over the unbounded domain $\mathbb{R}_\theta\times \mathbb{R}_Y$, starting at $k=4$ the profile $U_k$  becomes too large for $\eps^kU_k$ to be considered a ``corrector".  Thus, in the pulse case it is impossible to construct high-order approximate solutions, and  the method of this paper  cannot be used to justify approximate solutions.  

In the paper \cite{Williams} approximate, leading-term pulse solutions $\eps^2U_2$  were justified in the case of two space dimensions by a different method that relied on  precise microlocal energy estimates for certain ``singular systems" associated to \eqref{j2}.   These estimates were proved using Kreiss symmetrizers \cite{K} for the linearized SVK problem whose construction relied on the fact that in 2D,  the linearized problem is strictly hyperbolic.  For space dimensions three and higher, the linearized SVK problem fails to be strictly hyperbolic, and in fact exhibits complicated characteristics of \emph{variable} multiplicity.     Although Kreiss symmetrizers are now available in a number of situations where strict hyperbolicity fails, including some cases of variable multiplicity \cite{Met,MZ}, the linearized SVK problem in higher dimensions lies beyond the reach of current Kreiss-symmetrizer technology. 
The method of \cite{Williams} can also be used to justify approximate \emph{wavetrain} solutions $\eps^2U_2$ in 2D, but the lack of Kreiss symmetrizers prevents us from using that method for either pulses or wavetrains in space dimensions $\geq 3$.
\end{rem}

\subsection{Organization of the paper}
In part \ref{p2} we consider a class of second-order, quasilinear hyperbolic problems with Neumann boundary conditions \eqref{mainproblem}  that includes the SVK system \eqref{j2}.  The boundary data \emph{can} be taken to be oscillatory wavetrain data of the form \eqref{j3}, but it could have other forms.   We \emph{assume} that we are given an approximate solution $u^\eps_a$ of ``high-order" in the sense of Assumption \ref{ass1} on some fixed time interval independent of $\eps$,  and we seek a nearby exact solution of the form $u^\eps=u^\eps_a+v^\eps$.    The essential step is to study  the ``error problem" satisfied by $v^\eps$ in order to show that $v^\eps$ exists on a fixed time interval independent of $\eps\in(0,\eps_0]$ for some $\eps_0>0$,  and that $v^\eps$ is ``small" compared to $u^\eps_a$.     This is accomplished in the proof of  the main result of this part, Theorem \ref{main}. 

The proof of Theorem \ref{main}  has two main components.   The first is the well-posedness theory of \cite{Sh,SN} for problems of the form \eqref{mainproblem}.   The needed results are summarized in section \ref{apri}, and we apply the local existence and continuation results, Propositions \ref{localex} and \ref{continuation},  to the problem \ref{mainproblem} with $\eps$ \emph{fixed}.  As already noted this yields solutions on time intervals $T_\eps$ that apparently converge to zero as $\eps\to 0$.   These well-posedness results are based on estimates for  linearized hyperbolic and elliptic problems associated to \eqref{mainproblem} that are stated in Propositions \ref{hypest} and \ref{ellest}.   These estimates were proved by taking advantage of the special structure of the Neumann problem \eqref{mainproblem}.  Although they are not as precise as the microlocal estimates that can be proved with Kreiss symmetrizers in the case of two space dimensions, they are sufficiently precise for our  purposes here.  Roughly speaking, the relative lack of precision is compensated for by the fact that we now have a high-order approximate solution to work with.  

The second main component is the proof of simultaneous a priori  estimates for a trio of coupled nonlinear problems consisting of \eqref{mainproblem} and two other problems derived from it -  the problems \eqref{H1} (which is equivalent to \eqref{mainproblem}), \eqref{H2}, and \eqref{E}.    These estimates, which are summarized in Proposition  \ref{mainestimate}, are carried out in section \ref{energyestimates}.  The estimate of Proposition \ref{mainestimate} is given in terms of an $\eps$-dependent energy $\mathcal{E}^s_\eps(t)$ (Definition \ref{energy}) for the coupled systems.   An essential feature of this estimate for the purpose of obtaining a time of existence independent of $\eps$ is that it is \emph{uniform}: the constants that appear in it are either independent of $\eps$ or converge to zero as $\eps\to 0$.   The two components are put together in section \ref{mainR} in a continuous induction argument proving Theorem \ref{main}. 
The main novelties of part \ref{p2}  lie in the choice of an  energy $\mathcal{E}^s_\eps(t)$ that can be estimated ``without loss" in terms of itself as in Proposition \ref{mainestimate}, and in the proof of the uniform estimates of section \ref{energyestimates}. 

Part \ref{p3}  is devoted to the construction of high-order approximate solutions $U^\eps_a$ to the SVK system \eqref{j2}.   These solutions satisfy Assumption \ref{ass1}, so Theorem \ref{main} applies to show that they are close to exact solutions.  This result is formulated  in Theorem  \ref{maincor}, which is a more precise version of  Theorem \ref{mainintro}.  The writing of part \ref{p3} was strongly influenced by \cite{Marcou}, which treats surface waves for first-order conservation laws with linear, homogeneous boundary conditions $Cu=0$, and the second Chapter of \cite{Mar2}, which constructs approximate solutions consisting of a leading term $U_2$ and (part of a) first corrector $U_3$ for a simplified version of the SVK model.  We give more detail later about our debt to these works;   here we just note that the main novelty in our construction of approximate solutions lies in our construction of arbitrarily high order profiles.

In part \ref{tamewellp} we provide a new proof of a tame estimate for the amplitude equation \eqref{eq:AmplitudeEq2}. The solution of this equation determines the leading term in the approximate solution of SVK \eqref{j0}. A tame estimate is needed to obtain a time of existence that depends on a fixed low order of regularity of the solution, rather than a time that decreases to zero as higher regularity is considered.   Such an estimate (without slow spatial variables) was already proved in \cite{Hunter06} for a modified form of the amplitude equation; however, he did not use it to obtain a time of existence independent of high order regularity.  
%The proof we give here is simpler in one respect than that of \cite{Hunter06}, since  it depends only an \emph{obvious} estimate \eqref{hypotheseb} of the kernel $b(n_1,n_2,n_3)$ in  \eqref{eq:BilFourMult}.
%In addition our proof 
The proof we give here applies directly to the form of the amplitude equation given by \eqref{eq:AmplitudeEq2} and incorporates the slow spatial variables. 

Parts \ref{p2},  \ref{p3}, and \ref{tamewellp} are written so that they can be read independently of one another.  We include additional introductory material in the introductions to  parts \ref{p2} and \ref{p3}, sections \ref{introp2} and \ref{introp3}. 

\begin{rem}
Constants $C$,  $C_j$, $A_j$, $B_j$, etc. appearing in the various estimates are always independent of $\eps$ unless explicit $\eps$-dependence is indicated.
Also, we occasionally write    $|f|\lesssim |g|$ as shorthand for $|f|\leq C|g|$ for some $C>0$.  
\end{rem}

\part{Exact solutions near approximate ones}\label{p2}

\section{Introduction}\label{introp2}

Letting $(t,x)=(t,x',x_d)\in \mathbb{R}^{1+d}_+$ we consider the following $N\times N$ nonlinear hyperbolic problem with Neumann boundary conditions on a half-space:\footnote{We  write $A_i(D_xu)=A_i(\partial_1u,\dots,\partial_d u)$, where $u$ takes values in $\mathbb{R}^N$ and $A_i$ has $N$ real components.} 
%the (forward) traction problem of 3D (or $d$-dimensional, $d\geq 2$) nonlinear elasticity on a half-space:
\begin{align}\label{mainproblem}
\begin{split}
&P(u)=\partial_t^2 u-\sum^d_{i=1}\partial_i(A_i(D_xu))=0 \text{ in }x_d\geq 0\\
&Q(u)=\sum_{i=1}^d n_i A_i(D_xu)=g(t,x')\text{ on }x_d=0\\
&u=0 \text{ in }t\leq 0,
\end{split}
\end{align}
where $g=0$ in $t\leq 0$ and $\textbf{n}=(0,\dots,0,-1)\in \mathbb{R}^d$ is the outward unit normal.   Our main interest is in  highly oscillatory wavetrain boundary data
\begin{align}\label{bc}
g=g^\eps(t,x')=\eps^2 G\left(t,x',\frac{\beta\cdot (t,x')}{\eps}\right),
\end{align}
where $G(t,x',\theta)\in H^t(\mathbb{R}\times \mathbb{R}^{d-1}\times \mathbb{T})$ for some large $t$, but the results of this part depend on only on Assumptions \ref{a5a} and \ref{ass1} and thus apply to other types of boundary data.

The traction problem in nonlinear elasticity takes the form \eqref{mainproblem}, and as we will show  in part \ref{p3},  the 
 factor $\eps^2$ in \eqref{bc} gives the weakly nonlinear scaling.

We assume that a ``high-order" approximate solution $u^\eps_a(t,x)$ has been constructed.  This is by definition a  function with the regularity and growth properties listed in Assumption \ref{ass1} 
%\begin{align}\label{approx}
%u_a^\eps(t,x)=\eps^2U_2^\eps(t,x)+\dots+\eps^{M_1}U^\eps_{M_1}(t,x),
%\end{align}
which satisfies
\begin{align}\label{apsys}
\begin{split}
&P(u^\eps_a)=-\eps^{M}R^\eps:=-\mathcal{F}^\eps(t,x)\\
&Q(u^\eps_a)=g^\eps -\eps^{M}r^\eps:=g^\eps(t,x')-\mathcal{G}^\eps(t,x'),
\end{split}
\end{align}
where the functions $R^\eps$, $r^\eps$ have the properties given in Assumption \ref{ass1}, and $M>0$ is an integer that will be specified later.  In the case of the traction problem such an approximate solution is constructed in part \ref{p3}. 

We look for an exact solution of \eqref{mainproblem} in the form\footnote{The functions $u_a$ and $v$ clearly depend on $\eps$, $u_a=u^\eps_a$ and $v=v^\eps$, but we shall often suppress the superscripts $\eps$ on these and other obviously $\eps$-dependent functions.}
\begin{align}
u^\eps=u_a+v. 
\end{align}  
Let $v=\eps^M \nu$ and $w=\partial_t v=\eps^M\omega$, where $\omega=\partial_t\nu$.
The hyperbolic error problem that $v$ must satisfy is:

\begin{align}\label{a1}
\begin{split}
&(a) \partial_t ^2 v-\sum_{i=1}^d\partial_i[A_{i}(D_x(u_a+v))-A_i(D_xu_a)]=\CalF(t,x)\\
&(b) \sum_{i=1}^d n_i[A_{i}(D_x(u_a+v))-A_i(D_xu_a)]=\CalG(t,x') \text{ on }x_d=0.
\end{split}
\end{align}

Below we set $A_{ij}(U):=\partial A_i/\partial{U_j}$, where $U=(U_1,\dots,U_d)$ with $U_j\in\mathbb{R}^N$ a placeholder for $\partial_ju$. The problem \eqref{a1}  is equivalent to the following hyperbolic problem for $\nu$:
\begin{align}\label{H1}
\begin{split}
&(a) \partial_t^2\nu-\sum_{i,j=1}^d\partial_i[A_{ij}(D_x(u_a+v))\partial_j\nu]=\\
&\qquad \eps^{-M}[\CalF+\sum_{i=1}^d\partial_i\left(A_{i}(D_x(u_a+v))-A_i(D_xu_a)-\sum_{j=1}^d A_{ij}(D_x(u_a+v))\partial_jv\right):=\mathbb{F}_1\\
&(b) \sum_{i=1}^d n_i[A_{ij}(D_x(u_a+v))\partial_j\nu]=\\
&\qquad \eps^{-M}\left[\CalG-\left(\sum_i n_i[A_i(D_x(u_a+v))-A_i(D_x u_a)-\sum_j A_{ij}(D_x(u_a+v))\partial_jv]\right)\right]:=\mathbb{G}_1.
\end{split}
\end{align}

The function $\omega$ satisfies the hyperbolic problem obtained by differentiating \eqref{H1} with respect to $t$:\footnote{The form of the problems \eqref{H2} and \eqref{E} is similar to problems that appear in the induction scheme of \cite{SN}.}
\begin{align}\label{H2}
\begin{split}
&(a) \partial_t^2\omega-\sum_{i,j=1}^d\partial_i[A_{ij}(D_x(u_a+v))\partial_j\omega]= \eps^{-M}[\partial_t \CalF-H(\overline{D}_x^2v)]:=\mathbb{F}_2\\
&(b)\sum_{i=1}^d n_i[A_{ij}(D_x(u_a+v))\partial_j\omega]=\eps^{-M}[\partial_t\CalG-H_b(\overline{D}_xv)]:=\mathbb{G}_2,
\end{split}
\end{align}
where 
\begin{align}
\begin{split}
&H(\overline{D}_x^2v):=-\sum_{i,j=1}^d\partial_i[(A_{ij}(D_x(u_a+v))-A_{ij}(D_xu_a))\partial_j\partial_t u_a]\\
&H_b(\overline{D}_xv)=-\sum_{i,j=1}^d n_i(A_{ij}(D_x(u_a+v))-A_{ij}(D_xu_a))\partial_j\partial_t u_a.
\end{split}
\end{align}

The problem \eqref{H1} for $\nu$ can also be written as the following  \emph{elliptic} problem for $\nu$:
\begin{align}\label{E}
\begin{split}
&(a)-\sum_{i,j=1}^d \partial_i(A_{ij}(D_x u_a)\partial_j\nu)+\lambda\nu=\\
&\qquad-\partial_t\omega+\lambda\int^t_0\omega(s,x)ds+\eps^{-M}[\CalF(t,x)+E(\overline{D}^2_xv)]:=\mathbb{F}_3\\
&(b)\sum_{i,j=1}^d n_i(A_{ij}(D_x u_a)\partial_j\nu)= \eps^{-M}[\CalG(t,x')+E_b(\overline{D}_xv)]:=\mathbb{G}_3,
\end{split}
\end{align}
where $\lambda>0$ is large enough and 
\begin{align}
\begin{split}
&E(\overline{D}^2_x v):= -\sum_{i=1}^d\partial_i\left(A_{i}(D_x(u_a+v))-A_i(D_xu_a)-\sum_{j=1}^nA_{ij}(D_x u_a)\partial_jv\right)=\\
&\qquad -\sum^d_{i=1}\partial_i\left[\int^1_0 (1-\theta)(d^2A_i)(D_x(u_a+\theta v))(D_xv,D_xv)d\theta\right],\\
&E_b(D_xv)=\sum^d_{i=1} n_i \left[\int^1_0 (1-\theta)(d^2A_i)(D_x(u_a+\theta v))(D_xv,D_xv)d\theta\right].
\end{split}
\end{align}

\begin{rem}

Conversely, one can  check that sufficiently regular solutions $(\omega^\eps,\nu^\eps)$ of the coupled systems \eqref{H2}, \eqref{E} actually satisfy $\omega=\partial_t\nu$, and that $\nu$ then satisfies \eqref{H1}.  This can be shown by differentiating \eqref{E} with respect to $t$, and using the result of that together with \eqref{H2} to show that $\partial_t\nu-\omega$ satisfies an elliptic problem with vanishing interior and boundary data.  
Substituting $\partial_t \nu=\omega$ into \eqref{E}, we then see that $\nu$ satisfies \eqref{H1}. 
This argument, which is not needed in this paper,
%, which we apply for each particular $\eps$,  
was used  in \cite{SN} to prove local existence for \eqref{mainproblem}  for data with no $\eps$-dependence, Proposition \ref{localex} below.  
\end{rem}

Proposition \ref{localex} gives for each fixed $\eps$ a solution $(\omega^\eps,\nu^\eps)$ of \eqref{H1}, \eqref{H2}, and \eqref{E} on a time interval $T_\eps$ that depends on $\eps$.%\footnote{This is explained in Remark  ...}  
    \;Our main task in the error analysis is to prove that solutions of \eqref{H1} exist on a fixed time interval independent of $\eps$. 
This requires estimates that are uniform with respect to $\eps$.  We will show that one can prove such estimates for an appropriate $\eps$-dependent ``energy", $\mathcal{E}^s_\eps(t)$ (Definition \ref{energy}) defined in terms of $\nu$ and $\omega$,  by estimating solutions of the systems \eqref{H1}, \eqref{H2}, and \eqref{E} using 
the linear a priori estimates for hyperbolic and elliptic boundary problems given in section \ref{apri}.

The hyperbolic estimate of Proposition \ref{hypest} exhibits a loss of one-half spatial derivative on the boundary, corresponding to the degeneracy in the Neumann boundary  
condition (or more precisely, corresponding to the failure of the uniform Lopatinski condition in the elliptic region).    By estimating the systems  \eqref{H1}, \eqref{H2}, and \eqref{E} simultaneously, we are able, roughly speaking, to use the elliptic estimate for \eqref{E} to gain back what is lost in the hyperbolic estimates.   This process is carried out in section \ref{energyestimates}, where we obtain estimates uniform with respect to $\eps$ for an appropriate $\eps$-dependent ``energy" $\mathcal{E}^s_\eps(t)$, whose definition (Definition  \ref{energy}) involves both $\nu$ as in \eqref{H1}, \eqref{E} and $\omega$ as in \eqref{H2}.    This estimate of $\mathcal{E}^s_\eps(t)$ is stated in Proposition \ref{mainestimate}.   In the remainder of section \ref{mainR} we show how to use this estimate to prove Theorem \ref{main}, and thereby complete the error analysis.

\section{Norms and basic estimates}

In this section we define the norms and spaces needed to state and prove the main results of part \ref{p2}.  We also prove basic properties of the norms that will be used repeatedly in section \ref{energyestimates}.

\begin{notation}\label{nota1}
\emph{}
\textbf{1. }For any  nonnegative integers $s$, $k$ let  $$E^{s,k}[0,T]=\cap^s_{i=0}C^i([0,T]:H^{s+k-i}(\mathbb{R}^d_+))\text{ and }E^{s,k}_b[0,T]=\cap^s_{i=0}C^i([0,T]:H^{s+k-i}(\mathbb{R}^{d-1})).$$

We sometimes write $E^{s,0}[0,T]=E^{s}[0,T]$ and $E^{s,0}_b[0,T]=E^{s}_b[0,T]$.

\textbf{2. }  For $u\in E^{s,k}[0,T]$, $t\in [0,T]$, and $\eps\in (0,1]$ we define the $(t,\eps)$-dependent norm   $|u(t)|_{s,k,\eps}$ by
\begin{align}
|u(t)|_{s,k,\eps}:= \sup_{|\alpha|\leq s,|\beta|\leq  k}\eps^{|\alpha|+|\beta|}|\partial_{t,x}^\alpha\partial_x^\beta u(t,\cdot)|_{L^2(x)}.
\end{align}
%We will sometimes write this as 
%\begin{align}
%|u(t)|_{E^{M,L}_\eps}:=|(\overline{\eps D})^L(\overline{\eps D_x})^M u(t,\cdot)|_{L^2(x)}.
%\end{align}
For $f=f(t,x')\in E^{s,k}_b[0,T]$  we similarly define boundary norms
\begin{align}
 \langle f(t)\rangle_{s,k,\eps}:=\sup_{|\alpha|\leq s,|\beta|\leq  k}\eps^{|\alpha|+|\beta|}|\partial_{t,x'}^\alpha\partial_{x'}^\beta f(t,\cdot)|_{L^2(x')}.
\end{align}

We will sometimes write  $|u(t)|_{s,0,\eps}=|u(t)|_{s,\eps}$ and do similarly for the boundary norms.  
%We also set  $|u(t)|_{s,k,1}=|u(t)|_{s,k}$,  and do similarly for the boundary norms.   

\textbf{3. }For $u\in E^{s,k}[0,T]$ we set $|u|_{s,k,\eps,T}:=\sup_{t\in[0,T]}|u(t)|_{s,k,\eps}$ and for $f\in E^{s,k}_b[0,T]$ we set 
$\langle f\rangle_{s,k,\eps,T}:=\sup_{t\in[0,T]} \langle f(t)\rangle_{s,k,\eps}$.

\textbf{4. }If $T'<0$ one defines $E^{s,k}[T',T]$ and the corresponding norm $|u|_{s,k,\eps,T',T}$ as the obvious analogues of $E^{s,k}[0,T]$ and $|u|_{s,k,\eps,T}$.    When 
$u\in E^{s,k}[T',T]$ and vanishes in $t<0$, we write $|u|_{s,k,\eps,T',T}$ instead as simply $|u|_{s,k,\eps,T}$.  \footnote{The functions $u_a$, $\nu$, and $\omega$ in \eqref{H1}, \eqref{H2}, and \eqref{E} all vanish in $t<0$.}

\textbf{5. }  For a nonnegative integer $k$ and $f=f(t,x)$ we set
\begin{align}
\begin{split}
&D^k f=(\partial_{t,x}^\alpha f, |\alpha|=k)  \text{ and }\overline{D}^k f=(\partial_{t,x}^\alpha f, |\alpha|\leq k) \\
&D_x^k f=(\partial_{x}^\beta f, |\beta|=k)  \text{ and }\overline{D}_x^k f=(\partial_{x}^\beta f, |\beta |\leq k) .
\end{split}
\end{align}
We normally write $D$ or $\overline{D}$ in place of $D^1$ or $\overline{D}^1$.

\textbf{6. }All scalar functions appearing in part \ref{p2}  are \emph{real-valued}.  This applies to components of vectors and entries of matrices.  

\end{notation}

The next proposition is an immediate consequence of the definitions.

\begin{prop}
Let $C$ denote a constant independent of $\eps$ and let $\eps_0>0$.  
If $|\eps^2\overline{D}^2_x \nu|_{s,0,\eps}\leq C$ for $\eps\in (0,\eps_0]$, then
\begin{align}
\begin{split}
&|\eps^2D_x\nu|_{s,1,\eps}\leq C   \text{ and }|\eps^2\nu|_{s,2,\eps}\leq C \text{ for }\eps\in (0,\eps_0].
\end{split}
\end{align}

\end{prop}

\begin{rem}
Below for $\nu$ as in \eqref{H1}, \eqref{H2}, \eqref{E} we will need $\nu$ to satisfy both $|\eps^2\overline{D}^2_x \nu|_{s,0,\eps}\leq C$ and $|\eps\overline{D}_x \nu|_{s,0,\eps}\leq C$
for an appropriate choice of $s$ and range of (small) $\eps$.   Observe that the second of these two properties is \emph{not} a consequence of the first.
\end{rem}

When proving properties of the $|\cdot|_{s,k,\eps}$ norms the following simple observation (used already for the $|\cdot|_{s,0,\eps}$ norm in \cite{Marcou}) is quite useful.

\begin{lem}\label{aa0}
Given $u=u(t,x)$ and $\eps>0$ define the rescaled function $\tilde u(\tilde t,\tilde x)$ by $\tilde u(\frac{t}{\eps},\frac{x}{\eps})=u(t,x)$.
Then
\begin{align}\label{aa1}
|u(t)|_{s,k,\eps}=\eps^{d/2} |\tilde u(t/\eps)|_{s,k,1}.
\end{align}
Here $ |\tilde u(t/\eps)|_{s,k,1}$ means  $|\tilde u(\tilde t)|_{s,k,1}$ evaluated at $\tilde t=t/\eps$.
\end{lem}

\begin{proof}
This follows immediately by change of variables from
$$
\eps^{|\alpha|+|\beta|}\partial_{t,x}^\alpha\partial_{x}^\beta u(t,x)=\partial_{\tilde t,\tilde x}^\alpha\partial_{\tilde x}^\beta \tilde u(\frac{t}{\eps},\frac{x}{\eps}). 
$$

\end{proof}

%\begin{prop}[Product estimate]
%Suppose $a$, $b$, $c$, $d$ are $\geq 0$, $a+b\leq L$, $c+d\leq M$, and $L+M>\frac{d}{2}$.  Then 
%\begin{align}
%|(uv)(t)|_{E^{L-a-b,M-c-d}_\eps}\leq C\eps^{-\frac{d}{2}}|u(t)|_{E^{L-a,M-c}_\eps}|v(t)|_{E^{L-b,M-d}_\eps}
%\end{align}

%\end{prop}

\begin{prop}[Sobolev estimate]\label{sob}
Let $s>\frac{d}{2}$.   Then
\begin{align}
\begin{split}
(a) &|u(t)|_{L^\infty(x)} \lesssim \eps^{-\frac{d}{2}}|u(t)|_{s,\eps}\\
(b) &|\partial^\alpha_{t,x} u(t)|_{L^\infty(x)}\lesssim \eps^{-\frac{d}{2}-|\alpha|}|u(t)|_{s+|\alpha|,\eps}
\end{split}
\end{align}
\end{prop}

\begin{proof}
We have
\begin{align}
|u(t)|_{L^\infty(x)}=|\tilde u(\frac{t}{\eps})|_{L^\infty(\tilde x)}\lesssim |\tilde u(\frac{t}{\eps})|_{H^s(\tilde x)}\leq |\tilde u(\frac{t}{\eps})|_{s,0,1}=\eps^{-\frac{d}{2}}|u(t)|_{s,0,\eps}
\end{align}
by Lemma \ref{aa1}.  The  estimate (b) follows directly from (a).

\end{proof}

The following product estimate, used in the treatment of first order conservation laws in \cite{Marcou}, will be useful to us here.

\begin{prop}[Product estimate]\label{prod}
Suppose $a$ and $b$ are $\geq 0$, $a+b\leq s$,  and $s>\frac{d}{2}$.  Then 
\begin{align}
|(uv)(t)|_{s-a-b,\eps}\leq C\eps^{-\frac{d}{2}}|u(t)|_{s-a,\eps}|v(t)|_{s-b,\eps}
\end{align}
\end{prop}

\begin{proof}
Following \cite{Marcou} we begin with the well-known estimate
\begin{align}
|u(t)v(t)|_{H^{s-a-b}(x)}\leq C|u(t)|_{H^{s-a}(x)} |v(t)|_{H^{s-b}(x)}. 
\end{align}
Thus, for $k\leq s-a-b$ we have
\begin{align}
\begin{split}
&|\partial_t^k (u(t)v(t))|_{H^{s-a-b-k}(x)}\leq C \sup_{k_1+k_2=k}|\partial^{k_1}_t u(t)\partial^{k_2}_t v(t)|_{H^{s-a-k_1-b-k_2}(x)}\leq\\
&\qquad \sup_{k_1+k_2=k}|\partial^{k_1}_t u(t)|_{H^{s-a-k_1}(x)}   |\partial^{k_2}_t v(t)|_{H^{s-b-k_2}(x)}\leq C|u(t)|_{s-a,0,1}|v(t)|_{s-b,0,1}.
\end{split}
\end{align}
This implies $|u(t)v(t)|_{s-a-b,0,1}\leq C|u(t)|_{s-a,0,1}|v(t)|_{s-b,0,1}$, so an application of Lemma \ref{aa0} finishes the proof.

\end{proof}

\begin{prop}[Trace estimate]\label{trace}
For $g=g(t,x',x_d)$ and any nonnegative integer $s$ we have\footnote{Here $\Lambda_{x'}^{1/2}$ is the operator defined by $\mathcal{F}(\Lambda_{x'}^{1/2} h)=\sqrt{1+|\xi'|^2}\mathcal{F}(h)$, where $\mathcal{F}$ denotes the Fourier transform with respect to $x'$.}
\begin{align}
\eps\langle \Lambda_{x'}^{\frac{1}{2}}g(t)\rangle_{s,\eps}\lesssim |g(t)|_{s,1,\eps}.
\end{align}

\end{prop}

\begin{proof}
We have
\begin{align}
\begin{split}
&\eps \langle\Lambda^{\frac{1}{2}}_{x'}g(t)\rangle_{s,\eps}=\sup_{|\alpha_t,\alpha'|\leq  s}\eps^{\alpha_t+|\alpha'|+1}\left| \Lambda^{\frac{1}{2}}_{x'}\partial_t^{\alpha_t}\partial_{x'}^{\alpha'}g(t,x',0)\right|_{L^2(x')}\lesssim \\
&\sup_{|\alpha_t,\alpha'|\leq  s}\eps^{\alpha_t+|\alpha'|+1}\left( |\partial_t^{\alpha_t}\partial_{x'}^{\alpha'}g(t,x',x_d)|_{L^2(x)}+|\partial_t^{\alpha_t}\partial_{x'}^{\alpha'}\partial_xg(t,x',x_d)|_{L^2(x)}\right)\lesssim  |g(t)|_{s,1,\eps},
\end{split}
\end{align}
where  we have used $|\Lambda^{\frac{1}{2}}_{x'}h(x',0)|_{L^2(x')}\lesssim |h(x',x_d)|_{H^1(x)}$ to get the first inequality.

\end{proof}

\section{A priori estimates and local existence.}\label{apri}

In this section we state (with one essential modification in Proposition \ref{hypest}) the  results from \cite{Sh,SN} that will be used in this paper. 

For some $T>0$ consider the following linear hyperbolic problem:
\begin{align}\label{a2}
\begin{split}
&\partial_t^2 u-\sum^d_{i,j=1}\partial_i(A_{ij}(t,x)\partial_j u):=f(t,x) \text{ in } [-\infty,T]\times \mathbb{R}^d_+\\
&\sum^d_{i,j=1}n_i A_{ij}(t,x)\partial_ju:=g(t,x') \text{ on }[-\infty,T]\times \mathbb{R}^{d-1}\\
&u=0\text{ in }t\leq 0.
\end{split}
\end{align}
%where $f$ and $g$ vanish in $t\leq 0$.  

%Below we set $\|u(t,\cdot)\|:=|u(t,x)|_{L^2(x)}$ and  $\langle g(t,\cdot)\rangle :=|g(t,x')|_{L^2(x')}$. 

\begin{defn}

1. For any nonnegative integer $L$ let $\mathcal{B}^L_T$ denote the set of real-valued functions $f(t,x)$ on $[0,T]\times \overline{\mathbb{R}}^d_+$ with $|f|_{C^L([0,T]\times \overline{\mathbb{R}}^d_+)}<\infty$. 

% For any  nonnegative integer $L$ let  $E^L[0,T]=\cap^L_{i=0}C^i([0,T]:H^{L-i}(\mathbb{R}^d_+))$.

2.  Set $\|u(t,\cdot)\|:=|u(t,x)|_{L^2(x)}$ and  $\langle g(t,\cdot)\rangle :=|g(t,x')|_{L^2(x')}$. 
\end{defn}

\begin{ass}\label{assh}

a)  The entries of $A_{ij}$ belong to $\mathcal{B}^2_T$.

b)  For any $(t,x)\in [0,T]\times \overline{\mathbb{R}}^d_+$ we have $A_{ij}^t=A_{ji}$  ($M^t$ denotes the transpose of the matrix $M$).

c)  There exist positive constants $\delta_1$ and $\delta_2$ such that for $u\in H^2(\mathbb{R}^d_+)$ and $t\in [0,T]$
\begin{align}
\sum^d_{i,j=1}(A_{ij}(t,\cdot)\partial_ju,\partial_i u)_{L^2(x)}\geq \delta_1\|D_xu\|^2-\delta_2\|u\|^2.
\end{align}

\end{ass}

The next proposition gives an improved version of an estimate in \cite{Sh}.

\begin{prop}[Hyperbolic estimate]\label{hypest}
 
 Let $T>0$ and consider the problem \eqref{a2} where the coefficients $A_{ij}$ satisfy Assumption \ref{assh} and $u\in E^2[0,T]$ with $u(0,x)=\partial_tu(0,x)=0$.
 There exists $K_1=K_1(T,\delta_1,\delta_2, max_{ij}|A_{ij}|_{W^{1,\infty}})$ such that $u$ satisfies
\begin{align}\label{a3}
\|\overline{D} u(t,\cdot)\|^2\leq K_1\int^t_0 (\|f(s,\cdot)\|^2+\langle \Lambda^{\frac{1}{2}}_{x'}g(s,\cdot)\rangle^2)ds\text{ for  }t\in [0,T],
\end{align}
%where the  constant $K_1(T)$  depends only on the $W^{1,\infty}$ norms of the coefficients $A_{ij}$.
\end{prop}

\begin{proof}

The proof is identical to the proof of Theorem 6.8 in \cite{Sh}, except for one essential change.\footnote{This proposition is actually a special case of Theorem 6.8 of \cite{Sh}. More general hyperbolic systems are considered there; a similar remark applies to Proposition \ref{ellest}.}    The proof there uses (for example on p. 181) a commutator estimate (Theorem Ap. 5)
\begin{align}\label{aa3}
|e^{-\gamma t}[a(x),\Phi]u|_{L^2}\leq C |a|_{C^{1,\mu}}|e^{-\gamma t}u|_{L^2}, 
\end{align}
where $\Phi$ is a classical pseudodifferential operator of order one (depending on a parameter $\gamma\geq 1$) and $\mu$ is any small positive number.   In place of \eqref{aa3} one can use an improved estimate where the H\"older norm
$|a|_{C^{1,\mu}}$ is replaced by $|a|_{W^{1,\infty}}$, the Lipschitz norm of $a$ (\cite{CM}, Ta).    Thus, the constant $K_1$ here depends on $max_{ij}|A_{ij}|_{W^{1,\infty}})$ rather than 
$max_{ij}|A_{ij}|_{C^{1,\mu}})$ as in \cite{Sh}.

% estimate is proved in [Sh], except that the constant there depends on  a slightly higher norm of the coefficients ($C^{1+\mu}$, $\mu>0$).   Using commutator results of [CoMe], one can improve that result to obtain a constant that depends only on the $W^{1,\infty}$ norm of the coefficients.   

\end{proof}

\begin{rem}
Proposition \ref{hypest} will be applied in the error analysis of section \ref{energyestimates} to the problems \eqref{H1} and \eqref{H2}. 
The fact that $K_1$ depends just on the Lipschitz norm of the $A_{ij}$ is  crucial for the application we give here, since the coefficients $A_{ij}(D_xu_a+D_xv)$ in \eqref{H1}, \eqref{H2},  although quite regular for each fixed $\eps$, are no better than $W^{1,\infty}$ \emph{uniformly} with respect to small $\eps$.

   %We apply this Proposition to the problem \eqref{H2} for $\omega$ to control  $t$ derivatives and to the  problem \eqref{H1} for $\nu$ to control the blow up as $\eps\to 0$ of $D_x\nu$ (``gain epsilons").  

\end{rem}

Next we give an estimate for  the linear elliptic problem:
\begin{align}\label{a4}
\begin{split}
&-\sum^d_{i,j=1}\partial_i(A_{ij}(t,x)\partial_j u)+\lambda u:=f(t,x) \text{ in }[0,T]\times \mathbb{R}^d_+\\
&\sum^d_{i,j=1}n_i A_{ij}(t,x)\partial_ju:=g(t,x') \text{ on }[0,T]\times \mathbb{R}^{d-1},
\end{split}
\end{align}
where $t$ is treated as a parameter.  

\begin{prop}[Elliptic estimate, \cite{SN}]\label{ellest}

Consider the problem \eqref{a4} where the coefficients $A_{ij}$ satisfy Assumption \ref{assh} and $u\in E^2[0,T]$.    There exist positive constants $\lambda_0(T,\delta_1,\delta_2, max_{ij}|A_{ij}|_{W^{1,\infty}})$ and $K_2(T,\delta_1,\delta_2, max_{ij}|A_{ij}|_{W^{1,\infty}})$ such that for $\lambda\geq \lambda_0$ the function $u$ satisfies
\begin{align}\label{a5}
\|\overline{D}^2_xu(t,\cdot)\|\leq K_2\left(\|f(t,\cdot)\|+\langle \Lambda^{\frac{1}{2}}_{x'} g(t,\cdot)\rangle\right) \text{ for  }t\in [0,T].
\end{align}
 %The constants $K_2$ and $\lambda$ that appear here  depend only on the $W^{1,\infty}$ norms of the coefficients $A_{ij}$ (as well as constants that appear in the coercivity assumption on $A_{ij}$).\footnote{This estimate is proved in [Sh-N].}   
\end{prop}

This estimate is part of Theorem 4.4 of \cite{SN}.  
We will apply it to the problem \eqref{E} for $\nu$ to gain spatial derivatives.

     The following is our main structural assumption on the nonlinear problem \ref{mainproblem} (and \eqref{a6}).

\begin{ass}\label{a5a}
a)  Let $R>0$ and let $U\in \mathbb{R}^{dN}$; here $dN$ is the number of scalar arguments of the $\mathbb{R}^N$-valued functions  $A_i$ in \eqref{mainproblem}. The  $A_i$ are $C^\infty$ functions of $U$ and satisfy $A_i(0)=0$.

 b)  The $N\times N $ matrices $A_{ij}:=\partial A_i/\partial{U_j}$ satisfy $A_{ij}^t=A_{ji}$  ($M^t$ denotes the transpose of the matrix $M$).

c)  There exist positive constants $\delta_1$ and $\delta_2$ such that for $u\in H^2(\mathbb{R}^d_+)$ and $|U|\leq 4R$
\begin{align}
\sum^d_{i,j=1}(A_{ij}(U)\partial_ju,\partial_i u)_{L^2(x)}\geq \delta_1\|u\|_{H^1}^2-\delta_2\|u\|^2.
\end{align}

\end{ass}

\begin{rem}\label{NE}
The equations of  nonlinear elasticity for three-dimensional, isotropic, hyperelastic materials  are shown in \cite{SN}, pages 8-10, to satisfy Assumption \ref{a5a} for small enough $R>0$.    The Saint Venant-Kirchhoff model \eqref{j0} for which we construct approximate solutions $u^\eps_a$ in part \ref{p3}  belongs to this class of models.  
\end{rem}

Next we state a local existence theorem for the nonlinear initial boundary value problem
\begin{align}\label{a6}
\begin{split}
&P(u)=\partial_t^2 u-\sum^d_{i=1}\partial_i(A_i(D_xu))=f(t,x) \text{ in }x_d\geq 0\\
&Q(u)=\sum_{i=1}^d n_i A_i(D_xu)=g(t,x')\text{ on }x_d=0\\
&u(0,x)=v_0(x), \;u_t(0,x)=v_1(x). 
\end{split}
\end{align}

\begin{prop}[Local existence and uniqueness: \cite{SN}, Theorem 2.4]\label{localex}
Suppose that Assumption \ref{a5a} holds, let $L$ be an integer $\geq [d/2]+8$ (here $[r]$ is the greatest integer $\leq r$), and let $T_0$ and  $\mathbb{B}$  be positive constants.   Assume that the data in \eqref{a6} satisfy \footnote{Here $g\in E^{L-2,1/2}_b[0,T_0] \Leftrightarrow \Lambda^{1/2}_{x'} g\in E^{L-2}_b[0,T_0]$.}
\begin{align}\label{a7}
\begin{split}
&(a)\;v_0\in H^L(\mathbb{R}^d_+), v_1\in H^{L-1}(\mathbb{R}^d_+), f\in C^{L-1}([0,T_0];L^2(\mathbb{R}^d_+))\cap E^{L-2}[0,T_0], \\
&g\in C^{L-1}([0,T_0]:H^{1/2}(\mathbb{R}^d_+))\cap E^{L-2,1/2}_b[0,T_0],
\end{split}
\end{align}
\quad\qquad\qquad (b) $v_0$, $v_1$, $f$, and $g$ satisfy corner compatibility conditions of order $L-2$ and  $|v_1|_{L^\infty}+|\overline{D}_xv_0|_{L^\infty}\leq R$,\footnote{The norms on $f$ and $g$ are defined in Notation \ref{nota1}.}
\begin{align}\notag
(c) \;|v_0|_{H^{[d/2]+8}}+|v_1|_{H^{[d/2]+7}}+|f|_{[d/2]+6,0,1,T_0}+\langle\Lambda^{1/2}_{x'} g\rangle_{[d/2]+6,0,1,T_0}\leq \mathbb{B}.\qquad\quad
\end{align}
Then there exists $T_1=T_1(\mathbb{B})\in(0,T_0)$  such that \eqref{a6} has a unique solution $u\in E^L[0,T_1]$ satisfying $|u|_{W^{1,\infty}}\leq 3R$.

\end{prop}

\begin{prop}[Continuation]\label{continuation}
Suppose that Assumption \ref{a5a} holds, let $L$ be an integer $\geq [d/2]+8$,  and let $T_0$ and $\mathbb{B}$ be positive constants.  Suppose that $f$ and $g$ have the regularity in \eqref{a7}(a) and that for some ${T}\in (0,{T}_0]$ we are given a solution of \eqref{a6} such that $u\in E^L[0,T]$ with 
$$|u|_{L,0,1,T}\leq \mathbb{B}/2 \text{ and }|u|_{W^{1,\infty}([0,T]\times \mathbb{R}^d_+)}\leq R.$$
Then there is a time step $\Delta T$ depending on $\mathbb{B}$, $f$, and $g$, but not on $T$, such that $u$ extends to a solution on $[0,\min(T+\Delta T,T_0)]$ and satisfies $|u|_{L,0,1,\min(T+\Delta T,T_0)}\leq \mathbb{B}$. 

\end{prop}

\begin{rem}     

1.   The compatibility conditions referred to in \eqref{a7}(b) are rather complicated to state and their precise form is not needed in this paper; they are stated on p. 6 of \cite{SN}.   

2.  Proposition \ref{continuation} is not stated explicitly in \cite{SN}, but it is a Corollary of Proposition \ref{localex} {and} its proof.   

3.  Theorem 2.4 of \cite{SN} deals with a wider class of systems than \eqref{a6}.

\end{rem}

\section{Main result}\label{mainR}

In this section we describe the main result relating approximate and exact solutions, Theorem \ref{main},  and give an outline of the main argument, deferring some of the proofs until section \ref{energyestimates}. 

Throughout the remainder of Part \ref{p2}  we make the following assumption about the objects appearing in the problem \eqref{apsys} satisfied by the approximate solution $u^\eps_a$.

\begin{ass}\label{ass1}
Suppose $M>\frac{d}{2}+2$, 
 let $T$ be fixed once and for all as the time of existence of the approximate solution $u^\eps_a$ for $\eps\in (0,1]$, and set  $\Omega:=(-\infty,T]\times \overline{\mathbb{R}}^d_+$. 
We suppose $s\geq [\frac{d}{2}]+6$ (here $[r]$ is the greatest integer $\leq r$), that $u^\eps_a\in E^{s+2}(-\infty,T]$ and vanishes in $t<0$,
 and that for some positive constants $A_1$, $A_2$ we have
\begin{align}\label{ass1a}
\begin{split}
&|\eps^\alpha\partial_{t,x}^\alpha u^\eps_a|_{L^\infty(\Omega)}\leq \eps^2 A_1\text{ for all }|\alpha|\leq s+2\\
&|R^\eps(t)|_{s+1,\eps}+\langle r^\eps(t)\rangle_{s+2,\eps}\leq A_2 \text{ for all }t\in [0,T],
\end{split}
\end{align}
where $R^\eps$, $r^\eps$ are as in \eqref{apsys}.

\end{ass}

Next we define an $\eps$-dependent ``energy" $\mathcal{E}^s_\eps(t)$ for which an estimate uniform with respect to $\eps$ is stated in Proposition \ref{mainestimate}.  
%The proof of that estimate is  carried out in section \ref{energyestimates}.

\begin{defn} [The energy $\mathcal{E}^s_\eps(t)$]\label{energy}
 For  $\nu^\eps$, $\omega^\eps$  as in \eqref{H1}, \eqref{H2}, \eqref{E} and $t\in \mathbb{R}$, let $\mathcal{E}^s_\eps(t)=|\eps^2 \overline{D}^2_x \nu(t)|_{s,\eps}+|\eps \overline{D} \nu(t)|_{s,\eps}+|\eps^2 \overline{D} \omega(t) |_{s,\eps}$.  Here and below $s\geq [\frac{d}{2}]+6$ as in Assumption \ref{ass1}. 

\end{defn}

\begin{rem}\label{howtoapply}
 Propositions \ref{localex} or \ref{continuation} are stated for the system \eqref{a6} and thus apply to \eqref{mainproblem}.  We  want to apply these propositions to the systems \eqref{H1}, \eqref{H2}, \eqref{E} for \emph{fixed} $\eps$.    To do this, we use the simple observation that if $u$ is a solution of \eqref{mainproblem}, then the functions $\nu$, $\omega$ defined by
 \begin{align}\label{a8ac}
 \nu=\eps^{-M}(u-u_a),\;\; \omega=\partial_t\nu
 \end{align}
 satisfy \eqref{H1}, \eqref{H2}, and \eqref{E}.   Moreover, since $u_a\in E^{s+2}(-\infty,T]$ it is clear that 
 \begin{align}\label{a8ab}
 u\in E^{s+2}[0,T_\eps]\Rightarrow (\overline{D}^2_x\nu, \overline{D}\nu, \overline{D}\omega)\in E^s[0,T_\eps],
 \end{align}
and it follows from directly from the definition of the spaces $E^s[0,T_\eps]$ that the converse of \eqref{a8ab} is also true when \eqref{a8ac} holds.
\end{rem}

The following lemma is needed among other things for our applications of Propositions \ref{localex} and \ref{continuation} as well as the linear estimates of Propositions \ref{hypest} and \ref{ellest}.

\begin{lem}\label{a8a}

Suppose $T_\eps\in (0,T]$ and that for some $C_1>0$  we have $\mathcal{E}^s_\eps(t)\leq C_1$ for all $t\in [0,T_\eps]$.    Then \footnote{The  proof of \eqref{a8}(b) is contained in step \textbf{2} of the proof of Lemma \ref{icomh1}.   Part (a) is immediate from \eqref{ass1a}.} 
\begin{align}\label{a8}
\begin{split}
&(a)\;|u_a|_{W^{1,\infty}(\Omega)}\leq \eps A_1 \\
%&|\partial^{\gamma_k}D_x^m v(t)|_{s-|\gamma_k|,\eps}\leq C_1\eps^{M-|\gamma_k|-m}\\
%&|\partial^{\zeta}D_x^m \nu(t)|_{s-|\zeta|,\eps}\leq \mathcal{E}^s_\eps(t)\eps^{-|\zeta|-m}.
&(b)\;|v|_{W^{1,\infty}([0,T_\eps]\times \mathbb{R}^+_d)}\lesssim C_1\eps^{M-\frac{d}{2}-1}.
\end{split}
\end{align}
\end{lem}

\begin{rem}\label{a8aa}
Observe that  if $\mathcal{E}^s_\eps(t)\leq C_1$ for all $t\in [0,T_\eps]$ and $\eps\in (0,\eps_1]$, we can use Lemma \ref{a8a} to insure, by reducing $\eps_1$ if necessary, that for $\eps\in (0,\eps_1]$:
$$
 |u_a|_{W^{1,\infty}(\Omega)}<R/2 \text{ and }|v|_{W^{1,\infty}([0,T_\eps]\times \mathbb{R}^+_d)}<R/2,
$$
where $R$ is the constant in Assumption \ref{a5a}(a).
\end{rem}

The following theorem is the main result of this part.

\begin{thm}\label{main}

Consider the nonlinear hyperbolic Neumann problem \eqref{mainproblem} under Assumption \ref{a5a} and suppose that $u^\eps_a$ is an approximate solution satisfying Assumption \ref{ass1}.  
% let $s\geq [\frac{d}{2}]+3$,  and  let $T$ be the time of existence of the approximate solution.   

(a) There exist constants $\eps_3>0$ and $C_2>0$ such that for $\eps\in [0,\eps_3]$ the coupled systems \eqref{H1}, \eqref{H2}, \eqref{E} have a unique solution on the time interval $[0,T]$ which satisfies
\begin{align}\label{bound}
\mathcal{E}^s_\eps(t)\leq C_2 \text{ for all }t\in[0,T].
\end{align}

(b)   For $\eps\in [0,\eps_3]$ the problem \eqref{mainproblem} has a unique exact solution $u^\eps\in E^{s+2}(-\infty,T]$  given by 
$$
u^\eps=u^\eps_a+v^\eps=u^\eps_a+\eps^M\nu^\eps,
$$
where 
\begin{align}
|\eps^2 \overline{D}^2_x v^\eps(t)|_{s,\eps}+|\eps \overline{D} v^\eps(t)|_{s,\eps}+|\eps^2 \overline{D} \partial_t v^\eps(t) |_{s,\eps}\leq \eps^M C_2 \text{ for }t\in [0,T]\text{ and }\eps\in (0,\eps_3].
\end{align}
In particular this implies
$|v^\eps|_{W^{1,\infty}(\Omega)}\leq C_2\eps^{M-\frac{d}{2}-1}$.
 
\end{thm}

The proof of Theorem \ref{main} is based on the following \emph{a priori} estimate for the coupled systems, whose proof is carried out in section \ref{energyestimates}.

\begin{prop}\label{mainestimate}
%Let $s>...$ and let $T$ (which is independent of $\eps$)   be the time of existence of the approximate solution $u^\eps_a$.  
%Let $M$, $s$, and $T$ be as in Theorem \ref{main}. 
Suppose   $C_1$ and $\eps_1$ are positive constants and that for  $T_\eps\in (0,T]$ and $\eps\in (0,\eps_1]$, we are given a solution $(\nu^\eps,\omega^\eps)$ of 
the three coupled systems on $[0,T_\eps]$ which satisfies  $\mathcal{E}^s_\eps(t)\leq C_1$ for all $t\in [0,T_\eps]$.   Then there exist positive constants $B_1=B_1(T,A_1,K_1,K_2,\lambda)$\footnote{The constants $K_1$, $K_2$, and $\lambda$ appear  in the linear estimates \eqref{a3}, \eqref{a5}.}, $B_2=B_2(A_1,A_2,K_2)$, and $\eps_2=\eps_2(C_1,A_1,K_1,K_2,\lambda)$ such that for $\eps\in [0,\eps_2]$  and all $t \in [0,T_\eps]$ 
\begin{align}\label{apriori}
[\mathcal{E}^s_\eps(t)]^2\leq B_1\int ^t_0 \left([\mathcal{E}^s_\eps(\sigma)]^2+\eps^2|R^\eps(\sigma)|^2_{s+1,\eps}+\langle r^\eps(\sigma)\rangle ^2_{s+2,\eps}\right)d\sigma+\eps^2B_2.
\end{align}

\end{prop}

This proposition will allow us to apply a continuous induction argument to prove Theorem \ref{main}.

\begin{prop}\label{HimpliesC}
%Let $s$ and  $T$ be as in Proposition \ref{mainestimate}.
%For the continuous induction we want to show that 
There exist constants $C_2>0$ and $\eps_3>0$   such that for $\eps\in (0,\eps_3]$ and $T_\eps\in (0,T]$, if $(\nu^\eps,\omega^\eps)$ is a solution of 
the three coupled nonlinear systems \eqref{H1}, \eqref{H2}, \eqref{E} which satisfies $\mathcal{E}^s_\eps(t)\leq C_2$ for all $t\in [0,T_\eps]$, then in fact $\mathcal{E}^s_\eps(t)\leq C_2/2$ for all $t\in [0,T_\eps]$.

\end{prop}

Assuming Proposition \ref{mainestimate} we prove Proposition \ref{HimpliesC}.  In order to apply Proposition \ref{mainestimate} we note that for  any positive $C_1$ and $\eps_1>0$ small enough, for each $\eps\in (0,\eps_1]$ the continuation result, Proposition \ref{continuation}, applies to give some $T_\eps>0$ and a solution $(\nu^\eps,\omega^\eps)$ of the coupled systems on $[0,T_\eps]$  which satisfies $\mathcal{E}^s_\eps(t)\leq C_1$ for all $t\in [0,T_\eps]$.\footnote{Since $\nu$ and $\omega$ vanish in $t<0$, we have $\mathcal{E}^s_\eps(t)=0<\frac{C_1}{2}$ for $t<0$.   We also use Remark \ref{howtoapply} here.}

\begin{proof}[Proof of Proposition \ref{HimpliesC}]

We must choose  $C_2$ and $\eps_3$ with the desired properties.  We choose $C_2^2\geq 4(A_2^2+1)e^{B_1T}/B_1$.
%where $B$ is the constant in \eqref{apriori} and $A$ is such that 
%\begin{align}
%|R^\eps(\sigma)|_{E^{\cdot,\cdot}_\eps}+|r^\eps(\sigma)|_{E^{\cdot,\cdot}_{b,\eps}}\leq A\text{ for all }\sigma\in [0,T].
%\end{align}
Taking this $C_2$ as the choice of ``$C_1$" in Proposition \ref{mainestimate}, we let $\eps_3$ be a  corresponding choice of ``$\eps_2$".\footnote{Here we see   why it is essential that the constant $B_1$ in Proposition \ref{mainestimate} not depend on $C_1$.}   If $\eps\in (0,\eps_3]$ and $(\nu^\eps,\omega^\eps)$ satisfies 
$\mathcal{E}^s_\eps(t)\leq C_2$ for all $t\in [0,T_\eps]$, then \eqref{apriori} and Gronwall's inequality imply
\begin{align}
[\mathcal{E}^s_\eps(t)]^2\leq \int^t_0 e^{B_1(t-\sigma)}\left(\eps^2|R^\eps(\sigma)|^2_{s+1,\eps}+\langle r^\eps(\sigma)\rangle^2_{s+2,\eps}\right)d\sigma +\eps^2B_2\frac{e^{B_1t}}{B_1}\text{ for all }t\in [0,T_\eps]. 
\end{align}
Reducing $\eps_3$ if necessary so that $\eps^2B_2\leq 1$ for $\eps\in (0,\eps_3]$, we obtain  for all $t\in [0,T_\eps]$ and $0<\eps\leq \eps_3(C_2,A_1,A_2)$:
\begin{align}
[\mathcal{E}^s_\eps(t)]^2\leq A_2^2\int^t_0 e^{B_1(t-\sigma)} d\sigma +\eps^2B_2\frac{e^{B_1t}}{B_1}\leq (A_2^2+1)e^{B_1T}/B_1\leq C_2^2/4.
\end{align}

\end{proof}

Next we show that Proposition \ref{HimpliesC} implies Theorem \ref{main}.

\begin{proof}[Proof of Theorem \ref{main}]
\textbf{1. }We choose $\eps_3$ and $C_2$ to be the same as in Proposition \ref{HimpliesC}.   For each $\eps\in (0,\eps_3]$ we will use a continuous induction argument to show that the solution of the coupled systems exists and satisfies \eqref{bound} on all of $[0,T]$. 
%Although this kind of argument is standard, it is worth writing out carefully.

\textbf{2. }Suppose that for a given $\eps\in (0,\eps_3]$ and some $T_1<T$, we have $\mathcal{E}^s_\eps(t)\leq C_2/2$ for all $t\in[0,T_1]$.   Then the continuation result,  Proposition \ref{continuation}, implies that there exists a time step $\Delta T^\eps>0$ (which depends on $\eps$ and $C_2$ but not on $T_1$) such that $(\nu^\eps,\omega^\eps)$ extends to the time interval $[0,T_1+\Delta T^\eps]$
and satisfies $\mathcal{E}^s_\eps(t)\leq C_2$ for all $t\in[0,T_1+\Delta T^\eps]$.

\textbf{3. }Fix $\eps\in[0,\eps_3]$ and let $T^*_\eps=\sup\{T'\in[0,T]:\mathcal{E}^s_\eps(t)\leq C_2 \text{ for all }t\in [0,T']\}$.   Observe that since $\mathcal{E}^s_\eps(t)=0$ for $t\leq 0$, step \textbf{2} implies that $T^*_\eps\geq \Delta T^\eps$. We  now prove by contradiction that $T^*_\eps=T$, so  suppose $T^*_\eps<T$.   Then $\mathcal{E}^s_\eps(t)\leq C_2 \text{ for all }t\in [0,T^*_\eps-\frac{\Delta T^\eps}{2}]$, and hence Proposition \ref{HimpliesC} implies $\mathcal{E}^s_\eps(t)\leq C_2/2 \text{ for all }t\in [0,T^*_\eps-\frac{\Delta T^\eps}{2}]$.  By step \textbf{2} we have $\mathcal{E}^s_\eps(t)\leq C_2 \text{ for all }t\in [0,T^*_\eps+\frac{\Delta T^\eps}{2}]$.   Contradiction.   This proves part (a).

\textbf{4. }Using Remark \ref{howtoapply} and Lemma \ref{a8a}, we see that part (b) follows from part (a).  

\end{proof}

Thus it remains only to prove Proposition \ref{mainestimate}.

\section{Uniform estimates for the  coupled nonlinear systems}\label{energyestimates}

This section is devoted to proving Proposition \ref{mainestimate}.    Recall that 
$$\mathcal{E}^s_\eps(t)=|\eps^2 \overline{D}^2_x \nu(t)|_{s,\eps}+|\eps \overline{D} \nu(t)|_{s,\eps}+|\eps^2 \overline{D} \omega(t) |_{s,\eps}.$$
Under the assumptions of that Proposition, the strategy will be to apply the elliptic estimate \eqref{a5} to the problem \eqref{E} to estimate the first term in $\mathcal{E}^s_\eps(t)$, to apply the hyperbolic estimate \eqref{a3} to the problem \eqref{H1} to estimate the second term, and to apply the same hyperbolic estimate to the problem \eqref{H2} to estimate the third term.   We will use Remark \ref{a8aa} to insure that  these systems satisfy the hypotheses of Propositions \ref{hypest} and \ref{ellest}. 

Throughout this section the constants $\eps_1$, $C_1$, and $T_\eps$ are as given in the statement of Proposition \ref{mainestimate}.

\subsection{Tangential derivative estimates} 

The first step is to estimate tangential $\partial^\alpha:=\partial_{t,x'}^\alpha$ derivatives.  We define
\begin{align}\label{b1}
\begin{split}
&|u(t)|_{E^s_{\eps,tan}}=\sup_{|\alpha|\leq s}\eps^{|\alpha|}|\partial_{t,x'}^\alpha u(t,\cdot)|_{L^2(x)}\\
&\mathcal{E}^s_{\eps,tan}(t)=|\eps^2 \overline{D}^2_x \nu(t)|_{E^s_{\eps,tan}}+|\eps \overline{D} \nu(t)|_{E^s_{\eps,tan}}+|\eps^2 \overline{D} \omega(t) |_{E^s_{\eps,tan}}.
\end{split}
\end{align}
and proceed to show for $\eps$, $B_1$, $B_2$ as described in Proposition \ref{mainestimate}:
\begin{align}\label{b2}
[\mathcal{E}^s_{\eps,tan}(t)]^2\leq B_1\int ^t_0 \left([\mathcal{E}^s_\eps(\sigma)]^2+\eps^2|R^\eps(\sigma)|^2_{s+1,\eps}+\langle r^\eps(\sigma)\rangle ^2_{s+2,\eps}\right)d\sigma+\eps^2(B_1[\mathcal{E}^s_\eps(t)]^2+B_2).
\end{align}

We begin by estimating $|\eps \overline{D} \nu(t)|_{E^s_{\eps,tan}}$ by applying the estimate \eqref{a3} to $\eps^\alpha\partial^\alpha (\eps \eqref{H1})$ for $|\alpha|\leq s$.\footnote{Here 
``$\eps^\alpha\partial^\alpha (\eps \eqref{H1})$" denotes the result of multiplying the problem \eqref{H1} by $\eps$ and then applying the operator $\eps^\alpha\partial^\alpha$.}

\begin{lem}[Interior commutator for \eqref{H1}]\label{icomh1}
 Let $\eps_1$ and $C_1$ be as in Proposition \ref{mainestimate}. 
 There exist positive constants   $\eps_2(C_1)\leq \eps_1$  and $B_1=B_1(A_1)$ such that for $\eps\in [0,\eps_2]$  and all $t \in [0,T_\eps]$ 
\begin{align}\label{b3}
|\partial_i[\eps^\alpha\partial^\alpha,A_{ij}(D_x(u_a+v))]\partial_j(\eps\nu)|_{L^2(x)}\leq B_1 \mathcal{E}^s_\eps(t) \text{ for all }|\alpha|\leq s. 
\end{align}
Here (and below) $i,j\in [1,\dots,d]$. 

\end{lem}

\begin{proof}
\textbf{1. }We note first that each component of $[\eps^\alpha\partial^\alpha,A_{ij}(D_x(u_a+v))]\partial_j(\eps\nu)$ is a finite sum of terms of the form\footnote{In \eqref{b4} as well as in similar expressions below,  we suppress the indices that label components of $D_xu_a$, $D_xv$, or $\nu$.    Note also that here the $\beta_j$, $\gamma_k$ are \emph{multi}-indices.}
\begin{align}\label{b4}
\eps^{1+|\alpha|}H(D_xu_a,D_xv)(\partial^{\beta_1}D_xu_a)\dots(\partial^{\beta_q}D_xu_a)(\partial^{\gamma_1}D_x v)\dots(\partial^{\gamma_p}D_x v)\partial^\zeta\partial_j\nu,
\end{align}
where $H$ is a $C^\infty$ function of its arguments and 
\begin{align}\label{b5}
|\beta_k|\leq |\alpha|,\; |\gamma_\ell|\leq |\alpha|,\; |\zeta|\leq |\alpha|-1, \text{ and }|\beta|+|\gamma|+|\zeta|=|\alpha|\leq s.
\end{align}
Either $p$ or $q$ can be zero, but not both.\footnote{When, for example, $q=0$, this means that none of the terms $\partial^{\beta_k}D_xu_a$ appears in \eqref{b4}.} 
When the $\partial_i$ derivative in \eqref{b3} is taken, four cases arise depending on whether $\partial_i$ hits $H$, one of the $\partial^{\beta_k}D_xu_a$, one of the $\partial^{\gamma_k}D_x v$, or $\partial^\zeta\partial_j\nu$.

\textbf{2. }In treating these cases we will use the following estimates, where $m=1$ or $2$:
\begin{align}\label{b6}
\begin{split}
&|\partial^{\beta_k}D^m_x u_a(t)|_{L^\infty(x)}\leq A_1\eps^{2-|\beta_k|-m}\\
&|\partial^{\gamma_k}D_x^m v(t)|_{s-|\gamma_k|,\eps}\leq C_1\eps^{M-|\gamma_k|-m}\\
&|\partial^{\zeta}D_x^m \nu(t)|_{s-|\zeta|,\eps}\leq \mathcal{E}^s_\eps(t)\eps^{-|\zeta|-m}.
\end{split}
\end{align}
In this section we take, e.g.,  $\partial^{\beta_k}:=\partial_{t,x'}^{\beta_k}$, but the above estimates also hold for $\partial^{\beta_k}:=\partial_{t,x}^{\beta_k}$. 
These estimates follow directly from assumption \ref{ass1}, the assumption $\mathcal{E}^s_\eps(t)\leq C_1$, and the definition of $\mathcal{E}^s_\eps(t)$. 
For example, since $\mathcal{E}^s_\eps(t)\leq C_1$, we have 
\begin{align}\label{b6b}
\eps|\overline{D} v|_{s,\eps}\leq C_1\eps^M\Rightarrow |v,Dv|_{s,\eps}\leq C_1\eps^{M-1}\Rightarrow |\partial^{\gamma_k}(v,Dv)|_{s-|\gamma_k|,\eps}\leq C_1\eps^{M-|\gamma_k|-1}.
\end{align}
     
     We note also that the  Sobolev estimate of Proposition \ref{sob} implies
     \begin{align}\label{b6aa}
\begin{split}
&|(v,Dv)(t)|_{L^\infty(x)}\lesssim \eps^{-\frac{d}{2}}|(v,Dv)(t)|_{s,\eps}\lesssim C_1\eps^{M-\frac{d}{2}-1}\\
&|D(v,Dv)(t)|_{L^\infty(x)}\lesssim \eps^{-\frac{d}{2}}|D(v,Dv)(t)|_{s-1,\eps}\lesssim C_1\eps^{M-\frac{d}{2}-2}.
\end{split}
\end{align}
In the second estimate we have used \eqref{b6b} with $|\gamma_k|=1$.\footnote{We considered $(v,Dv)$ instead of $D_xv$ here in order to include a proof of Lemma \ref{a8a}(b). }

\begin{rem}\label{imprem}
 Since $M>\frac{d}{2}+2$ we see that the Lipschitz norm of $D_x v^\eps$ is  bounded ($\lesssim C_1$) on $[0,T_\eps]\times\overline{\mathbb{R}}^d_+$ when $\mathcal{E}^s_\eps(t)\leq C_1$ on $[0,T_\eps]$.    Moreover, we can (and do) choose $\eps_2$ so that 
\begin{align}\label{b6aaa}
|D_xv|_{W^{1,\infty}([0,T_\eps]\times \mathbb{R}^d_+)}\leq 1\text{ for }\eps\in (0,\eps_2].
\end{align}
A bound like \eqref{b6aaa} independent of $C_1$ is needed to carry out the argument of Proposition \ref{HimpliesC}.   This is because the constants $K_1$, $K_2$, and $\lambda$ in the linear estimates \eqref{a3} and \eqref{a5} depend on $|D_xv|_{W^{1,\infty}}$, while $B_1$ in Proposition \ref{mainestimate} in turn depends on $K_1$, $K_2$, and $\lambda$.  For the argument of Proposition \ref{HimpliesC} to work, 
the constant $B_1$ must not depend on $C_1$. 

%2. In view of \eqref{b6} and \eqref{b6aa} we can (and do) choose $\eps_2$ so that $|D_x(u_a+v)|_{L^\infty}\leq R$, where $R>0$ is as in Assumption 
\end{rem}

\textbf{3. }Consider the case where $\partial_i$ hits one of the factors $\partial^{\beta_k}D_xu_a$, say the first, producing a term we denote as $\partial^{\beta_1}D_x^2u_a$.  When $p\geq 1$, setting $|\gamma|:=\sum^p_{k=1}|\gamma_k|$ and using $s-|\gamma|-|\zeta|\geq 0$, we apply \eqref{b6} and the product estimate of Proposition \ref{prod} $p$ times to obtain\footnote{A similar use of the product estimate is made in section 1.10 of \cite{Marcou} in her study of first-order hyperbolic conservation laws. }
\begin{align}\label{b6a}
\begin{split}
&|(\partial^{\gamma_1}D_x v)\dots(\partial^{\gamma_p}D_x v)\partial^\zeta\partial_j\nu|_{L^2(x)}\lesssim \\
&\quad \eps^{-pd/2}|(\partial^{\gamma_1}D_x v)|_{s-|\gamma_1|,\eps}\dots|(\partial^{\gamma_p}D_x v)|_{s-|\gamma_p|,\eps}|\partial^\zeta\partial_j\nu|_{s-|\zeta|,\eps}\lesssim\eps^{-pd/2}C_1^p\eps^{Mp-|\gamma|-p}\mathcal{E}^s_\eps(t)\eps^{-|\zeta|-1}.
\end{split}
\end{align}

Using \eqref{b6} and \eqref{b6aa} we also have
\begin{align}\label{b6bb}
|\eps^{1+|\alpha|}H(D_xu_a,D_xv)(\partial^{\beta_1}D^2_xu_a)(\partial^{\beta_2}D_xu_a)\dots(\partial^{\beta_q}D_xu_a)|_{L^\infty(x)}\lesssim \eps^{1+|\alpha|}A_1^q\eps^{q-|\beta|-1},
\end{align}
so by combining these two estimates we find
\begin{align}\label{b7}
\begin{split}
&|\eps^{1+|\alpha|}H(D_xu_a,D_xv)(\partial^{\beta_1}D^2_xu_a)(\partial^{\beta_2}D_xu_a)\dots(\partial^{\beta_q}D_xu_a)(\partial^{\gamma_1}D_x v)\dots(\partial^{\gamma_p}D_x v)\partial^\zeta\partial_j\nu|_{L^2(x)}\leq \\
&\quad C_1^pA_1^q\mathcal{E}^s_\eps(t)\eps^{p(M-\frac{d}{2}-1)}\eps^{|\alpha|-|\beta|-|\gamma|-|\zeta|}\eps^{q-1}\leq C_1^pA_1^q\mathcal{E}^s_\eps(t)\eps^{\mu}\text{ for }t\in [0,T_\eps].
\end{split}
\end{align}
We see that $\mu>0$ using \eqref{b5},  $q\geq 0$, $p\geq 1$, and $M-\frac{d}{2}-1>1$, so  there exists $\eps_2(C_1)$ such that the right side of  \eqref{b7} is $\leq A^q_1\mathcal{E}^s_\eps(t)$ for $\eps\in(0,\eps_2]$.\footnote{The constant $\eps_2$ will be decreased a finite number of times before we arrive at the final choice of $\eps_2$ that appears in the statement of Proposition \ref{mainestimate}.}

When $p=0$ we must have $q \geq 1$, and instead of \eqref{b7} we obtain
\begin{align}\label{b8}
\begin{split}
&|\eps^{1+|\alpha|}H(D_xu_a,D_xv)(\partial^{\beta_1}D^2_xu_a)(\partial^{\beta_2}D_xu_a)\dots(\partial^{\beta_q}D_xu_a)\partial^\zeta\partial_j\nu|_{L^2(x)}\leq \\
&\quad A_1^q\mathcal{E}^s_\eps(t)\eps^{|\alpha|-|\beta|-|\zeta|}\eps^{q-1}\leq A_1^q\mathcal{E}^s_\eps(t).
\end{split}
\end{align}

\textbf{4. }Next consider the case when $\partial_i$ hits $\partial^\zeta\partial_j\nu$ producing $\partial^\zeta D_x^2\nu$.  When $p\geq 1$ instead of \eqref{b7} we find by arguing as above
\begin{align}\label{b9}
\begin{split}
&|\eps^{1+|\alpha|}H(D_xu_a,D_xv)(\partial^{\beta_1}D_xu_a)(\partial^{\beta_2}D_xu_a)\dots(\partial^{\beta_q}D_xu_a)(\partial^{\gamma_1}D_x v)\dots(\partial^{\gamma_p}D_x v)\partial^\zeta D_x^2\nu|_{L^2(x)}\leq \\
&\quad C_1^pA_1^q\mathcal{E}^s_\eps(t)\eps^{-1}\eps^{p(M-\frac{d}{2}-1)}\eps^{|\alpha|-|\beta|-|\gamma|-|\zeta|}\eps^{q}\leq C_1^pA_1^q\mathcal{E}^s_\eps(t)\eps^{\mu}\text{ for }t\in [0,T_\eps].
\end{split}
\end{align}
Here $\mu>0$ has the same value as in \eqref{b7}; the ``extra" factor of $\eps^{-1}$ that was previously contributed by $D_x^2 u_a$ is now contributed by $D_x^2\nu$.
When $p=0$ we clearly obtain an estimate just like \eqref{b8}.

\textbf{5. }The  case when $\partial_i$  hits one of the $\partial^{\gamma_k}D_xv$ in \eqref{b4} clearly yields exactly the same right hand side as in \eqref{b9}.    Finally, when $\partial_i$ hits $H$ in \eqref{b4} one obtains an expression that can be estimated using  the other cases.
\end{proof}

Next we estimate the interior forcing term in $\eps^\alpha\partial^\alpha(\eps\eqref{H1})$, namely $\eps^{1+|\alpha|}\partial^\alpha\mathbb{F}_1$.

\begin{lem}[Interior forcing for \eqref{H1}]\label{iforh1}
 Let $\eps_1$ and $C_1$ be as in Proposition \ref{mainestimate}. 
 There exist positive constants   $\eps_2(C_1)\leq \eps_1$  and $B_1=B_1(A_1)$ such that for $\eps\in [0,\eps_2]$  and all $t \in [0,T_\eps]$\footnote{The constant $B_1$ in \eqref{b10} is, of course, not necessarily the same as  the $B_1$ in Lemma \ref{icomh1}.   We are only  interested in keeping track of what such constants depend on, not their particular values.   
 We will continue to redefine certain constants in this way in order to reduce the  number of distinct labels needed for constants.}
\begin{align}\label{b10}
|\eps^{1+|\alpha|}\partial^\alpha\mathbb{F}_1(t)|_{L^2(x)}\leq B_1 \mathcal{E}^s_\eps(t)+\eps|R^\eps(t)|_{s,\eps} \text{ for all }|\alpha|\leq s. 
\end{align}
%Here (and below) $i,j\in [1,\dots,d]$. 

\end{lem}

\begin{proof}
\textbf{1. }The first term of $\mathbb{F}_1$ is $\eps^{-M}\mathcal{F}=R^\eps$, which obviously gives rise to the term $\eps|R^\eps(t)|_{s,\eps}$ in \eqref{b10}.

\textbf{2. }The remaining part of $\mathbb{F}_1$ can be written in the form 
\begin{align}\label{b11}
\sum^d_{i=1}\partial_i[H_i(D_xu_a,D_xv)(D_xv,D_x\nu)], 
\end{align}
where $H_i$ is a $C^\infty$ function of its arguments.   First expanding out  
\begin{align}\label{b11a}
\eps^{1+|\alpha|}\partial^\alpha[H_i(D_xu_a,D_xv)(D_xv,D_x\nu)]
\end{align}
we obtain a sum of terms of the form
\begin{align}\label{b12}
\eps^{1+|\alpha|}H(D_xu_a,D_xv)(\partial^{\beta_1}D_xu_a)\dots(\partial^{\beta_q}D_xu_a)(\partial^{\gamma_1}D_x v)\dots(\partial^{\gamma_p}D_x v)(\partial^{\zeta_1}D_xv)\partial^{\zeta_2}D_x\nu,
\end{align}
where 
\begin{align}
|\beta|+|\gamma|+|\zeta|=|\alpha|\leq s, 
\end{align}
and $q$ or $p$ (possibly both) can be zero.   Applying $\partial_i$ to \eqref{b12}, one again obtains four cases depending on whether $\partial_i$ hits $H$, one of the $\partial^{\beta_k}D_xu_a$, one of the $\partial^{\gamma_k}D_xv$, or one of the final two factors in \eqref{b12}.

\textbf{3. }For example, in the fourth case we obtain by arguing as in the proof of Lemma \ref{icomh1}:
\begin{align}\label{b13}
\begin{split}
&|\eps^{1+|\alpha|}H(D_xu_a,D_xv)(\partial^{\beta_1}D_xu_a)(\partial^{\beta_2}D_xu_a)\dots(\partial^{\beta_q}D_xu_a)(\partial^{\gamma_1}D_x v)\dots(\partial^{\gamma_p}D_x v)(\partial^{\zeta_1}D_xv)\partial^{\zeta_2} D_x^2\nu|_{L^2(x)} \\
&\leq C_1^{p+1}A_1^q\mathcal{E}^s_\eps(t)\eps^{-1}\eps^{(p+1)(M-\frac{d}{2}-1)}\eps^{|\alpha|-|\beta|-|\gamma|-|\zeta|}\eps^{q}\leq C_1^{p+1}A_1^q\mathcal{E}^s_\eps(t)\eps^{\mu}\text{ for }t\in [0,T_\eps],
\end{split}
\end{align}
where $\mu>0$. In \eqref{b13} we have exhibited the ``extra" factor of $\eps^{-1}$ contributed by $\partial^\zeta D_x^2\nu$.  The other factors of $C_1$, $A_1$, and $\eps$
arise just as in \eqref{b6a}, \eqref{b6b}.

\textbf{4. }The second and third cases (described in step \textbf{2}) yield exactly the same estimate. As before the first case can be treated using the second and third cases.

\end{proof}

It remains to estimate the boundary commutator and boundary forcing for $\eps^\alpha\partial^\alpha(\eps\eqref{H1})$.

\begin{lem}[Boundary commutator for \eqref{H1}]\label{bcomh1}
 Let $\eps_1$ and $C_1$ be as in Proposition \ref{mainestimate}. 
 There exist positive constants   $\eps_2(C_1)\leq \eps_1$  and $B_1=B_1(A_1)$ such that for $\eps\in [0,\eps_2]$  and all $t \in [0,T_\eps]$ 
\begin{align}\label{b14}
\eps|\Lambda^{\frac{1}{2}}_{x'}[\eps^{|\alpha|}\partial^\alpha,A_{ij}(D_x(u_a+v))]\partial_j\nu|_{L^2(x')}\leq B_1 \mathcal{E}^s_\eps(t) \text{ for all }|\alpha|\leq s. 
\end{align}
\end{lem}

\begin{proof}
\textbf{1. }Denote the commutator $[\eps^\alpha\partial^\alpha,A_{ij}(D_x(u_a+v))]\partial_j\nu$ by $\mathcal{C}$ and observe that $\eps\mathcal{C}$ is a sum of terms of the form \eqref{b4}.  By the trace estimate of Proposition \ref{trace} we have
\begin{align}\label{b15}
\eps\langle \Lambda^{\frac{1}{2}}_{x'}\mathcal{C}(t)\rangle_{0,\eps}\leq |\mathcal{C}(t)|_{0,1,\eps}\leq |\mathcal{C}(t)|_{L^2(x)}+|\eps D_x\mathcal{C}(t)|_{L^2(x)}.
\end{align} 
Since $\eps D_x\mathcal{C}(t)$ is a sum of terms of the form $D_x$\eqref{b4}, the second term on the right in \eqref{b15} has already been estimated in Lemma \ref{icomh1} 

\textbf{2. }To estimate  $|\mathcal{C}(t)|_{L^2(x)}$ we must estimate  terms like $|\eps^{-1}$\eqref{b4}$|_{L^2(x)}$. In the cases $p\geq 1$, $p=0$ we obtain
respectively
\begin{align}
\begin{split}
&|\eps^{-1}\eqref{b4}|_{L^2(x)}\leq  C_1^pA_1^q\mathcal{E}^s_\eps(t)\eps^{p(M-\frac{d}{2}-1)}\eps^{|\alpha|-|\beta|-|\gamma|-|\zeta|}\eps^{q-1}\leq C_1^pA_1^q\mathcal{E}^s_\eps(t)\eps^{\mu}\\
&|\eps^{-1}\eqref{b4}|_{L^2(x)}\leq A_1^q\mathcal{E}^s_\eps(t)\eps^{|\alpha|-|\beta|-|\zeta|}\eps^{q-1}\leq A_1^q\mathcal{E}^s_\eps(t),
\end{split}
\end{align}
where $\mu>0$ in the first case, and $q\geq 1$ in the second.
\end{proof}

\begin{lem}[Boundary forcing for \eqref{H1}]\label{bforh1}
 Let $\eps_1$ and $C_1$ be as in Proposition \ref{mainestimate}. 
 There exist positive constants   $\eps_2(C_1)\leq \eps_1$  and $B_1=B_1(A_1)$ such that for $\eps\in [0,\eps_2]$  and all $t \in [0,T_\eps]$ 
\begin{align}\label{b16}
|\eps\Lambda^{\frac{1}{2}}_{x'}\eps^{|\alpha|}\partial^\alpha\mathbb{G}_1(t)|_{L^2(x')}\leq B_1 \mathcal{E}^s_\eps(t)+\langle r^\eps(t)\rangle_{s+1,\eps} \text{ for all }|\alpha|\leq s. 
\end{align}
%Here (and below) $i,j\in [1,\dots,d]$. 
\end{lem}

\begin{proof}
\textbf{1. }The first term of $\mathbb{G}_1$ is $\eps^{-M}\mathcal{G}=r^\eps$, which clearly gives rise to the term $\langle r^\eps(t)\rangle_{s+1,\eps}$ in \eqref{b10}.

\textbf{2. }The remaining part of $\mathbb{G}_1$ is a sum of terms like 
\begin{align}\label{b17}
H_i(D_xu_a,D_xv)(D_xv,D_x\nu), 
\end{align}
where $H_i$ is a $C^\infty$ function of its arguments.   Denote $\eps^\alpha\partial^\alpha\eqref{b17}$  by $\mathcal{K}$ and observe that $\eps\mathcal{K}$ is a sum of terms of the form \eqref{b12}.  By the trace estimate of Proposition \ref{trace} we have
\begin{align}\label{b18}
\eps\langle \Lambda^{\frac{1}{2}}_{x'}\mathcal{K}(t)\rangle_{0,\eps}\leq |\mathcal{K}(t)|_{0,1,\eps}\leq |\mathcal{K}(t)|_{L^2(x)}+|\eps D_x\mathcal{K}(t)|_{L^2(x)}.
\end{align} 
Since $\eps D_x\mathcal{K}(t)$ is a sum of terms of the form $D_x$\eqref{b12}, the second term on the right in \eqref{b18} has already been estimated in Lemma \ref{iforh1}.

\textbf{3. }To estimate  $|\mathcal{K}(t)|_{L^2(x)}$ we must estimate  terms like $|\eps^{-1}$\eqref{b12}$|_{L^2(x)}$. For any integer $p\geq 0$ we obtain
\begin{align}
\begin{split}
&|\eps^{-1}\eqref{b12}|_{L^2(x)}\leq  C_1^{p+1}A_1^q\mathcal{E}^s_\eps(t)\eps^{(p+1)(M-\frac{d}{2}-1)}\eps^{|\alpha|-|\beta|-|\gamma|-|\zeta|}\eps^{q-1}\leq C_1^{p+1}A_1^q\mathcal{E}^s_\eps(t)\eps^{\mu}\\
%&|\eps^{-1}\eqref{b12}|_{L^2(x)}\leq A_1^q\mathcal{E}^s_\eps(t)\eps^{|\alpha|-|\beta|-|\zeta|}\eps^{q-1}\leq A_1^q\mathcal{E}^s_\eps(t),
\end{split}
\end{align}
where $\mu>0$.
\end{proof}

The next proposition summarizes what we have shown so far.

\begin{prop}\label{b19}
%Let $s>...$ and let $T$ (which is independent of $\eps$)   be the time of existence of the approximate solution $u^\eps_a$.  
%Let $M$, $s$, and $T$ be as in Theorem \ref{main}. 
Suppose $\eps_1$ and  $C_1$ are positive constants and that for $\eps\in (0,\eps_1]$ and $T_\eps\in [0,T]$, we are given a solution $(\nu^\eps,\omega^\eps)$ of 
the three coupled systems on $[0,T_\eps]$ which satisfies  $\mathcal{E}^s_\eps(t)\leq C_1$ for all $t\in [0,T_\eps]$.   Then there exist positive constants   $\eps_2(C_1)\leq \eps_1$   and $B_1(A_1,K_1)$ such that for $\eps\in [0,\eps_2]$  and all $t \in [0,T_\eps]$ 
\begin{align}\label{b20}
|\eps \overline{D} \nu(t)|^2_{E^s_{\eps,tan}}\leq B_1\int ^t_0 \left([\mathcal{E}^s_\eps(\sigma)]^2+\eps^2|R^\eps(\sigma)|^2_{s,\eps}+\langle r^\eps(\sigma)\rangle ^2_{s+1,\eps}\right)d\sigma.
\end{align}

\end{prop}

\begin{proof}
The proposition is obtained by 
applying the hyperbolic estimate \eqref{a3} to $\eps^\alpha\partial^\alpha (\eps \eqref{H1})$ for $|\alpha|\leq s$.   The function that plays the role of ``$f$" in \eqref{a3} is a sum of terms  estimated in Lemmas    \ref{icomh1} and \ref{iforh1}, while the function that plays the role of ``$g$" is a sum of terms estimated in Lemmas \ref{bcomh1} and \ref{bforh1}.
\end{proof}

We now estimate $|\eps^2 \overline{D} \omega(t)|_{E^s_{\eps,tan}}$ by applying the estimate \eqref{a3} to $\eps^\alpha\partial^\alpha (\eps^2 \eqref{H2})$ for $|\alpha|\leq s$.\footnote{Here 
``$\eps^\alpha\partial^\alpha (\eps \eqref{H1})$" denotes the result of multiplying the problem \eqref{H1} by $\eps$ and then applying the operator $\eps^\alpha\partial^\alpha$.}

\begin{lem}[Interior commutator for \eqref{H2}]\label{icomh2}
 Let $\eps_1$ and $C_1$ be as in Proposition \ref{mainestimate}. 
 There exist positive constants   $\eps_2(C_1)\leq \eps_1$  and $B_1=B_1(A_1)$ such that for $\eps\in [0,\eps_2]$  and all $t \in [0,T_\eps]$ 
\begin{align}\label{b21}
|\partial_i[\eps^\alpha\partial^\alpha,A_{ij}(D_x(u_a+v))]\partial_j(\eps^2\omega)|_{L^2(x)}\leq B_1 \mathcal{E}^s_\eps(t) \text{ for all }|\alpha|\leq s. 
\end{align}
Here (and below) $i,j\in [1,\dots,d]$. 

\end{lem}

\begin{proof}

The proof can be obtained by repeating verbatim the proof of Lemma \ref{icomh1}  and replacing $\nu$ with $\eps \omega$ wherever $\nu$ occurs in that proof.
\end{proof}

\begin{lem}[Interior forcing for \eqref{H2}]\label{iforh2}
 Let $\eps_1$ and $C_1$ be as in Proposition \ref{mainestimate}. 
 There exist positive constants   $\eps_2(C_1)\leq \eps_1$  and $B_1=B_1(A_1)$ such that for $\eps\in [0,\eps_2]$  and all $t \in [0,T_\eps]$ 
\begin{align}\label{b22}
|\eps^{2+|\alpha|}\partial^\alpha\mathbb{F}_2(t)|_{L^2(x)}\leq B_1 \mathcal{E}^s_\eps(t)+\eps|R^\eps(t)|_{s+1,\eps} \text{ for all }|\alpha|\leq s. 
\end{align}
%Here (and below) $i,j\in [1,\dots,d]$. 

\end{lem}

\begin{proof}
\textbf{1. }The first term of $\mathbb{F}_2$ is $\eps^{-M}\partial_t\mathcal{F}=\partial_tR^\eps$, which obviously gives rise to the term $\eps|R^\eps(t)|_{s+1,\eps}$ in \eqref{b22}.

\textbf{2. }The remaining part of $\mathbb{F}_2$ can be written in the form 
\begin{align}\label{b23}
\sum^d_{i=1}\partial_i[H_i(D_xu_a,D_xv)D_x\nu \;\partial_t\partial_j u_a] 
\end{align}
where $H_i$ is a $C^\infty$ function of its arguments.   First expanding out  
\begin{align}\label{b24}
\eps^{2+|\alpha|}\partial^\alpha[H_i(D_xu_a,D_xv)D_x\nu\; \partial_t\partial_ju_a]
\end{align}
we obtain a sum of terms of the form
\begin{align}\label{b25}
\eps^{2+|\alpha|}H(D_xu_a,D_xv)(\partial^{\beta_1}D_xu_a)\dots(\partial^{\beta_q}D_xu_a)(\partial^{\gamma_1}D_x v)\dots(\partial^{\gamma_p}D_x v)(\partial^{\zeta}D_x\nu)\partial^{\kappa}\partial_t\partial_j u_a,
\end{align}
where 
\begin{align}\label{b26}
|\beta|+|\gamma|+|\zeta|+|\kappa|=|\alpha|\leq s, 
\end{align}
and any of the summands in \eqref{b26} can be zero.      Applying $\partial_i$ to \eqref{b25}, one again obtains four cases depending on whether $\partial_i$ hits $H$, one of the $\partial^{\beta_k}D_xu_a$, one of the $\partial^{\gamma_k}D_xv$,  $\partial^{\zeta}D_x\nu$, or $\partial^{\kappa}\partial_t\partial_j u_a$.

\textbf{3. }Consider for example the third case.   When $p\geq 1$, letting ``case 3" denote  the typical term in that case, we obtain  by arguing as in the proof of Lemma \ref{icomh1}:
\begin{align}\label{b27}
\begin{split}
&|\text{case 3}|\lesssim \eps^{2+|\alpha|}(A_1^q\eps^{q-|\beta|})(C_1^p\eps^{-p\frac{d}{2}}\eps^{-|\gamma|-p}\eps^{Mp})(\mathcal{E}^s_\eps\eps^{-|\zeta|-2})(A_1\eps^{-|\kappa|})=\\
& C_1^{p}A_1^{q+1}\mathcal{E}^s_\eps(t)\eps^{p(M-\frac{d}{2}-1)}\eps^{|\alpha|-|\beta|-|\gamma|-|\zeta|-|\kappa|}\eps^{q}\leq C_1^{p}A_1^{q+1}\mathcal{E}^s_\eps(t)\eps^{\mu}\text{ for }t\in [0,T_\eps],
\end{split}
\end{align}
where $\mu>1$. 

When $p=0$ we simply obtain
\begin{align}
|\text{case 3}|\lesssim A_1^{q+1}\mathcal{E}^s_\eps(t)\eps^{|\alpha|-|\beta|-|\zeta|-|\kappa|}\eps^{q}\leq A_1^{q+1}\mathcal{E}^s_\eps(t).
\end{align}
%In \eqref{b13} we have exhibited the ``extra" factor of $\eps^{-1}$ contributed by $\partial^\zeta D_x^2\nu$.  The other factors of $C_1$, $A_1$, and $\eps$
%arise just as in \eqref{b6a}, \eqref{b6b}.

\textbf{4. }The estimate in case 3 for $p\geq 1$ readily implies the same estimate for case 2, and it is clear that the remaining two cases give the same estimates as in case 3.

\end{proof}

\begin{lem}[Boundary commutator for \eqref{H2}]\label{bcomh2}
 Let $\eps_1$ and $C_1$ be as in Proposition \ref{mainestimate}. 
 There exist positive constants   $\eps_2(C_1)\leq \eps_1$  and $B_1=B_1(A_1)$ such that for $\eps\in [0,\eps_2]$  and all $t \in [0,T_\eps]$ 
\begin{align}\label{b28}
\eps^2|\Lambda^{\frac{1}{2}}_{x'}[\eps^{|\alpha|}\partial^\alpha,A_{ij}(D_x(u_a+v))]\partial_j\omega|_{L^2(x')}\leq B_1 \mathcal{E}^s_\eps(t) \text{ for all }|\alpha|\leq s. 
\end{align}
\end{lem}

\begin{proof}

The proof can be obtained by repeating verbatim the proof of Lemma \ref{bcomh1}  and replacing $\nu$ with $\eps \omega$ wherever $\nu$ occurs in that proof.

\end{proof}

\begin{lem}[Boundary forcing for \eqref{H2}]\label{bforh2}
 Let $\eps_1$ and $C_1$ be as in Proposition \ref{mainestimate}. 
 There exist positive constants   $\eps_2(C_1)\leq \eps_1$  and $B_1=B_1(A_1)$ such that for $\eps\in [0,\eps_2]$  and all $t \in [0,T_\eps]$ 
\begin{align}\label{b29}
|\eps^2\Lambda^{\frac{1}{2}}_{x'}\eps^{|\alpha|}\partial^\alpha\mathbb{G}_2(t)|_{L^2(x')}\leq B_1 \mathcal{E}^s_\eps(t)+\langle r^\eps(t)\rangle_{s+2,\eps} \text{ for all }|\alpha|\leq s. 
\end{align}
%Here (and below) $i,j\in [1,\dots,d]$. 
\end{lem}

\begin{proof}
\textbf{1. }The first term of $\mathbb{G}_2$ is $\eps^{-M}\partial_t\mathcal{G}=\partial_tr^\eps$, which clearly gives rise to the term $\langle r^\eps(t)\rangle_{s+2,\eps}$ in \eqref{b10}.

\textbf{2. }The remaining part of $\mathbb{G}_2$ is a sum of terms like 
\begin{align}\label{b30}
H_i(D_xu_a,D_xv)D_x\nu \;\partial_t\partial_j u_a
\end{align}
where $H_i$ is a $C^\infty$ function of its arguments.   Denote $\eps^\alpha\partial^\alpha\eqref{b30}$  by $\mathcal{K}$ and observe that $\eps^2\mathcal{K}$ is a sum of terms of the form \eqref{b25}.  By the trace estimate of Proposition \ref{trace} we have
\begin{align}\label{b31}
\eps^2\langle \Lambda^{\frac{1}{2}}_{x'}\mathcal{K}(t)\rangle_{0,\eps}\leq \eps|\mathcal{K}(t)|_{0,1,\eps}\leq |\eps\mathcal{K}(t)|_{L^2(x)}+|\eps^2 D_x\mathcal{K}(t)|_{L^2(x)}.
\end{align} 
Since $\eps^2 D_x\mathcal{K}(t)$ is a sum of terms of the form $D_x$\eqref{b25}, the second term on the right in \eqref{b31} has already been estimated in Lemma \ref{iforh2}.

\textbf{3. }To estimate  $|\eps\mathcal{K}(t)|_{L^2(x)}$ we must estimate  terms like $|\eps^{-1}$\eqref{b25}$|_{L^2(x)}$. For the cases  $p\geq 1$, $p=0$  we obtain by the usual procedure, respectively,
\begin{align}
\begin{split}
&|\eps^{-1}\eqref{b25}|_{L^2(x)}\lesssim C_1^{p}A_1^{q+1}\mathcal{E}^s_\eps(t)\eps^{\mu}, \text{ where }\mu>0\\
&|\eps^{-1}\eqref{b25}|_{L^2(x)}\lesssim A_1^{q+1}\mathcal{E}^s_\eps(t).
\end{split}
\end{align}
\end{proof}

We now update Proposition \ref{b19} to reflect the results of the previous four lemmas.

\begin{prop}\label{b32}
%Let $s>...$ and let $T$ (which is independent of $\eps$)   be the time of existence of the approximate solution $u^\eps_a$.  
%Let $M$, $s$, and $T$ be as in Theorem \ref{main}. 
Suppose $\eps_1$ and  $C_1$ are positive constants and that for $\eps\in (0,\eps_1]$ and $T_\eps\in [0,T]$, we are given a solution $(\nu^\eps,\omega^\eps)$ of 
the three coupled systems on $[0,T_\eps]$ which satisfies  $\mathcal{E}^s_\eps(t)\leq C_1$ for all $t\in [0,T_\eps]$.   Then there exist positive constants   $\eps_2(C_1)\leq \eps_1$   and $B_1(A_1,K_1)$ such that for $\eps\in [0,\eps_2]$  and all $t \in [0,T_\eps]$ 
\begin{align}\label{b33}
|\eps \overline{D} \nu(t)|^2_{E^s_{\eps,tan}}+|\eps^2 \overline{D} \omega(t)|^2_{E^s_{\eps,tan}}\leq B_1\int ^t_0 \left([\mathcal{E}^s_\eps(\sigma)]^2+\eps^2|R^\eps(\sigma)|^2_{s+1,\eps}+\langle r^\eps(\sigma)\rangle ^2_{s+2,\eps}\right)d\sigma.
\end{align}

\end{prop}

\begin{proof}
The proposition is obtained by 
applying the hyperbolic estimate \eqref{a3} to $\eps^\alpha\partial^\alpha (\eps^2 \eqref{H2})$ for $|\alpha|\leq s$, and combining the result with that of Proposition \ref{b19}.   

\end{proof}

It remains to estimate $|\eps^2 \overline{D}^2_x \nu(t)|_{E^s_{\eps,tan}}$ by applying the elliptic estimate \eqref{a3} to $\eps^\alpha\partial^\alpha (\eps^2 \eqref{E})$ for $|\alpha|\leq s$.

\begin{lem}[Interior commutator for \eqref{E}]\label{icome}
Let $\eps_1$ and $C_1$ be as in Proposition \ref{mainestimate}. 
 There exist positive constants   $\eps_2(C_1)\leq \eps_1$  and $B_1=B_1(A_1)$ such that for $\eps\in [0,\eps_2]$  and all $t \in [0,T_\eps]$ 
\begin{align}\label{b34}
|\partial_i[\eps^\alpha\partial^\alpha,A_{ij}(D_x u_a)]\partial_j(\eps^2\nu)|_{L^2(x)}\leq \eps B_1 \mathcal{E}^s_\eps(t) \text{ for all }|\alpha|\leq s. 
\end{align}

\end{lem}

\begin{proof}
We must estimate a sum of the terms of the form $\eps\partial_i\eqref{b4}$, except that now $H=H(D_xu_a)$ and $p=0$. Thus, the work is already done in  the proof of Lemma \ref{icomh1}. 

\end{proof}

\begin{lem}[Interior forcing for \eqref{E}]\label{ifore}
 Let $\eps_1$ and $C_1$ be as in Proposition \ref{mainestimate}. 
 There exist positive constants   $\eps_2(C_1)\leq \eps_1$  and $B_1=B_1(A_1)$ such that for $\eps\in [0,\eps_2]$  and all $t \in [0,T_\eps]$ 
\begin{align}\label{b35}
\begin{split}
&|\eps^{2+|\alpha|}\partial^\alpha\mathbb{F}_3(t)|_{L^2(x)}\leq \\
&\;\;\;|\eps^2\overline{D}\omega(t)|_{E^s_{\eps,tan}}+\left(\lambda\int_0^t\mathcal{E}^s_{\eps,tan}(\sigma)d\sigma+\eps\lambda\mathcal{E}^s_{\eps,tan}(t)\right)+\eps B_1 \mathcal{E}^s_\eps(t)+\eps^2A_2 \text{ for all }|\alpha|\leq s. 
\end{split}
\end{align}

%Here (and below) $i,j\in [1,\dots,d]$. 

\end{lem}

\begin{proof}
From equation \eqref{E}(a) we see that there are four terms to estimate. 

\textbf{1. }The  term  $\eps^{-M}\mathcal{F}=R^\eps$ in $\mathbb{F}_3$  clearly gives rise to the term $\eps^2 A_2$ in \eqref{b10}.

\textbf{2. }The term $\eps^{-M}E(\overline{D}^2_x v)$  can be written in the form 
\begin{align}\label{b36}
\sum^d_{i=1}\partial_i[H_i(D_xu_a,D_xv)(D_xv,D_x\nu)], 
\end{align}
just like \eqref{b11}.  The estimates of Lemma \ref{iforh1} thus show 
$$\eps^{2+|\alpha|}|\partial^\alpha\eps^{-M}E(\overline{D}^2_x v)(t)|_{L^2(x)}\leq \eps B_1\mathcal{E}^s_\eps(t).$$

\textbf{3. }We have $\eps^{2+|\alpha|}|\partial^\alpha\partial_t\omega|_{L^2(x)}\leq |\eps^2\overline{D}\omega(t)|_{E^2_{\eps,tan}}$, a term estimated in \eqref{b33}.

\textbf{4. }With $\partial^\alpha=\partial_{t,x'}^{\alpha_0,\alpha'}$,   when $\alpha_0\geq 1$ or $\alpha_0=0$  we have respectively
\begin{align}
\begin{split}
&\partial^\alpha \left(\lambda\int^t_0 \omega(\sigma,x)d\sigma\right)=\lambda \partial^\beta\omega, \text{ where }|\beta|=|\alpha|-1,\\
&\partial^\alpha \left(\lambda\int^t_0 \omega(\sigma,x)d\sigma\right)=\lambda\int^t_0 \partial^{\alpha'}_{x'}\omega(\sigma,x)d\sigma, \text{ where }|\alpha'|=|\alpha|,
\end{split}
\end{align}
and corresponding estimates
\begin{align}
\begin{split}
&\eps^{2+|\alpha|}\lambda |\partial^\beta\omega|_{L^2(x)}\leq \eps\lambda |\eps^2\omega |_{E^{s-1}_{\eps,tan}}\leq \eps\lambda\mathcal{E}^s_{\eps,tan}(t)\\
&\eps^{2+|\alpha|}|\lambda\int^t_0 \partial^{\alpha'}_{x'}\omega(\sigma,x)d\sigma|_{L^2(x)}\leq \lambda\int_0^t\mathcal{E}^s_{\eps,tan}(\sigma)d\sigma.
\end{split}
\end{align}

\end{proof}

\begin{lem}[Boundary commutator for \eqref{E}]\label{bcome}
Let $\eps_1$ and $C_1$ be as in Proposition \ref{mainestimate}. 
 There exist positive constants   $\eps_2(C_1)\leq \eps_1$  and $B_1=B_1(A_1)$ such that for $\eps\in [0,\eps_2]$  and all $t \in [0,T_\eps]$ 
\begin{align}\label{b37}
\eps^2|\Lambda^{\frac{1}{2}}_{x'}[\eps^{|\alpha|}\partial^\alpha,A_{ij}(D_xu_a)]\partial_j\nu|_{L^2(x')}\leq \eps B_1 \mathcal{E}^s_\eps(t) \text{ for all }|\alpha|\leq s. 
\end{align}

\end{lem}

\begin{proof}
The estimate is immediate from the argument in the $p=0$ case  of the proof of Lemma \ref{bcomh1}.
\end{proof}

\begin{lem}[Boundary forcing for \eqref{E}]\label{bfore}
 Let $\eps_1$ and $C_1$ be as in Proposition \ref{mainestimate}. 
 There exist positive constants   $\eps_2(C_1)\leq \eps_1$  and $B_1=B_1(A_1)$ such that for $\eps\in [0,\eps_2]$  and all $t \in [0,T_\eps]$ 
\begin{align}\label{b38}
|\eps^2\Lambda^{\frac{1}{2}}_{x'}\eps^{|\alpha|}\partial^\alpha\mathbb{G}_3(t)|_{L^2(x')}\leq \eps B_1 \mathcal{E}^s_\eps(t)+\eps A_2 \text{ for all }|\alpha|\leq s. 
\end{align}
%Here (and below) $i,j\in [1,\dots,d]$. 
\end{lem}

\begin{proof}
We have $\mathbb{G}_3=\eps^{-M}\mathcal{G}+\eps^{-M}E_b(\overline{D}_xv)$, where  the second term is a sum of terms of the form \eqref{b17}.   So this lemma follows from the proof of Lemma \ref{bforh1}. 

\end{proof}

Combining the results of the previous four lemmas, we obtain the following estimate for $|\eps^2\overline{D}_x^2\nu(t)|_{E^s_{\eps,tan}}$ by applying the elliptic estimate \eqref{a3} to 
$\eps^\alpha\partial^\alpha (\eps^2 \eqref{E})$ for $|\alpha|\leq s$.

\begin{prop}\label{b39}
%Let $s>...$ and let $T$ (which is independent of $\eps$)   be the time of existence of the approximate solution $u^\eps_a$.  
%Let $M$, $s$, and $T$ be as in Theorem \ref{main}. 
Suppose $\eps_1$ and  $C_1$ are positive constants and that for $\eps\in (0,\eps_1]$ and $T_\eps\in [0,T]$, we are given a solution $(\nu^\eps,\omega^\eps)$ of 
the three coupled systems on $[0,T_\eps]$ which satisfies  $\mathcal{E}^s_\eps(t)\leq C_1$ for all $t\in [0,T_\eps]$.   Then there exist positive constants   $B_1(A_1)$ and $\eps_2(C_1)\leq \eps_1$  such that for $\eps\in [0,\eps_2]$  and all $t \in [0,T_\eps]$ 
\begin{align}\label{b40}
|\eps^2 \overline{D}^2_x \nu(t)|_{E^s_{\eps,tan}}\leq K_2\left(\eps B_1\mathcal{E}^s_\eps(t)+|\eps^2\overline{D}\omega|_{E^s_{\eps,tan}}+\lambda\int^t_0\mathcal{E}^s_{\eps,\tan}(\sigma)d\sigma+\eps\lambda \mathcal{E}^s_{\eps,\tan}(t)+\eps A_2\right).
\end{align}
\end{prop}

Putting together Propositions \ref{b32} and \ref{b39} we obtain
\begin{prop}\label{b41}
Under the assumptions of Proposition \ref{b39} there exist positive constants \\$B_1(T,A_1,K_1,K_2,\lambda)$, $B_2(A_2,K_2)$, and $\eps_2(C_1,A_1,K_1,K_2,\lambda)\leq \eps_1$  such that for $\eps\in [0,\eps_2]$  and all $t \in [0,T_\eps]$ 
\begin{align}\label{b42}
[\mathcal{E}^s_{\eps,tan}(t)]^2\leq B_1\int ^t_0 \left([\mathcal{E}^s_\eps(\sigma)]^2+\eps^2|R^\eps(\sigma)|^2_{s+1,\eps}+\langle r^\eps(\sigma)\rangle ^2_{s+2,\eps}\right)d\sigma+\eps^2(B_1[\mathcal{E}^s_\eps(t)]^2+B_2).
\end{align}

\end{prop}

\begin{proof}
Take the square of \eqref{b40}, use \eqref{b33} to estimate the term $|\eps^2\overline{D}\omega|^2_{E^s_{\eps,tan}}$ that appears on the right, and add the resulting estimate of $|\eps^2 \overline{D}^2_x \nu(t)|^2_{E^s_{\eps,tan}}$ to the estimate \eqref{b33}.  The left side now equals $[\mathcal{E}^s_{\eps,tan}(t)]^2$ and can be used to absorb the term $K_2^2\eps^2\lambda^2 [\mathcal{E}^s_{\eps,tan}(t)]^2$ from the right, provided $\eps_2$ is small enough.  Finally, to obtain \eqref{b42}  we have also used
\begin{align}
K_2^2\lambda^2\left(\int^t_0\mathcal{E}^s_{\eps,\tan}(\sigma)d\sigma\right)^2\leq K_2^2\lambda^2T^2\int^t_0[\mathcal{E}^s_{\eps,\tan}(\sigma)]^2d\sigma \leq  K_2^2\lambda^2T^2\int^t_0[\mathcal{E}^s_{\eps}(\sigma)]^2d\sigma.
\end{align}

\end{proof}

\subsection{Normal derivative estimates}

In order to complete the proof of Proposition \ref{mainestimate} we must estimate 
\begin{align}\label{c1}
|\eps^{|\alpha|}\partial_{t,x}^\alpha (\eps\overline{D}\nu, \eps^2\overline{D}\omega,\eps^2\overline{D}^2_x\nu)|_{L^2(x)}\text{ for }|\alpha|\leq s
\end{align}
when $\partial_d$ derivatives are present in $\partial^\alpha_{t,x}$.  In this section $\partial^\beta$ will always be taken to mean $\partial^{(\beta_0,\beta',\beta_d)}_{t,x',x_d}$, where $\beta_d\leq |\beta|$.   We have a noncharacteristic boundary ($A_{dd}$ is nonsingular), so we can ``use the equation" to control normal derivatives in an inductive argument starting with the control we now have over tangential derivatives (the case $\alpha_d=0$ in \eqref{c1}).   Although this type of argument is  standard, we have three equations here and a rather complicated object $\mathcal{E}^s_\eps$ to estimate, so some care is needed both to formulate the induction assumption concisely and to avoid unnecessary work.  Thus, we shall provide some details. 

First we define for $s_0\leq s$:
\begin{align}\label{c2}
\begin{split}
&|u(t)|_{E^s_{\eps,s_0}}=\sup_{|\alpha|\leq s, \alpha_d\leq s_0}\eps^{|\alpha|}|\partial^\alpha u(t,\cdot)|_{L^2(x)}\\
&\mathcal{E}^s_{\eps,s_0}(t)=|\eps^2 \overline{D}^2_x \nu(t)|_{E^s_{\eps,s_0}}+|\eps \overline{D} \nu(t)|_{E^s_{\eps,s_0}}+|\eps^2 \overline{D} \omega(t) |_{E^s_{\eps,s_0}}.
\end{split}
\end{align}
By induction on $s_0$ we will prove:

\begin{prop}\label{c3}
Let $s_0\in\{0,,\dots,s\}$. 
Suppose $\eps_1$ and  $C_1$ are positive constants and that for $\eps\in (0,\eps_1]$ and $T_\eps\in [0,T]$, we are given a solution $(\nu^\eps,\omega^\eps)$ of 
the three coupled systems on $[0,T_\eps]$ which satisfies  $\mathcal{E}^s_\eps(t)\leq C_1$ for all $t\in [0,T_\eps]$.   Then there exist positive constants $B_1=B_1(T,A_1,K_1,K_2,\lambda)$,  $B_2=B_2(A_1,A_2,K_2)$, and $\eps_2=\eps_2(C_1,A_1,K_1,K_2,\lambda)$ such that for $\eps\in [0,\eps_2]$  and all $t \in [0,T_\eps]$ 
% that for $0\leq s_0\leq s$, there are constants  $\eps_2$, $B_1$, $B_2$ as described in Proposition \ref{mainestimate} such that for $\eps\in (0,\eps_2]$ and all $t\in [0,T_\eps]$:
\begin{align}\label{c4}
[\mathcal{E}^s_{\eps,s_0}(t)]^2\leq B_1\int ^t_0 \left([\mathcal{E}^s_\eps(\sigma)]^2+\eps^2|R^\eps(\sigma)|^2_{s+1,\eps}+\langle r^\eps(\sigma)\rangle ^2_{s+2,\eps}\right)d\sigma+\eps^2(B_1[\mathcal{E}^s_\eps(t)]^2+B_2).
\end{align}
\end{prop}

\begin{proof}

\textbf{1. }The case $s_0=s$ is the same as the estimate asserted in  Proposition \ref{mainestimate}; the case $s_0=0$ is  treated in Proposition \ref{b41}.

\textbf{2. Induction assumption. }Let $s_0<s$ and assume that \eqref{c4} holds for this $s_0$.  It remains to show that \eqref{c4} holds for $s_0+1$.
It is perhaps surprising that the terms  $|\eps^2 \overline{D} \omega(t) |_{E^s_{\eps,s_0+1}}$ and $|\eps \overline{D} \nu(t) |_{E^s_{\eps,s_0+1}}$ 
can be estimated \emph{without} having to ``use the equation." Also, in estimating the remaining term $|\eps^2 \overline{D}_x^2 \nu(t) |_{E^s_{\eps,s_0+1}}$ 
we will only need to use one equation, namely \eqref{H1}(a).

\textbf{3. The term $|\eps^2 \overline{D} \omega(t) |_{E^s_{\eps,s_0+1}}.$ }With $\alpha=(\alpha_0,\alpha',s_0+1)$ satisfying $|\alpha|\leq s$ (here and in the remaining steps), we first estimate
\begin{align}\label{c4a}
|\eps^2 {D} \omega(t) |_{E^s_{\eps,s_0+1}}=\eps^{|\alpha|}|\partial_t^{\alpha_0}\partial_{x'}^{\alpha'}\partial_d^{s_0+1}(\eps^2 D\omega)|_{L^2(x)}
\end{align}
If $D$ is replaced by $\partial_t$ in \eqref{c4a}, we can swap this $\partial_t$ with one of the $\partial_d$ derivatives  to obtain $\eqref{c4a}\leq  \mathcal{E}^s_{\eps,s_0}(t)$.
When $D$ is replaced by $D_x$ we write  $\omega=\partial_t\nu$ to obtain
\begin{align}\label{c5}
\eps^{|\alpha|}|\partial_t^{\alpha_0}\partial_{x'}^{\alpha'}\partial_d^{s_0+1}(\eps^2 D_x\partial_t\nu)|_{L^2(x)}=\eps^{|\alpha|}|\partial_t^{\alpha_0}\partial_{x'}^{\alpha'}\partial_d^{s_0}\partial_t(\eps^2 D_x\partial_d\nu)|_{L^2(x)}\leq \mathcal{E}^s_{\eps,s_0}(t).
\end{align}
  When $D$ is absent in \eqref{c4a}, the desired estimate is immediate.

\textbf{4. The term $|\eps \overline{D} \nu(t) |_{E^s_{\eps,s_0+1}}.$}
We consider only
\begin{align}\label{c6}
\eps^{|\alpha|}|\partial_t^{\alpha_0}\partial_{x'}^{\alpha'}\partial_d^{s_0+1}(\eps D\nu)|_{L^2(x)}.
\end{align}
When $D$ is replaced by $\partial_t$ or $\partial_{x'}$, we can swap that derivative with one of the $\partial_d$ derivatives in $\partial_d^{s_0+1}$ to  obtain $\eqref{c6}\leq \mathcal{E}^s_{\eps,s_0}(t)$.   When $D$ is replaced by $\partial_d$, we can rewrite \eqref{c6} as 
$$ 
\eps^{|\alpha|-1}|\partial_t^{\alpha_0}\partial_{x'}^{\alpha'}\partial_d^{s_0}(\eps^2 \partial^2_d\nu)|_{L^2(x)}\leq \mathcal{E}^s_{\eps,s_0}(t).
$$

\textbf{5. The term $|\eps^2 \overline{D}^2_x \nu(t) |_{E^s_{\eps,s_0+1}}.$}    We will show
\begin{align}\label{c6a}
|\eps^2 \overline{D}^2_x \nu(t) |_{E^s_{\eps,s_0+1}}\leq C(A_1)\mathcal{E}^s_{\eps,s_0}+\eps C(A_1)\mathcal{E}^s_\eps+\eps C(A_1,A_2).
\end{align}

To estimate   $|\eps^2 {D}^2_x \nu(t) |_{E^s_{\eps,s_0+1}}$ we consider
\begin{align}\label{c7}
\eps^{|\alpha|}|\partial_t^{\alpha_0}\partial_{x'}^{\alpha'}\partial_d^{s_0+1}(\eps^2 D^2_x\nu)|_{L^2(x)}.
\end{align}
If exactly one $\partial_d$ appears in $D^2_x$, for example, if $D^2_x=\partial_i\partial_d$ with $i<d$, we can swap the $\partial_i$ with one of the $\partial_d$ derivatives in 
$\partial_d^{s_0+1}$ to obtain $\eqref{c7}\leq \mathcal{E}^s_{\eps,s_0}(t)$. We get the same estimate, of course, if no $\partial_d$ appears in $D_x^2$.

To treat the remaining case where $D_x^2=\partial^2_d$ in \eqref{c7},  we first use equation \eqref{H1}(a) to write
\begin{align}\label{c8}
\begin{split}
&\partial_d^2\nu=-A_{dd}^{-1}\left[-\partial_t^2\nu+\partial_d(A_{dd})\partial_d\nu +\sum_{i\;or\;j\neq d}\partial_i(A_{ij}\partial_j\nu)+R^\eps+\sum_i\partial_i[H_i(D_xu_a,D_xv)(D_xv,D_x\nu)]\right].
\end{split}
\end{align}
Here $i,j\in\{1,\dots,d\}$, the coefficients $A_{ij}= A_{ij}(D_xu_a+D_xv)$, and we have used \eqref{b11}.

Replacing ${D}^2_x\nu$ in \eqref{c7} by the right side of \eqref{c6}, we examine first the contribution to \eqref{c7} from 
%\begin{align}
$-A_{dd}^{-1}\left(\sum_{i\;or\;j\neq d}A_{ij}\partial^2_{ij}\nu\right)$, which is a sum of terms of the form $H_{ij}(D_xu_a,D_xv)\partial^2_{ij}\nu$ with $i$ or $j\neq d$.
%\end{align}
Thus, we must estimate terms 
%like $\eps^{|\alpha|+2}|\partial^\alpha(H(D_xu_a,D_xv)\partial^2_{ij}\nu)|_{L^2(x)}$, and therefore 
like  
\begin{align}\label{c9}
|\eps^{2+|\alpha|}H\cdot(\partial^{\beta_1}D_xu_a)(\partial^{\beta_2}D_xu_a)\dots(\partial^{\beta_q}D_xu_a)(\partial^{\gamma_1}D_x v)\dots(\partial^{\gamma_p}D_x v)\partial^\zeta \partial^2_{ij}\nu|_{L^2(x)},
\end{align}
where $|\beta|+|\gamma|+|\zeta|=|\alpha|$, at least one of $i,j$ (say $i$) is $\neq d$, and the total number of $\partial_d$ derivatives appearing in $(\partial^\beta,\partial^\gamma,\partial^\zeta)$ is $s_0+1$.

First consider the ``worst" case, that is, when $p=0$.  We obtain
\begin{align}\label{c10}
\eqref{c9}\leq \eps^{2+|\alpha|}\eps^{q-|\beta|}A_1^q |\partial^\zeta\partial^2_{ij}\nu|_{L^2(x)}\leq \eps^{2+|\alpha|}\eps^{q-|\beta|}A_1^q \mathcal{E}^s_{\eps,s_0}\eps^{-|\zeta|-2}\leq A_1^q\mathcal{E}^s_{\eps,s_0}(t).
\end{align}
Here the second inequality is immediate when $\zeta_d\leq s_0$;  if $\zeta_d=s_0+1$ we swap $\partial_i$ with one of the $\partial_d$ derivatives in $\partial^\zeta$.

When $p\geq 1$ we obtain using the product estimate of Proposition \ref{prod} $p$ times:
\begin{align}\label{c11}
\eqref{c9}\leq  \eps^{2+|\alpha|}\eps^{q-|\beta|}A_1^q \eps^{p(M-\frac{d}{2}-1)}\eps^{-|\gamma|} \mathcal{E}^s_\eps\eps^{-|\zeta|-2}\leq \eps A_1^q\mathcal{E}^s_\eps
\end{align}
for $\eps\in (0,\eps_2]$ if $\eps_2$ small enough (use $p(M-\frac{d}{2}-1)>1$). 

\begin{rem}
The estimate \eqref{c11}  illustrates that when factors like $\partial^{\gamma_j}D_xv$ are present, it is not necessary to use 
the induction assumption to obtain an estimate consistent with \eqref{c6a}.   Every term in the  contribution to \eqref{c7}
when $D^2_x\nu$  is replaced by 
$$-A_{dd}^{-1}\left(\sum_i\partial_i[H_i(D_xu_a,D_xv)(D_xv,D_x\nu)]\right)$$ 
includes at least one such factor and is again dominated by $\eps C(A_1)\mathcal{E}^s_\eps$.
\end{rem}

Next consider  the contribution to \eqref{c7}
when $D^2_x\nu$  is replaced by $-A_{dd}^{-1}\partial_t^2\nu=-A_{dd}^{-1}\partial_t\omega$.  Using step \textbf{3} we easily obtain 
\begin{align}\label{c12}
\eps^{2+|\alpha|}\partial^\alpha |(A_{dd}^{-1}\partial^2_t\nu)|_{L^2(x)}\leq C(A_1)\mathcal{E}^s_{\eps,s_0}+\eps C(A_1)\mathcal{E}^s_\eps.   
\end{align}
The second term here arises when $\partial^\alpha$ hits $A_{dd}$. 

By entirely similar or easier estimates we find
\begin{align}\label{c13}
\begin{split}
&\eps^{2+|\alpha|}|\partial^\alpha [A_{dd}^{-1}\partial_d(A_{dd})\partial_d\nu]|_{L^2(x)}\leq \eps C(A_1)\mathcal{E}^s_\eps\\
&\eps^{2+|\alpha|}|\partial^\alpha [A_{dd}^{-1}\partial_i(A_{ij})\partial_j\nu]|_{L^2(x)}\leq \eps C(A_1)\mathcal{E}^s_\eps, \text{ where }i\text{ or }j\neq d\\
&\eps^{2+|\alpha|}|\partial^\alpha (A_{dd}^{-1}R^\eps)|_{L^2(x)}\leq \eps C(A_1,A_2).
\end{split} 
\end{align}
This completes the proof of \eqref{c6a}.

\textbf{6. }The results of steps \textbf{3,4,5} imply 
\begin{align}
\mathcal{E}^s_{\eps,s_0+1}(t) \leq C(A_1)\mathcal{E}^s_{\eps,s_0}+\eps C(A_1)\mathcal{E}^s_\eps+\eps C(A_1,A_2),
\end{align}
which in turn implies that \eqref{c4} holds for $s_0+1$.   This completes the induction step.

\end{proof} 

We can now finish the proof of Proposition \ref{mainestimate}.

\begin{proof}[End of the proof of Proposition \ref{mainestimate}.]
The case $s_0=s$ in Proposition \ref{c3} gives
\begin{align}
[\mathcal{E}^s_{\eps}(t)]^2\leq B_1\int ^t_0 \left([\mathcal{E}^s_\eps(\sigma)]^2+\eps^2|R^\eps(\sigma)|^2_{s+1,\eps}+\langle r^\eps(\sigma)\rangle ^2_{s+2,\eps}\right)d\sigma+\eps^2(B_1[\mathcal{E}^s_\eps(t)]^2+B_2).
\end{align}
For $\eps_2$ small enough the term $\eps^2 B_1[\mathcal{E}^s_\eps(t)]^2$ can be absorbed into the left side, yielding the estimate \eqref{apriori} of Proposition \ref{mainestimate}.

\end{proof}

As explained in section \ref{mainR}, this completes the proof of Theorem \ref{main}.

\part{Construction of approximate solutions for the traction problem}\label{p3}

In this part we construct high order approximate solutions to the equations of the Saint Venant-Kirchhoff model of nonlinear elasticity with traction boundary conditions.   It will simplify the exposition and greatly lighten the notation to carry out the construction in two space dimensions, but the construction in higher dimensions goes through with only obvious (and almost exclusively notational) changes.  We change notation slightly from part \ref{p2} and denote the normal variable $x_d$ here by $y$ and the single tangential spatial variable by $x$ in place of the earlier $x'$.

%TODO: Explain restriction to 2D, extension to multiD, notation $y$ instead of $x_d$.

\section{Introduction }\label{introp3}

    We consider the equations of the Saint Venant-Kirchhoff model \eqref{j0} in two space dimensions:
	\begin{equation}\label{eq:InitIntSys}
		\begin{split}
		&\partial_t^2\phi-\nabla\cdot(\nabla\phi\sigma(\nabla\phi))=0 \text{ on }y>0\\
&	\nabla\phi\sigma(\nabla\phi){n}=\varepsilon^2G(t,x,\frac{\beta\cdot(t,x)}{\varepsilon}):=
		\eps^2\begin{bmatrix}
			f \\
			g
		\end{bmatrix}\text{ on }y=0\\
		&\phi(t,x,y)=(x,y) \text{ and }G=0 \text{ in }t\leq 0,
	\end{split}
	\end{equation}
	where $\phi=(\phi_1,\phi_2)$ is the deformation, the $2\times 2$ matrix $\sigma$ is the stress (defined below \eqref{j0}),  ${n}=\begin{pmatrix} 0 \\ -1 \end{pmatrix}$, 
	and the boundary forcing is given by $G(t,x,\theta)\in H^\infty([0,T_0]\times \mathbb{R}_x\times \mathbb{T}_\theta)$ for some $T_0>0$.\footnote{Whenever we use an expression like 
	$G(t,x,\theta)\in H^\infty([0,T_0]\times \mathbb{R}_x\times \mathbb{T}_\theta)$, where $G$ is a function that vanishes in $t<0$, it is to be understood that $G$ vanishes to infinite order at $t=0$.}We take $\beta$ of the form
	\begin{align}\label{beta}
	\beta=(-c,1)\in\mathbb{R}^2
	\end{align}
	for a $c$ whose choice is discussed below.   The case where $G$ has finite regularity can easily be treated, but at the cost of much additional bookkeeping.  In order to highlight the phenomenon of internal rectification we assume that the 
 the Fourier mean (or zero-th Fourier mode $G^0(t,x)$) of $G$ is zero.\footnote{The construction goes through just as well if  $G^0$ is not zero.  See Remark \ref{internal} for more on internal rectification.}
The matrix $\nabla\phi$ is given by 	
\begin{equation*}
		\nabla\phi=
		\begin{bmatrix}
			\partial_x\phi_1(t,x,y) & \partial_y\phi_1(t,x,y)\\
			\partial_x\phi_2(t,x,y) & \partial_y\phi_2(t,x,y),
		\end{bmatrix}
	\end{equation*}
	%and $\sigma(\nabla\phi)=\lambda\Tr(E)I+2\mu E$ where $E=\frac{1}{2}^t\nabla\phi\nabla\phi-I$ where $\lambda$ and $\mu$ are the Lam\'e constants. 
	%The Saint Venant-Kirchhoff model is this particular choice of stress, $\sigma$ and strain $E$. 
	and we observe that  $\nabla\varphi\sigma(\nabla\phi)$ is a $2\times 2$ matrix whose entries are cubic polynomials in $\nabla\varphi$.
	
	We now write the system \eqref{j2} for the displacement $U(t,x,y)=\phi(t,x,y)-(x,y)$ as 
	\begin{equation}\label{eq:UInitIntSys}
		\begin{split}
	&(a)	\partial_t^2U+\nabla\cdot(L(\nabla U)+Q(\nabla U)+C(\nabla U))=0\text{ on }y>0\\
	&(b) L_2(\nabla U)+Q_2(\nabla U)+C_2(\nabla U)=\eps^2\begin{bmatrix}
			f\\
			g
		\end{bmatrix}\text{ on }y=0\\
		&U=0 \text{ in }t\leq 0,
\end{split}
	\end{equation}
	 where $L=(L_1,L_2)$, $Q=(Q_1,Q_2)$, and $C=(C_1,C_2)$ are respectively linear, quadratic, and cubic functions of $\nabla U$.  
	 
	 \subsection{Choice of $c$}   Letting $(\tau,\xi,\omega)$ denote variables dual to $(t,x,y)$, we can write the principal symbol of the operator obtained by linearizing the left side of \eqref{eq:UInitIntSys} at $\nabla U=0$ as 
\begin{align}\label{oo1}
\begin{split}
&L(\tau,\xi,\omega)=\begin{pmatrix}\tau^2-(r-1)\xi^2-|\xi,\omega|^2 &-(r-1)\xi\omega\\-(r-1)\xi\eta&\tau^2-(r-1)\omega^2-|\xi,\omega|^2\end{pmatrix}
%&M_{\omega_j}(\sigma,\xi_1,\xi_2)=\begin{pmatrix}-2c\sigma-2r\xi_1-2\omega_j\xi_2&-(r-1)\omega_j\xi_1-(r-1)\xi_2\\-(r-1)\omega_j\xi_1-(r-1)\xi_2&-2c\sigma-2\xi_1-2r\omega_j\xi_2\end{pmatrix}
\end{split}
\end{align}
The constant $r>1$  is the ratio of the squares of pressure $c_d$ and shear $c_s$ velocities.\footnote{We have $c_s^2=\mu$ and  $c_d^2=(\lambda+2\mu)$, where  $\lambda$, $\mu$ are the Lam\'e constants.  The form \eqref{oo1} is obtained by taking units of time so that $c_s=1$.  Observe $r=c_d^2/c_s^2>1$ since $\lambda+\mu>0$.}   
The matrix $L(\beta,\omega)$ has characteristic roots $\omega_j$ and vectors $r_j$ satisfying 
\begin{align}\label{oo2}
\det L(\beta,\omega_j)=0\text{ and }L(\beta,\omega_j)r_j=0, j=1,\dots 4.
\end{align}
\begin{align}\label{ray1}
\omega_j^2=c^2-1 \text{ for }j=1,3; \; \quad \omega_j^2= \frac{c^2}{r}-1\text{ for }j=2,4.  
\end{align}
The boundary frequency $\beta$ is said to lie in the \emph{elliptic region} when  $c^2-1$ and $\frac{c^2}{r}-1$ are negative.     The $\omega_j$ are then purely imaginary, and we take $\omega_1$, $\omega_2$ to have positive imaginary part.       Thus, we have   $\omega_3=\overline{\omega_1}$,  $\omega_4=\overline{\omega_2}$ and we take 
\begin{align}\label{oo3}
r_1=\begin{pmatrix}-\omega_1\\1\end{pmatrix},\;r_2=\begin{pmatrix}1\\ \omega_2\end{pmatrix},\;r_3=\overline{r}_1, r_4=\overline{r}_2.
\end{align}

 If we define $q=q(c)>0$ by 
\begin{align}\label{ray2}
q^2=-\omega_1\omega_2,
\end{align}
then the condition for $\beta=(-c,1)$ to be a \emph{Rayleigh frequency} is that
\begin{align}\label{ray}
(2-c^2)^2=4q^2(c),    \text{ or equivalently } 2-c^2=2q.
\end{align}
This equation is equivalent to the statement that 
\begin{align}\label{oo4}
\det \mathcal{B}_{Lop}=0,
\end{align}
where $\mathcal{B}_{Lop}$   is the Lopatinski matrix  \eqref{f5} derived in section \ref{sec:GenObs}. 
For the existence of $0<c<1$ satisfying \eqref{ray} we refer for example to \cite{T2}, and we fix $\beta=(-c,1)$ for this choice of $c$.   We will see below that it is the vanishing of the determinant \eqref{oo4} that gives rise to Rayleigh waves.

In the remainder of this part we will construct approximate solutions  of \eqref{eq:UInitIntSys} of the form
\begin{equation}\label{d1}
		U^\varepsilon_a(t,x,y)=\sum_{n=2}^{N}\varepsilon^k U_k(t,x,y,\theta,Y)|_{\theta=\frac{x-ct}{\varepsilon},Y=\frac{y}{\varepsilon}},
	\end{equation}
	where the profiles $U_k$ belong to the space $S$ of Definition \ref{d1a}. 
The function  $U_{a}^\varepsilon(t,x,y)$ is constructed (Theorem \ref{theo:Error}) to satisfy
			\begin{align}\label{ooh}
			\begin{split}
			&\partial_t^2U_{a}^\varepsilon+\nabla\cdot(L(\nabla U_{a}^\varepsilon)+Q(\nabla U_{a}^\varepsilon)+C(\nabla U_{a}^\varepsilon))=\varepsilon^{N-1} E'_N(t,x,y,\frac{x-ct}{\varepsilon},\frac{y}{\varepsilon})\text{ on }y>0\\
	&L_2(\nabla U_{a}^\varepsilon)+Q_2(\nabla U_{a}^\varepsilon)+C_2(\nabla U_{a}^\varepsilon)-\eps^2\begin{bmatrix}
				f\\
				g
			\end{bmatrix}=\varepsilon^Ne_N(t,x,0,\frac{x-ct}{\varepsilon},0)\text{ on }y=0,
		\end{split}
		\end{align}
		where  $E'_N, e_N$ lie in the space $S^e$ of Definition \ref{def:MixedTerm}. 

In Theorem \ref{maincor} we combine the results of theorems \ref{main} and  \ref{theo:Error} to state our main result for nonlinear elasticity,  which makes precise the sense in which  the approximate solution
is close to exact solution.

In Chapter 2 of \cite{Mar2} A. Marcou constructed $U_2$ and part of $U_3$ for a simplified SVK-type model:  there were no cubic terms $C(\nabla U)$ and a number of the quadratic terms in $Q(\nabla U)$ were dropped.  It will be clear from the exposition below that her analysis of that model was helpful to us in constructing the approximate solution.

	\section{Spaces of profiles}\label{sec:Hyp}
	In this section we define spaces that contain all the kinds  of functions that will arise  in the construction of profiles.  The first two definitions concern functions defined on the 
	``interior", that is,  the set where $y>0$, $Y>0$. 
	
	\begin{defn}\label{d1a}
	Let the space $S$ be given by $S=\underline{S}\oplus S^*$ where $\underline{S}=H^\infty([0,T]\times\mathbb{R}_x\times[0,\infty)_y)$ is the usual Sobolev space and $S^*$ consists of functions $u^*(t,x,\theta,Y)\in H^\infty([0,T]\times\mathbb{R}_x\times\mathbb{T}_\theta\times[0,\infty)_Y)$ satisfying  the additional restriction that:
	\begin{equation}
	|\partial^\alpha_{t,x,\theta,Y} u^*(t,x,\theta,Y)|_{L^2(\mathbb{R}_x)}\leq C_{\alpha}e^{-\delta Y}
	\end{equation}
	where $\alpha$ is a multi-index, and $C_\alpha$ and $\delta$ are positive constants.    Note that $\delta$ is independent of $\alpha$ but not independent of $u^*$. 
	\end{defn}
	
This space was used, for example, by Marcou \cite{Marcou} in her study of first-order hyperbolic conservation laws, and also in Chapter 2 of \cite{Mar2} in her study of an SVK-type model. 

%and also in the second chapter of her thesis ..., where she constructed the leading two profiles for a simplified version of the Saint Venant-Kirchhoff model.  
	
	The intervals $y\in[0,\infty)$ in $\underline{S}$ and $Y\in[0,\infty)$ in $S^*$ contain different variables. A given element $u\in S$ can be written as $u(t,x,y,\theta,Y)=\underline{u}(t,x,y)+u^*(t,x,\theta,Y)$, where $\underline{u}\in\underline{S}$ and $u^*\in S^*$. Moreover, since each $u\in S$ is periodic with respect to $\theta$, we can further decompose $u$ as 
	\begin{align}
	\begin{split}
	&u(t,x,y,\theta,Y)=\underline{u}(t,x,y)+u^{0*}(t,x,Y)+\sum_{n\not=0}u^{n}(t,x,Y)e^{in\theta}=
	u^0(t,x,y,Y)+u^{osc}(t,x,\theta,Y),
       \end{split}
       \end{align}
	where 
	\begin{align}
	u^0(t,x,y,Y):=\underline{u}(t,x,y)+u^{0*}(t,x,Y)\text{ and }u^{osc}(t,x,\theta,Y):=\sum_{n\not=0}u^{n}(t,x,Y)e^{in\theta}.
	\end{align}

	%Later on in section \ref{sec:Order} we will decompose the sum further.    
	
	The spaces $\underline{S}$ and $S^*$ are each closed under multiplication, but   $S$ is  not closed under multiplication, since it does not contain products $\underline u v^*$, where $\underline u\in \underline S$, $v^*\in S^*$.   This forces us to introduce the extended space $S^e$.

	\begin{defn}[The space $S^e$]\label{def:MixedTerm}
		1) A function $u(t,x,y,\theta,Y)$ is called \emph{mixed}  if it is a finite linear combination of functions of the form $\underline{a}(t,x,y)b^*(t,x,\theta,Y)$ where $\underline{a}\in\underline{S}$ and $b^*\in S^*$.   Let $S^m$ be the space of all such linear combinations.  
		
		2)   The extended space $S^e:=S\oplus S^m$.

		\end{defn}
		
		A function $u(t,x,y,\theta,Y)\in S^e$ is periodic in $\theta$ so we can write
		\begin{align}\label{uinse}
		u=u^0(t,x,y,Y)+\sum_{n\neq 0}u^n(t,x,y,Y)e^{in\theta}= u^0(t,x,y,Y)+u^{osc}(t,x,y,\theta,Y)
		\end{align}
		where 
		\begin{align}\label{uinse2}
		\begin{split}
		&u^0(t,x,y,Y)=\underline u(t,x,y)+u^{0*}(t,x,Y)+u^{0,m}(t,x,y,Y)\text{ with }u^{0*}\in S^*, \;u^{0,m}\in S^m\\
		&u^{osc}(t,x,y,\theta,Y)=u^{osc,*}(t,x,\theta,Y)+u^{osc,m}(t,x,y,\theta,Y)\text{ with }u^{osc,*}\in S^*, \;u^{osc,m}\in S^m.
		\end{split}
		\end{align}

		On the ``boundary", that is, the set where $y=0$, $Y=0$ we have
		\begin{defn}
		Let $S^b=H^\infty([0,T]\times \mathbb{R}_x\times \mathbb{T}_\theta)$. 
	\end{defn}
		
	Functions in $S^b$ can be written $f(t,x,\theta)=\underline f(t,x)+ f^{osc}(t,x,\theta)$, where the Fourier mean of $f^{osc}$ is zero.  We note also that 
	\begin{align}
	u\in S^e\Rightarrow u|_{y=0,Y=0}\in S^b.
	\end{align}

		The following proposition, whose proof is immediate from \eqref{uinse} and the definitions, records several of the properties of $S^e$. 
		
		\begin{prop}\label{d2a}
		For elements $u\in S^e$ we refer here to the pieces defined in \eqref{uinse} and \eqref{uinse2}. 
		
		1) The space $S^e$ is closed under multiplication.
		
		2) For $u(t,x,y,\theta,Y)\in S^e$ we have $\lim_{Y\to\infty} u=\underline u(t,x,Y)$.
		
		3) Any piece of $u$ whose superscripts include one or more of $*$, $osc$, or $m$ is exponentially decaying in $Y$. 
		\end{prop}
		
		In the nonlinearity of the Saint Venant-Kirchhoff model, we have  products of elements of $S$. In order to deal with the fact that such products lie in $S^e$ but not necessarily in $S$,  it will be useful to Taylor expand functions $u^m\in S^m$ in the $y$ variable as follows.   
		\begin{align}\label{d2}
		\begin{split}
		&\qquad\qquad u^m(t,x,y,\theta,Y)=\\
		&u^m(t,x,0,\theta,Y)+\eps \frac{y}{\eps} \partial_yu^m(t,x,0,\theta,Y)+\dots+\eps^k\frac{y^k}{\eps^k}\frac{1}{k!}\partial_y^ku^m(t,x,0,\theta,Y)+\eps^{k+1}\frac{y^{k+1}}{\eps^{k+1}}r_{k+1}(t,x,y,\theta,Y),
		\end{split}
		\end{align}
		where $r_{k+1}\in S^m$. Next define a modification of $u^m$, $u^{m,mod}_{k+1}\in S^m$, by 
		\begin{align}\label{d2aa}
		\begin{split}
		&\qquad\qquad u^{m,mod}_{k+1}(t,x,y,\theta,Y)=\\
		&u^m(t,x,0,\theta,Y)+\eps Y\partial_yu^m(t,x,0,\theta,Y)+\dots+\eps^kY^k\frac{1}{k!}\partial_y^ku^m(t,x,0,\theta,Y)+\eps^{k+1}R_{k+1}(t,x,y,\theta,Y),
		\end{split}
		\end{align}
		where $R_{k+1}=Y^{k+1}r_{k+1}(t,x,y,\theta,Y)\in S^m$. This turns out to be useful because of the following two properties:
		\begin{align}\label{d3}
		\begin{split}
		&u^{m,mod}_{k+1}-\eps^{k+1}R_{k+1}\in S^*\\
		&u^m(t,x,y,\theta,Y)|_{Y=\frac{y}{\eps}}=u^{m,mod}_{k+1}(t,x,y,\theta,Y))|_{Y=\frac{y}{\eps}}.
		\end{split}
		\end{align}
		Roughly speaking, the properties \eqref{d3} will allow us to replace elements of $S^m$ by elements of $S^*$ at the price of an error $\eps^{k+1}R_{k+1}\in \eps^{k+1}S^m$ which is 
		harmless for the purposes of constructing approximate solutions.
	\begin{rem}
		%1) In order to close $S$ under products we could redefine $S^*$, in a slight abuse of notation, to be $H^\infty(t,x,y,\theta,Y)$ with the same exponential decay in $Y$. This approach, however, introduces new difficulties in determining the profiles. For example, two portions, $U_{k,\alpha}$ and $U_{k,h}$ discussed in Section \ref{sec:Order}, are determined by their traces on $y=Y=0$. Therefore, extending $S^*$ to include dependence on $y$ makes defining $U_{k,\alpha}$ and $U_{k,h}$ on the interior somewhat ambiguous.\\
		The piece  $u^0(t,x,y,Y)$ is the ``Fourier mean" of $u\in S^e$.  Since it is common  in geometric optics to use  $\underline{u}$ to denote  the Fourier mean, in order to avoid confusion we will not refer to either $\underline{u}$ or $u^0$ as the ``mean" of $u$.
	\end{rem}

	\section{Cascade of profile equations}\label{sec:Cascade}
	
	 We look for an approximate solution of \eqref{eq:UInitIntSys} given by the following ansatz:
	\begin{equation}\label{eq:ansatz}
		U^\varepsilon_a(t,x,y)=\sum_{n=2}^{N}\varepsilon^n U_n(t,x,y,\theta,Y|_{\theta=\frac{x-ct}{\varepsilon},Y=\frac{y}{\varepsilon}},
	\end{equation}
	%where each $U_n$ is in the space $S$, and $N$ is some positive integer to be chosen later. 
	%As a small remark, notice that the lowest order term in this expansion is $O(\varepsilon^2)$. Suppose we have $U^\varepsilon(t,x,y,\frac{x-ct}{\varepsilon},\frac{y}{\varepsilon})$, and we apply $\partial_x$. From the chain rule, this results in $\partial_xU^\varepsilon=\partial_xU^\varepsilon+\frac{1}{\varepsilon}\partial_\theta U^\varepsilon$. Applying similar logic to the other derivatives, suggests we make the following substitutions for the derivatives in \eqref{eq:UInitIntSys}:
	%\begin{equation*}
	%	\partial_x\rightarrow\partial_x+\frac{1}{\varepsilon}\partial_\theta \quad \partial_y\rightarrow\partial_y+\frac{1}{\varepsilon}\partial_Y \quad \partial_t\rightarrow\partial_t-\frac{c}{\varepsilon}\partial_\theta
	%\end{equation*}
%\\
Plugging in $U_a^\varepsilon$ into \eqref{eq:UInitIntSys} and grouping terms according to powers of $\eps$ gives:
		\begin{equation}\label{d4}
		\begin{split}
			&\partial_t^2 U_{a}^\varepsilon+\nabla\cdot(L(\nabla U_{a}^\varepsilon)+Q(\nabla U_{a}^\varepsilon)+C(\nabla U_{a}^\varepsilon))=\\
			& \quad \left[\sum_{k=2}^{N} \varepsilon^{k-2}\left(L_{ff}(U_k)-\begin{pmatrix}
				H_{k-1}\\
				K_{k-1}
			\end{pmatrix}\right)+\varepsilon^{N-1}E_N^\eps\right]|_{\theta=\frac{x-ct}{\varepsilon},Y=\frac{y}{\varepsilon}}\text{ on }y>0
		\end{split}
		\end{equation}

\begin{equation}\label{d5}
		\begin{split}
		&L_2(\nabla U_{a}^\varepsilon)+Q_2(\nabla U_{a}^\varepsilon)+C_2(\nabla U_{a}^\varepsilon)-\eps^2\begin{pmatrix}f\\g\end{pmatrix}|_{\theta=\frac{x-ct}{\eps}}=\\
		&\quad  \left[\sum_{k=2}^{N} \varepsilon^{k-1}\left(l_{f}(U_k)-\begin{pmatrix}
				h_{k-1}\\
				k_{k-1}
			\end{pmatrix}\right)+\varepsilon^{N}e^\eps_N\right]|_{\theta=\frac{x-ct}{\varepsilon},Y=\frac{y}{\varepsilon}}\text{ on }y=0.
		\end{split}
		\end{equation}
The operators $L_{ff}$ and $l_f$ are defined below in \eqref{eq:linearoperators}.   The functions $H_{k-1}$, $K_{k-1}$, $h_{k-1}$, $k_{k-1}$ as well as $E^\eps_N$ and $e^\eps_N$  are determined by \eqref{d4} and \eqref{d5} as nonlinear functions of the profiles $U_2, \dots, U_{k-1}$ and belong to $S^e$.   Formulas for $H_{k-1},\dots, k_{k-1}$ that are as explicit as we need for the profile construction are given below.   
		
	Clearly, in order to obtain high order approximate solutions we would like to choose the $U_k\in S$ so that the following equations hold:\footnote{It will turn out that these equations can be solved with $U_k\in S$ for $k=2,\dots,5$.  For $k\geq 6$ we will need to modify the right side of the interior equation.   See Remark \ref{nec}.}
	\begin{equation}\label{eq:CascadeInterior}
		\begin{split}
		&L_{ff}(U_k)=\begin{pmatrix}
		H_{k-1}\\
		K_{k-1}
		\end{pmatrix}\text{ on }y>0, Y>0\\
		&l_f(U_k)=\begin{pmatrix}
		h_{k-1}\\
		k_{k-1}
		\end{pmatrix}\text{ on }y=0,Y=0.
	\end{split}
	\end{equation}

	To specify the objects appearing in \eqref{d4} and \eqref{d5} we first define linear operators involving derivatives with respect to fast variables $\theta$, $Y$ and slow variables $t$, $x$, $y$.     The constant $r>1$ in the formulas below is same as in \eqref{oo1}.
	%\footnote{We have $c_s^2=\mu$ and  $c_d^2=(\lambda+2\mu)$, where  $\lambda$, $\mu$ are the Lam\'e constants.  The form \eqref{eq:linearoperators} is obtained by taking units of time so that $c_s=1$.  Observe $r=c_d^2/c_s^2>1$ since $\lambda+\mu>0$.}   

	\begin{align}\label{eq:linearoperators}
		\begin{split}
		& L_{ff}:=\begin{pmatrix} %FAST-FAST INTERIOR
			c^2-r& 0 \\
			0 & c^2-1
		\end{pmatrix}\partial_{\theta\theta}-\begin{pmatrix}
			0 & r-1 \\
			r-1 & 0
		\end{pmatrix}\partial_{\theta Y}-\begin{pmatrix}
			1 & 0\\
			0 & r
		\end{pmatrix}\partial_{YY}\\
		& L_{fs}:=-2c\partial_{t\theta}-\begin{pmatrix} %FAST SLOW INTERIOR
			2r & 0\\
			0 & 2 \\
		\end{pmatrix}\partial_{x\theta}-\begin{pmatrix}
			0 & r-1 \\
			r-1 & 0
		\end{pmatrix}[\partial_{xY}+\partial_{y\theta}]-\begin{pmatrix}
			2 & 0\\
			0 & 2r
		\end{pmatrix}\partial_{yY}\\
		& L_{ss}:=\partial_{tt}-\begin{pmatrix} %SLOW SLOW INTERIOR
			r & 0\\
			0 & 1
		\end{pmatrix}\partial_{xx}-\begin{pmatrix}
			0 & r-1\\
			r-1 & 0
		\end{pmatrix}\partial_{xy}-\begin{pmatrix}
			1 & 0\\
			0 & r
		\end{pmatrix}\partial_{yy}
		\end{split}
		\end{align}
		\begin{align}
		\begin{split}
	       & l_f:=\begin{pmatrix} %FAST BOUNDARY
			0 & 1 \\
			r-2 & 0
		\end{pmatrix}\partial_\theta+\begin{pmatrix}
			1 & 0 \\
			0 & r
		\end{pmatrix}\partial_Y\\
		& l_s:=\begin{pmatrix} %SLOW BOUNDARY
			0 & 1\\
			r-2 & 0
		\end{pmatrix}\partial_x+\begin{pmatrix}
			1 & 0\\
			0 & r
		\end{pmatrix}\partial_y.
	\end{split}
	\end{align}
	Next we give formulas for the terms $(H_{k-1}$, $K_{k-1})$ and $(h_{k-1}$, $k_{k-1})$ in \eqref{eq:CascadeInterior}.   Profiles $U_j$ with $j<2$ are defined to be zero.  
	The various operators $A_{\dots}$, $B_{\dots}$, $Q_j(\dots)$, $C_j(\dots)$ that appear are defined further below.
		\begin{equation}\label{eq:CascadeIntRHS}
		\begin{split}
			&\begin{pmatrix}
				H_{k-1}\\
				K_{k-1}
			\end{pmatrix}=
		 -L_{fs}(U_{k-1})-L_{ss}(U_{k-2})+\\
		&\sum_{i+j=k-2}A_{sss}(U_i,U_j)+\sum_{i+j=k-1}A_{fss}(U_i,U_j)+\sum_{i+j=k}A_{ffs}(U_i,U_j)\\
		& +\sum_{i+j=k+1}A_{fff}(U_i,U_j)+\sum_{l+m+n=k-2}B_{ssss}(U_l,U_m,U_n)+\sum_{l+m+n=k-1}B_{fsss}(U_l,U_m,U_n)\\ &+\sum_{l+m+n=k}B_{ffss}(U_l,U_m,U_n)+ \sum_{l+m+n=k+1}B_{fffs}(U_l,U_m,U_n) +\sum_{l+m+n=k+2}B_{ffff}(U_l,U_m,U_n)
		\end{split}
	\end{equation}
	and $(h_{k-1},k_{k-1})$, as in \eqref{eq:CascadeInterior}, is given by the expression:
	\begin{equation}\label{eq:CascadeBdyRHS}
		\begin{split}
			\begin{pmatrix}
				h_{k-1}\\
				k_{k-1}
		\end{pmatrix}
		=& -l_s(U_{k-1})-\sum_{i+j=k+1}Q_2(\partial_{\f};\partial_{\f})(U_i,U_j)-\sum_{i+j=k}[Q_2(\partial_{\f};\partial_{\s})(U_i,U_j)+Q_2(\partial_{\s};\partial_{\f})(U_i,U_j)]\\
		& -\sum_{i+j=k-1}Q_2(\partial_{\s};\partial_{\s})(U_i,U_j)-\sum_{l+m+n=k+2}C_2(\partial_{\f};\partial_{\f};\partial_{\f})(U_l,U_m,U_n)-\\
		& \sum_{l+m+n=k+1}[C_2(\partial_{\f},\partial_{\f};\partial_{\s})(U_l,U_m,U_n)+C_2(\partial_{\f};\partial_{\s};\partial_{\f})(U_l,U_m,U_n)+\\ &+C_2(\partial_{\s};\partial_{\f};\partial_{\f})(U_l,U_m,U_n)]\\
		&-\sum_{l+m+n=k}[C_2(\partial_{\f},\partial_{\s};\partial_{\s})(U_l,U_m,U_n)+C_2(\partial_{\s};\partial_{\f};\partial_{\s})(U_l,U_m,U_n)+\\ &+C_2(\partial_{\s};\partial_{\s};\partial_{\f})(U_l,U_m,U_n)]\\
		&-\sum_{l+m+n=k-1}C_2(\partial_{\s};\partial_{\s};\partial_{\s})(U_l,U_m,U_n)+\begin{pmatrix}f_{k-1}\\g_{k-1}\end{pmatrix},
		\end{split}
	\end{equation}
		where $\begin{pmatrix}f_{k-1}\\g_{k-1}\end{pmatrix}=\begin{pmatrix}f\\g\end{pmatrix}$ as in \eqref{eq:UInitIntSys} if $k=3$ and is zero otherwise.
	%The $Q_j$'s are quadratic functions of two profiles and the $C_j$'s are cubic functions of three profiles that are derived in the following manner. 

	The notation here is an extension of that used by Marcou in Chapter 2 of \cite{Mar2}.
	Recall that in \eqref{eq:UInitIntSys} we had $\nabla\cdot Q(\nabla U)=\partial_x Q_1(\nabla U)+\partial_y Q_2(\nabla U)$ with $Q_1$ the first column of $Q$ and $Q_2$ the second.    The $Q_j$ are quadratic in $\nabla U$, so with some abuse we can write
	\begin{align}
	Q_j(\nabla U)=Q_j(\partial_{\s};\partial_{\s})(U,U),
	\end{align}
	where the first pair of derivatives act on the first profile in the argument, and the second  pair acts on the second argument. 
Thus,  the expression on the right denotes a column vector whose entries are linear combinations of terms of the form $\partial_x u\partial_y v$, $\partial_yu\partial_y u$, etc.,
where$ \;U=\begin{pmatrix}u\\v\end{pmatrix}$.\footnote{The exact coefficients in the linear combinations could be determined, but they are not needed for the profile construction.}

	% One useful property of $Q$ is that it is bilinear, and so we can write $Q_j(\nabla U)=Q_j(\nabla U,\nabla U)=Q_j(\partial_{\s};\partial_{\s})(U,U)$ where in the third statement we placed the derivatives into a separate set of arguments. The first pair of derivatives act on the first profile in the argument, and the second derivative pair acts on the second argument. This is to make it easy to swap the slow derivatives $\partial_{\s}$ for the fast derivatives $\partial_{\f}$. 

	 Similarly, each entry of $Q_2(\partial_{\s};\partial_{\f})(U_i,U_j)$ is a linear combination of terms like $\partial_y u_i \partial_\theta v_j$, $\partial_x v_i \partial_Y u_j$, $\partial_y v_i\partial_\theta v_j$ for $U_i=\begin{pmatrix}
		u_i\\
		v_i
		\end{pmatrix}$.
		The trilinear functions $C_1,C_2$ are defined in an analogous manner.   Thus,  each entry of $C_2(\partial_{\f};\partial_{\s};\partial_{\f})(U_l,U_m,U_n)$ is a linear combination of terms of the form $\partial_\theta u_l\partial_yv_m\partial_Yu_n$, $\partial_Yv_l\partial_xv_m\partial_\theta v_n$, etc..

		The $A$ and $B$ functions are related to the $Q$ and $C$ functions by the following relations:\footnote{Here and in \eqref{eq:BinTermsofC} we suppress some of the $(U_i,U_j)$ or $(U_l,U_m,U_n)$ arguments.}
	\begin{align}\label{eq:AinTermsofQ}
		\begin{split}
		&A_{fff}(U_i,U_j):=\partial_{\theta}[Q_1(\partial_{\f};\partial_{\f})(U_i,U_j)]+\partial_Y[Q_2(\partial_{\f};\partial_{\f})(U_i,U_j)]\\
		&A_{ffs}:=\partial_{\theta}Q_1(\partial_{\f};\partial_{\s})+\partial_YQ_2(\partial_{\f},\partial_{\s})+ \partial_{\theta}Q_1(\partial_{\s};\partial_{\f})+\partial_YQ_2(\partial_{\s},\partial_{\f})+\\ &\qquad \partial_{x}Q_1(\partial_{\f};\partial_{\f})+\partial_yQ_2(\partial_{\f},\partial_{\f})
		\end{split}
		\end{align}
		\begin{align}
		\begin{split}
		&A_{fss}:=\partial_{\theta}Q_1(\partial_{\s};\partial_{\s})+\partial_YQ_2(\partial_{\s},\partial_{\s})+ \partial_{x}Q_1(\partial_{\s};\partial_{\f})+\partial_{y}Q_2(\partial_{\s},\partial_{\f})+\\ 
		&\qquad\partial_{x}Q_1(\partial_{\f};\partial_{\s})+\partial_yQ_2(\partial_{\f},\partial_{\s})\\
		&A_{sss}(U_i,U_j):=\partial_x[Q_1(\partial_{\s};\partial_{\s})(U_i,U_j)]+\partial_y[Q_2(\partial_{\s};\partial_{\s})(U_i,U_j)]
	\end{split}
	\end{align}
	For example, each entry of the $\partial_{x}Q_1(\partial_{\s};\partial_{\f})(U_i,U_j)$ term of $A_{fss}(U_i,U_j)$ is a linear combination of terms of  the form $\partial_x(\partial_y u_i\partial_\theta u_j)$, 
	$\partial_x(\partial_yv_i\partial_Yu_j)$, etc..

	\begin{align}\label{eq:BinTermsofC}
		 \begin{split}
		 &B_{ffff}:=\partial_{\theta}C_1(\partial_{\f};\partial_{\f};\partial_{\f})+\partial_YC_2(\partial_{\f};\partial_{\f};\partial_{\f})\\
		&B_{fffs}:=\partial_{\theta}C_1(\partial_{\f};\partial_{\f};\partial_{\s})+\partial_YC_2(\partial_{\f};\partial_{\f};\partial_{\s})
			 +\partial_{\theta}C_1(\partial_{\f};\partial_{\s};\partial_{\f})+\partial_YC_2(\partial_{\f};\partial_{\s};\partial_{\f})\\
			&\quad +\partial_{\theta}C_1(\partial_{\s};\partial_{\f};\partial_{\f})+\partial_YC_2(\partial_{\s};\partial_{\f};\partial_{\f})
			 +\partial_xC_1(\partial_{\f};\partial_{\f};\partial_{\f})+\partial_yC_2(\partial_{\f};\partial_{\f};\partial_{\f})
		\end{split}
		\end{align}
		\begin{align}
		\begin{split}
			B_{ffss}&:=\partial_{\theta}C_1(\partial_{\f};\partial_{\s};\partial_{\s})+\partial_YC_2(\partial_{\f};\partial_{\s};\partial_{\s})+ \partial_{\theta}C_1(\partial_{\s};\partial_{\f};\partial_{\s})+\partial_YC_2(\partial_{\s};\partial_{\f};\partial_{\s})\\
			&+ \partial_{x}C_1(\partial_{\f};\partial_{\f};\partial_{\s})+\partial_yC_2(\partial_{\f};\partial_{\f};\partial_{\s})
			+\partial_{x}C_1(\partial_{\f};\partial_{\s};\partial_{\f})+\partial_yC_2(\partial_{\f};\partial_{\s};\partial_{\f})\\
			& + \partial_{x}C_1(\partial_{\s};\partial_{\f};\partial_{\f})+\partial_yC_2(\partial_{\s};\partial_{\f};\partial_{\f})
			\partial_{\theta}C_1(\partial_{\s};\partial_{\s};\partial_{\f})+\partial_YC_2(\partial_{\s};\partial_{\s};\partial_{\f})
		\end{split}
		\end{align}
		\begin{align}
		\begin{split}
			&B_{fsss}:=\partial_{\theta}C_1(\partial_{\s};\partial_{\s};\partial_{\s})+\partial_YC_2(\partial_{\s};\partial_{\s};\partial_{\s})
			+\partial_{x}C_1(\partial_{\f};\partial_{\s};\partial_{\s})+\partial_yC_2(\partial_{\f};\partial_{\s};\partial_{\s})\\
			&\quad + \partial_{x}C_1(\partial_{\s};\partial_{\f};\partial_{\s})+\partial_yC_2(;\partial_{\s};\partial_{\f};\partial_{\s})
			+\partial_{x}C_1(\partial_{\s};\partial_{\s};\partial_{\f})+\partial_yC_2(\partial_{\s};\partial_{\s};\partial_{\f})
		\end{split}
		\end{align}
		\begin{align}
		B_{ssss}&:=\partial_{x}C_1(\partial_{\s};\partial_{\s};\partial_{\s})+\partial_yC_2(\partial_{\s};\partial_{\s};\partial_{\s})
	\end{align}
	%where $Q_1,Q_2$ are quadratic functions derived from the columns of $Q(\nabla U)$ and $C_1,C_2$ derived from the columns of $C(\nabla U)$ appearing in \eqref{eq:UInitIntSys}. Since we modified the derivatives, the variables indicated in the first set of parentheses determine the derivatives appearing in the expansion of the $Q_j$'s and $C_j$'s. For example, $Q_1(\partial_{y,x};\partial_{Y,\theta})(U_i,U_j)$ has terms of the form $\partial_yu_i\partial_\theta v_j$. For $j=1,2$, $Q_j$ are functions of two profiles and three profiles for $C_j$. As an example, $Q_2(\partial_{y,x};\partial_{y,x})$ is a lengthy sum of terms of the form $\partial_yu_i\partial_xv_j$ for $U_i=(u_i,v_i)$. From this, we can observe that $Q_j$ and the $A$'s are bilinear functions of two profiles $U_i,U_j$.
For example, each entry of the  $\partial_{\theta}C_1(\partial_{\s};\partial_{\f};\partial_{\s})(U_l,U_m,U_n)$ term of $B_{ffss}(U_l,U_m,U_n)$ is a linear combination of terms of the form $\partial_\theta(\partial_x u_l\partial_Y u_m\partial_x v_n)$, etc..

	Notice that for a given natural number $k$, if  $A_{fff}(U_i,U_j)$ and $B_{fff}(U_l,U_m,U_n)$ are coefficients of $\varepsilon^k$, the following inequalities hold:
	\begin{equation*}
		A_{fff}(U_i,U_j)\implies \  i+j=k+3\implies \ 2\leq i,j\leq k+1
	\end{equation*}
	\begin{equation*}
		B_{ffff}(U_l,U_m,U_n)\implies \ l+m+n=k+4 \implies \ 2\leq l,m,n \leq k.
	\end{equation*}
	This implies that $\begin{pmatrix}H_{k-1}\\K_{k-1}\end{pmatrix}$, which is part of  the coefficient of $\varepsilon^{k-2}$ in \eqref{d4},   depends only on the profiles $U_2,...,U_{k-1}$, . Moreover, from these bounds we observe that cubic terms appear first in the expressions for $H_3,K_3$.
	
	\begin{rem}\label{nec}
	A necessary condition for obtaining a solution $U_k\in S$ to the equation $L_{ff}(U_k)=\begin{pmatrix}H_{k-1}\\K_{k-1}\end{pmatrix}$ is that the  right side belong to $S^*$ .    Inspection of \eqref{eq:CascadeIntRHS} already shows that $(H_2,K_2)\in S^*$.  It turns out that $\underline {U_2}(t,x,y)=0$, which will imply that  $(H_3,K_3)\in S^*$ and $(H_4,K_4)\in S$, but we will find that $(H_j,K_j)\in S^m$  not $S$ for $j\geq 5$.\footnote{More precisely, there is no reason for these function to lie in $S$; barring unlikely cancellations, they lie in $S^m$ not  $S$.}  Thus, in section \ref{sec:AnalysisU2} we will modify the $(H_j,K_j)$, $j\geq 5$,  to elements $(H_j',K_j')\in S$, using \eqref{d2}-\eqref{d3} .  By a careful choice of $\underline{U_{j-1}}$ we will eventually arrange $(H_j',K_j')\in S^*$.

Observe that there is no constraint of the form $h_{k-1},k_{k-1}\in S^*$. This is because the boundary condition in \eqref{eq:CascadeInterior} involves only  the traces of these functions on the boundary $y=Y=0$. 
%TODO:mention Marcou.

% In principle there is no reason for $H_{k- 1}\in S^*$, however, this is true by our choice of $\underline{U}_{k-2}$ and some modifications that are discussed in more details later on. The key difference is that in Marcou, the terms are only quadratic in the previous profiles, and here there are both quadratic and cubic terms and have significantly more terms than the quadratic terms appearing in her work. 
	\end{rem}

	\section{Solvability conditions for  $L_{ff}(U)=F$, $l_f(U)=G$}\label{sec:GenObs}
	Motivated by the form of the profile equations \eqref{eq:CascadeInterior}, we consider in this section the general question of finding real-valued solutions $U=\begin{pmatrix}u\\v\end{pmatrix}\in S$ of systems  of the form
\begin{equation}\label{eq:GenInt}
		\begin{split}
		&L_{ff}(U)=F \text{ on }y,Y>0\\
	&l_f(U)=G\text{ on }y,Y=0,
	\end{split}
	\end{equation}
	 where $F\in S^e$, $G\in S^b$.

	 A necessary condition for the existence of a real-valued solution $U\in S$ is   that $F$ and $G$ be real-valued with $F\in S^*$, $G\in S^b$,  and so we assume that.   We will see that the existence of a solution depends on certain \emph{additional} solvability conditions being satisfied.   The operators $L_{ff}$ and $l_f$ both annihilate elements of $\underline S$, so we will look for $U=U^{0*}+U^{osc}\in S^*$.    Writing $F=F^{0*}+F^{osc}$ and $G=\underline{G}+G^{osc}$, we look for $U$ of the form
	 \begin{align}\label{f1}
	 U=U^{0*}+U^{osc}=U^{0*}+U_p+U_h,
	 \end{align}
	 where we try to make these pieces satisfy
	 \begin{align}\label{f2}
	 \begin{split}
	&(a)\;L_{ff}U^{0*}=F^{0*},\;\;l_f(U^{0*})=\underline{G}\\
	 &(b)\;L_{ff}U_p=F^{osc}\\
	 &(c)\;L_{ff}(U_h)=0, l_f(U_h)=G^{osc}-l_f(U_p).
	 \end{split}
	 \end{align}
	 Here and below equations with $L_{ff}$ or $l_f$ on the left hold respectively on $y,Y>0$ or $y,Y=0$. 
	 
	 We will use Fourier series to analyze these equations.   For $F^{osc}$ and  $F^{0*}$ we write 
	 \begin{align}
	 \begin{split}
	 F^n(t,x,Y)=\begin{pmatrix}f_1^n\\f_2^n\end{pmatrix}, n\neq 0,\qquad\quad  F^{0*}(t,x,Y)=\begin{pmatrix}f^{0*}_1\\f^{0*}_2\end{pmatrix}.
	 \end{split}
\end{align}
		Consider first  \eqref{f2}(a).   Since $U^{0*}=U^{0*}(t,x,Y)$ is independent of $\theta$,  the interior equation simplifies to:
	 \begin{equation*}
	- \begin{pmatrix}
		 \partial_Y^2 u^{0*}\\
		 r\partial_Y^2v^{0*}
	 \end{pmatrix}=\begin{pmatrix}
		 f_1^{0*}\\
		 f_2^{0*}
	 \end{pmatrix},
	 \end{equation*}
	%since $F\in S^*$, which implies $\underline{F}=\underline{F}^0=0$, and so a particular solution of \eqref{eq:GenIntFour} is given by:
	which has the unique solution in $S^*$ given by 
	\begin{equation}\label{eq:PartSolZero}
	U^{0*}=\begin{pmatrix}
		u^{0*}\\
		v^{0*}
	\end{pmatrix}=\begin{pmatrix}
		-\int_{Y}^{\infty}\int_{s}^{\infty}f_1^{0*}(t,x,z)dzds\\
		-\frac{1}{r}\int_{Y}^{\infty}\int_{s}^{\infty}f_2^{0*}(t,x,z)dzds
	\end{pmatrix}.
	\end{equation}
	   Observing that $l_f(U^{0*})=\begin{pmatrix}
		\int_{0}^{\infty}f_1^{0*}(t,x,z)dz\\
		\int_{0}^{\infty}f_2^{0*}(t,x,z)dz\end{pmatrix}$,
		we see that  the boundary condition in \eqref{f2}(a) can hold only if we impose the solvability condition 
		\begin{align}\label{f3}
		\begin{pmatrix}
		\int_{0}^{\infty}f_1^{0*}(t,x,z)dz\\
		\int_{0}^{\infty}f_2^{0*}(t,x,z)dz\end{pmatrix}=\underline{G}(t,x).
		\end{align}

	Next consider the interior equations in \eqref{f2}(b)(c), which we will solve  by diagonalization after rewriting them as  first-order systems.   The equation for the $n$th Fourier mode in \eqref{f2}(b) is  
		\begin{align}\label{eq:GenIntFour}
			\begin{split}
			&-\partial_{YY}u^n-in(r-1)\partial_Yv^n-n^2(c^2-r)u^n=f_1^n\\
			&-r\partial_{YY}v^n-in(r-1)\partial_Yu^n-n^2(c^2-1)v^n=f_2^n,
		\end{split}
		\end{align}
where  $(u^n,v^n)=U_p^n(t,x,Y)$, $n\neq 0$.	 Introducing $\tilde{U}=(U,\partial_Y U)$ and $\tilde{F}=(0,F)$, we rewrite this as the $4\times 4$  first order system
	\begin{equation}\label{eq:FirstOrder}
		(\partial_Y-G(\beta,n))\tilde U^n=\left(\partial_Y-\begin{pmatrix}
			0 & I \\
			D(\beta,n) & B(n)
		\end{pmatrix}\right)
		\begin{pmatrix}
			U^n \\
			\partial_Y U^n
		\end{pmatrix}
		=\begin{pmatrix}
			0 \\
			-\begin{pmatrix}1&0\\0&1/r\end{pmatrix}F^n
		\end{pmatrix}:=\tilde F^n,
	\end{equation}
	 where the  matrices $B(n)$ and $D(\beta,n)$ are given by:
	 \begin{equation*}
		 B(n)=in\begin{pmatrix}
			 0 & 1-r \\
			 \frac{1}{r}-1 & 0
		 \end{pmatrix}\quad D(\beta,n)=n^2 \begin{pmatrix}
		 r-c^2 & 0\\
		 0 & \frac{1-c^2}{r}
		 \end{pmatrix}
	 \end{equation*}
	and $\beta=(-c,1)$ as before. The matrix $G(\beta,n)$ has eigenvalues $in\omega_j$, $j=1,\dots,4$, where 
	$$\omega_1^2=c^2-1, \omega_2^2=\frac{c^2}{r}-1, \omega_3=\overline{\omega_1} \text{ and }\omega_4=\overline{\omega_2}$$
	 ($\omega_1$, $\omega_2$ are purely imaginary with positive imaginary part), 
	 and corresponding right eigenvectors:
	\begin{equation*}
		R_1(n)=\begin{pmatrix}
			-\omega_1 \\
			1\\
			-in\omega_1^2\\
			in\omega_1
		\end{pmatrix}\quad R_2(n)=\begin{pmatrix}
		1 \\
		\omega_2\\
		in\omega_2\\
		in\omega_2^2\\
		\end{pmatrix}\quad R_3(n)=\overline{R_1(-n)} \quad R_4(n)=\overline{R_2(-n)}.
	\end{equation*}
	We also choose corresponding  left (row) eigenvectors $L_i(n)$, $i=1,\dots,4$ satisfying $L_iR_j=\delta_{ij}$.
	
	In the case when $F^{osc}=0$ (that is, $L_{ff}(U_h)=0$) we use the $L_i(n)$, $R_j(n)$ to diagonalize the left side of \eqref{eq:FirstOrder} and find easily that real, decaying solutions of $L_{ff}(U_h)=0$ must have the form\footnote{The functions $\tilde U_h=(U_h,\partial_YU_h)$ have the same form with $R_j(n)$ in place of $r_j$.}  
	\begin{equation*}
		U_h(t,x,\theta,Y)=\sum_{n\not=0}U_h^n(t,x,Y)e^{in\theta}
	\end{equation*}
	where the $U_h^n$ are given by the formulas:
	\begin{equation}\label{eq:GenHomSol}
		U_h^n=\begin{cases}
			\sigma_1(t,x,n)e^{in\omega_1Y}r_1+\sigma_2(t,x,n)e^{in\omega_2Y}r_2 &\text{for } n>0\\
			\sigma_3(t,x,n)e^{in\omega_3Y}r_3+\sigma_4(t,x,n)e^{in\omega_4Y}r_4 &\text{for } n<0.
		\end{cases}
	\end{equation}
Here $r_j$ is a column vector consisting of  the first two components of $R_j(n)$, 
	and  the $\sigma_j$ are undetermined scalar functions  satisfying $\sigma_3(t,x,n)=\bar{\sigma}_1(t,x,-n)$ and $\sigma_4(t,x,n)=\bar{\sigma}_2(t,x,-n)$ for $n<0$.\footnote{The $\omega_j$ and $r_j$ here are the same as in \eqref{ray1}, \eqref{oo3}.}

	%To return from the first order system to the original problem, we make the following definition:
	%\begin{equation*}
	%	r_1=\begin{pmatrix}
	%	-\omega_1\\
	%	1
	%	\end{pmatrix}\quad r_2=\begin{pmatrix}
	%		1\\
	%		\omega_2
	%	\end{pmatrix}\ \text{and }r_3=\bar{r}_1\quad r_4=\bar{r}_2
	%\end{equation*}
	%This amounts to taking the first two components of the $R_j$.
	
	Similarly, diagonalization yields particular real, decaying solutions of $L_{ff}(U_p)=F^{osc}$ with Fourier modes of the form 
	\begin{align}\label{f4a}
	 U^n_p(t,x,Y)=\sum^4_{j=1}\tau^n_j(t,x,Y)r_j,
	\end{align}
	where
	\begin{equation}\label{eq:PartSolNonzero}
		\tau^n_j(t,x,Y):=\begin{cases}
			\int_0^Y e^{in\omega_j(Y-s)}F_j^n(t,x,s)ds\quad \text{for } j=1,2\text{ and } n>0\\
			\int_{\infty}^{Y}e^{in\omega_j(Y-s)}F_j^n(t,x,s)ds\quad \text{for }j=3,4\text{ and }n>0\\
			\int_{\infty}^{Y}e^{in\omega_j(Y-s)}F_j^n(t,x,s)ds\quad\text{for } j=1,2\text{ and }n<0\\
			\int_0^Y e^{in\omega_j(Y-s)}F_j^n(t,x,s)ds\quad \text{for } j=3,4\text{ and } n<0
		\end{cases}
	\end{equation}
	  and  $F_j^n=L_j(n)\tilde{F}^n$. 	
%	Notice that the above equation only gives information about $u^{0*},v^{0*}$. This is because $\underline{u},\underline{v}$ are eliminated by $\partial_Y$, and so the above equation is insufficient to determine $\underline{u},\underline{v}$. Moreover, we cannot solve for $u^0$ and $v^0$ in $S$ if $\underline{f}_1,\underline{f}_2\not=0$, since $l_f(U)\in S^*$. This is also apparent from formula \eqref{eq:PartSolZero}, because if $\underline{F}\not=0$, then we have that $u^{0*},v^{0*}$ grow quadratically in $Y$.
	
	Finally, consider the boundary equation in \eqref{f2}(c), where the right side is now determined.  We have
	\begin{align}\label{f4}
	l_f(U_h^n)=\begin{pmatrix}\partial_Y u^n+inv^n\\r\partial_Y v^n+(r-2)inv^n\end{pmatrix}=\begin{pmatrix}
		0 & in & 1 & 0\\
		(r-2)in & 0 & 0 & r
		\end{pmatrix}\tilde U_h^n:=C(\beta,n)\tilde U_h^n, \;n\neq 0.
\end{align}
Using \eqref{eq:GenHomSol} for $n>0$ we write
\begin{align}\label{f5}
\begin{split}
&C(\beta,n)\tilde U_h^n=[C(\beta,n)R_1, C(\beta,n)R_2]\begin{pmatrix}
		\sigma_1(t,x,n)\\
		\sigma_2(t,x,n)
		\end{pmatrix}=in\begin{pmatrix}
			2-c^2 & 2\omega_2 \\
			2\omega_1 & c^2-2
		\end{pmatrix}\begin{pmatrix}
		\sigma_1\\
		\sigma_2
		\end{pmatrix}:=in\mathcal{B}_{Lop}\begin{pmatrix}
		\sigma_1\\
		\sigma_2
		\end{pmatrix}
\end{split}
\end{align}	
	Recall that the boundary frequency $\beta=(-c,1)$ was chosen so that $\mathcal{B}_{Lop}$ is singular.   Clearly,\footnote{Here we use  $2-c^2=2q$.} 
	\begin{align}\label{f6}
	\ker \mathcal{B}_{Lop}=\mathrm{span }\begin{pmatrix} \omega_2 \\ -q \end{pmatrix},\quad \mathrm{coker }\;\mathcal{B}_{Lop}=\mathrm{span }(q \;\;\omega_2), \text{ where }q^2=-\omega_1\omega_2 \text{ and }q>0.
	\end{align}
	We obtain a solvability condition for $ l_f(U_h)=G^{osc}-l_f(U_p)$ by considering  
	\begin{align}\label{f6b}
	l_f(U_h^n)=in\mathcal{B}_{Lop}\begin{pmatrix}\sigma_1(t,x,n)\\\sigma_2(t,x,n)\end{pmatrix}=G^n-C(\beta,n)\tilde U^n_p, \text{ for, say,  }n>0.
	\end{align}
	With \eqref{f6} we see that 
\begin{align}\label{f7}
	\begin{split}
	&(q \;\;\omega_2)\left( G^n-C(\beta,n)\tilde U^n_p\right)=0, \;\;n\neq 0,  \text{ or equivalently }\\
	&(q \;\;\omega_2) \left( G^{osc}-l_f(U_{p}\right)=0
	\end{split}
	\end{align}
   is a necessary and sufficient condition for the existence of a solution in $S^*$ of \eqref{f2}(c).   
	
	Assuming that \eqref{f7} holds we now complete the construction of $U_h$.   For $n>0$ we want to choose $\sigma_1$, $\sigma_2$ so that 
	\eqref{f6b} holds.  Although  $\mathcal{B}_{Lop}$ has a one dimensional kernel $\mathcal{K}:=\mathrm{span}\begin{pmatrix}\omega_2\\-q\end{pmatrix}$,  	
 we can fix a solution of \eqref{f6b} by taking 
		\begin{equation}\label{f7b}
		\begin{pmatrix}
		\sigma_{1}(t,x,n)\\
		\sigma_{2}(t,x,n)
		\end{pmatrix}=\frac{1}{in}\mathcal{B}_{Lop}^{-1} \left(G^n-C(\beta,n)\tilde{U}_{k,p}^n\right) \text{ for }n>0,
		\end{equation}
	where $\mathcal{B}_{Lop}^{-1}$ is the inverse of $\mathcal{B}_{Lop}:\mathcal{K}^\perp\to \mathrm{Im} \mathcal{B}_{Lop}$.

\begin{rem}[Solvability conditions]\label{f6a}
Summarizing, we have found the following solvability conditions for obtaining a solution $U\in S$  to \eqref{eq:GenInt} when $F=\underline F+F^*+F^m\in S^e$, $G\in S^b$:
\begin{align}\label{f8}
\begin{split}
& a)\; \underline F = 0, \;\; F^m=0  \quad (\text{ equivalently, }F\in S^*)\\
&b) \begin{pmatrix}
		\int_{0}^{\infty}f_1^{0*}(t,x,z)dz\\
		\int_{0}^{\infty}f_2^{0*}(t,x,z)dz\end{pmatrix}=\underline{G}(t,x)\\
&c)\; (q \;\;\omega_2)\left( G^{osc}-l_f(U_p)\right)=0.
%\; (q \;\;\omega_2)\left( G^n-C(\beta,n)\tilde U^n_p\right)=0, \;\;n\neq 0.
\end{split}
	\end{align}
	Given any such solution $U$, we can obtain  another solution in $S$ by  adding 	any $\underline V(t,x,y)\in \underline S$.    

\end{rem}
	
	We close this section by observing that since  $\mathcal{B}_{Lop}$ is singular,   there are nontrivial decaying solutions $U\in S$  to 
	\begin{align}\label{f9}
	L_{ff}(U)=0, \;l_f(U)=0.
	\end{align}
	   The kernel of $\mathcal{B}_{Lop}$ is spanned by $\begin{pmatrix} \omega_2 \\ -q \end{pmatrix}$, so using \eqref{eq:GenHomSol} we see that for $n>0$
	\begin{align}
	l_f(U^n)=in\mathcal{B}_{Lop}\begin{pmatrix}
			\sigma_1(t,x,n)\\
			\sigma_2(t,x,n)
		\end{pmatrix}=0\Leftrightarrow \begin{pmatrix}
			\sigma_1(t,x.n)\\
			\sigma_2(t,x,n)
		\end{pmatrix}=\alpha(t,x,n)\begin{pmatrix}
			\omega_2\\
			-q
		\end{pmatrix}
	\end{align}
	for some scalar function $\alpha$ to be determined.    Thus, we obtain nontrivial, real decaying solutions $U_\alpha(t,x,\theta,Y)\in S^*$ of \eqref{f9} defined by
	\begin{align}\label{f10}
	\begin{split}
	&U_\alpha^n(t,x,Y)=\alpha(t,x,n)\left(\omega_2 e^{in\omega_1Y}r_1-qe^{in\omega_2Y}r_2\right):=\alpha(t,x,n)\hat r(n,Y), \text{ for }n>0\\
	&U_\alpha^n(t,x,Y)=\alpha(t,x,n)\hat r(n,Y), \text{ for }n<0,
	\end{split}
	\end{align}
	 where 
	 \begin{align}\label{f11}
	 \alpha(t,x,n)=\overline{\alpha}(t,x,-n)\text{ and }\hat r(n,Y)=\overline{\hat r}(t,x,-n).
	\end{align}
	We will see below that solutions like $U_\alpha$  are used in the analysis of the profile equations to insure that the third solvability condition \eqref{f8}(c) holds.

	\section{Order of construction}\label{sec:Order}
	
	Consider again the cascade of profile equations
	\begin{align}\label{g1}
	L_{ff}(U_k)=\begin{pmatrix}H_{k-1}\\K_{k-1}\end{pmatrix},\;\; l_f(U_k)=\begin{pmatrix}h_{k-1}\\k_{k-1}\end{pmatrix},  \;\; k=2,3,\dots,...
	\end{align}
	where the $H_j,...,k_j$  are given by \eqref{eq:CascadeIntRHS}, \eqref{eq:CascadeBdyRHS}, $U_j=0$ for $j\leq 1$, and we seek $U_k\in S$.  
   In general when the expressions  \eqref{eq:CascadeIntRHS}, \eqref{eq:CascadeBdyRHS} are evaluated using profiles $U_j\in S$, one obtains elements of $S^e$ not $S^*$ (as is required by Remark \ref{f6a}).    Thus we need to work with modifications of the $H_j,K_j$ which we denote by $H_j',K_j'$.  These have to be defined inductively; for example, we will see that the choice of $\underline{U_{k-1}}$ is made so that $\underline {H'_k}=0$, $\underline {K'_k}=0$.  The definition of the $H_j',K_j'$ is given in \eqref{e4}; in this section, whose purpose is mainly to lay out the order of construction of the pieces of $U_k$,  we only need to know that $H_j',K_j'$ depend only on $U_k$ for $k\leq j$ and belong to $S$.

	We split the $U_k$'s into five pieces, 
	\begin{align}\label{g2}
	\begin{split}
	&U_k(t,x,y,\theta,Y)=\underline{U}_k(t,x,y)+U_k^{0*}(t,x,Y)+U_{k,h}(t,x,\theta,Y)+U_{k,p}(t,x,\theta,Y)+U_{k,\alpha}(t,x,\theta,Y),\\
	&\text{so now } U^0(t,x,y,Y)=\underline {U_k}+U_k^{0*},\;\;U^{osc}_k=U_{k,h}+U_{k,p}+U_{k,\alpha}.
\end{split}
\end{align}
	The terms $\underline {U_k}$, $U_{k,\alpha}$ are without analogues in \eqref{f1}: we will see that they are chosen to arrange solvability conditions for $U_{k+1}$.

	 The pieces of $U_k$ are constructed to satisfy
	\begin{align}\label{g3}
		\begin{split}
		&(a)\; L_{ff}(U_k^{0*})=\begin{pmatrix} H'^0_{k-1}\\ K'^0_{k-1}\end{pmatrix}\\
		&(b)\; L_{ff}(U_{k,p})=\begin{pmatrix} H'^{osc}_{k-1}\\K'^{osc}_{k-1}\end{pmatrix}\\
		&(c)\; L_{ff}(U_{k,h})=0 \quad\quad l_f(U_{k,h})=\begin{pmatrix} h_{k-1}^{osc}\\k_{k-1}^{osc} \end{pmatrix}-l_f(U_{k,p})\\
		&(d)\; L_{ff}(U_{k,\alpha})=0 \quad\quad l_f(U_{k,\alpha})=0\quad\quad (q\;\;\omega_2) \left(\begin{pmatrix} h_{k}^{osc}\\k_{k}^{osc} \end{pmatrix}-l_f(U_{k+1,p})\right)=0\\
		&(e)\;  \begin{pmatrix}\underline{H'_{k+1}}\\\underline{K'_{k+1}}\end{pmatrix}=0 \quad\quad \int_{0}^{\infty}\begin{pmatrix} H'^0_{k}\\ K'^0_{k}\end{pmatrix}dY=\begin{pmatrix} h_{k}^0\\ k_{k}^0\end{pmatrix}
	\end{split}
	\end{align}
	where the first equation in each line is on $y,Y>0$ and the second or third, when present, is on $y=Y=0$.   The equations in \eqref{g3} have obvious counterparts in \eqref{f2}, \eqref{f8}, \eqref{f9}.  Those  in line (e) are used to determine $\underline{U_k}$.     We sometimes refer to \eqref{g3} as \eqref{g3}$_k$. 
	
	The construction is done inductively.  When solving for $U_k$ we assume that real profiles $U_2$,\dots, $U_{k-1}$ in $S$ have been found satisfying \eqref{g3}$_j$, for $j\leq k-1$. 
	We also suppose that the $U_j$, $j\geq k$, although undetermined, lie in $S$.

	The first elements to determine are $U_{k,p}$ and $U_k^{0,*}$, which  depend only on profiles $U_j$, $j\leq k-1$.  From \eqref{eq:PartSolZero} we obtain
	%\footnote{Here and in \eqref{eq:UkP} we suppress the prime on $H_{k-1},K_{k-1}$.}
	\begin{equation}\label{eq:Uk0*}
		U_k^{0*}=-\int_{Y}^{\infty}\int_{s}^{\infty}\begin{pmatrix}
		H_{k-1}^{'0*}(t,x,z)\\
		\frac{1}{r}K_{k-1}^{'0*}(t,x,z)U_p
		\end{pmatrix}dzds.
	\end{equation}

	From \eqref{f4a} we see that  $U_{k,p}^n(t,x,Y)$  is given for $n>0$ by\footnote{Here we use $L_1(n)=(-in(r-c^2),-in\omega_1, \omega_1,-r)/(-2i\omega_1c^2n)$, \;$L_2(n)=(in\omega_2r,ik(c^2-1),1,r\omega_2)/(2i\omega_2c^2n)$.}
	\begin{equation}\label{eq:UkP}
		\begin{split}U_{k,p}^n&=
		-\int_0^Y \frac{e^{in\omega_1(Y-s)}}{-2i\omega_1c^2n}[\omega_1H_{k-1}^{'n}-K_{k-1}^{'n}]r_1+\frac{e^{in\omega_2(Y-s)}}{2i\omega_2c^2n}[H_{k-1}^{'n}+\omega_2K_{k-1}^{'n}]r_2ds\\
		&-\int_{\infty}^{Y}\frac{e^{in\omega_3(Y-s)}}{2i\omega_1c^2n}[-\omega_1H_{k-1}^{'n}-K_{k-1}^{'n}]r_3 + \frac{e^{in\omega_4(Y-s)}}{-2i\omega_2c^2n}[H_{k-1}^{'n}-\omega_2K_{k-1}^{'n}]r_4ds,
		\end{split}
	\end{equation}
	and for $n<0$ we have $U_{k,p}^n=\overline{U_{k,p}^{-n}}$. 
	%It can be verified from the properties of $H_{k-1}$, the $L_j$'s and $r_j$'s that $U_{k,p}^{-n}=\overline{U_{k,p}^n}$, which ensures that $U_{k,p}$ is a real valued function. Moreover, 3) of Proposition \ref{prop:SBasicProps} guarantees that $U_{k,p}$ is in $S^*$ if $H_{k-1},K_{k-1}\in S^*$.
	By definition of $S^*$  (Definition \ref{d1a})  Sobolev norms $H^s(t,x)$, $s\in \mathbb{N}$ of $H_{k-1}^{'n}(t,x,Y)$ and  $K_{k-1}^{'n}(t,x,Y)$   are rapidly decaying with respect to $n$ and exponentially decaying with respect to $Y$; so \eqref{eq:UkP} implies  $U_{k,p}\in S^*$. 
	
	Knowing $U_{k,p}$ we can now construct $U_{k,h}$.   By the induction assumption, line (c) of \eqref{g3}$_{k-1}$ shows that the solvability condition (recall \eqref{f7}) for 
	 $l_f(U_{k,h})=\begin{pmatrix} h_{k-1}^{osc}\\k_{k-1}^{osc} \end{pmatrix}-l_f(U_{k,p})$ does hold.
	Using \eqref{eq:GenHomSol} and \eqref{f7b} we obtain
	\begin{equation}\label{g4}
		U_{k,h}^n=\sigma_{1,k}(t,x,n)e^{in\omega_1Y}r_1+\sigma_{2,k}(t,x,n)e^{in\omega_2Y}r_2\text{ for }n>0,
	\end{equation}
	where 
	\begin{equation}\label{ukhc}
		\begin{pmatrix}
		\sigma_{1,k}(t,x,n)\\
		\sigma_{2,k}(t,x,n)
		\end{pmatrix}=\frac{1}{in}\mathcal{B}_{Lop}^{-1} \left(\begin{pmatrix}
			h_{k-1}^n\\
			k_{k-1}^n
		\end{pmatrix}-C(\beta,n)\tilde{U}_{k,p}^n\right).
		\end{equation}
	
	Moving next to $U_{k,\alpha}$, we see that the first two equations of line (d) of \eqref{g3}$_k$, together with \eqref{f10}, \eqref{f11},  imply that 
	\begin{align}\label{g5}
	U_{k,\alpha}^n(t,x,Y)=\alpha_k(t,x,n)\hat r(n,Y), \;n\neq 0
	\end{align}
	for some function $\alpha_k$ to be determined.    This function is determined in section \ref{sec:Amplitudes} so that  the third equation in line (d) of  \eqref{g3}$_k$ holds.  
	It turns out to depend only on the pieces of $U_k$ that are already known and previous profiles. 
		
	The last piece to construct is $\underline{U_k}$. This function is determined in section \ref{sec:AnalysisU2} so that the equations in line (e) of  \eqref{g3}$_k$ hold.  It turns out to depend only on $U_k^{0,*}$,  $U_{k,p}$, $U_{k,h}$, $U_{k,\alpha}$ and previous profiles.

	%There is not very much flexibility to force $U_k^{0*}$ to satisfy boundary conditions since it is a particular solution.

	% It turns out that generically $H_{k-1}, K_{k-1}\not\in S$, this issue will be analyzed in more detail in section \ref{sec:AnalysisUk}. 	
	
	\begin{remark}[Order of construction]
		%There is a little bit of flexibility in the order because $U_k^{0*}$ is only dependent on the previous profiles, so it could be determined before $U_{k,p}$, $U_{k,h}$, or $U_{k,\alpha}$. 
To summarize, the order of construction is 
\begin{align}
U_{k,p} \text{ or }U_k^{0,*}, U_{k,h}, U_{k,\alpha}, \underline{U_k}.
\end{align}	
The first two pieces depend only on previous profiles,  $U_{k,h}$ depends also on $U_{k,p}$, while $U_{k,\alpha}$ depends also  on $U_k^{0,*}$,  $U_{k,p}$ and $U_{k,h}$. Finally, $\underline{U_k}$  depends also on all four other pieces of $U_k$.
	\end{remark}

\section{Amplitude equations}\label{sec:Amplitudes}

We now discuss the construction of $U_{k,\alpha}$, which is chosen so that the third equation in line (d) of \eqref{g3}$_k$, the solvability condition for $U_{k,h+1}$,  holds.  Recall from \eqref{f10} that the Fourier modes of $U_{k,\alpha}$ have the form 
\begin{align}\label{h1}
	U_{k,\alpha}^n(t,x,Y)=\alpha_k(t,x,n)\hat r(n,Y), \;n\neq 0
	\end{align}
for an ``amplitude" $\alpha_k$ to be determined.   The solvability condition can be written
\begin{equation}\label{h2}
\begin{pmatrix}
q & \omega_2
\end{pmatrix}
\left(
\begin{pmatrix}
h^n_k\\
k^n_k
\end{pmatrix}
-C(\beta,n)\tilde{U}^n_{k+1,p}\right)=0, \;n\neq 0.
\end{equation}
%It follows from the general argument of  \cite{Coddington} (Chapter 11, Theorem 4.1)  that  this is an equivalent form of equation $\eqref{eq:Amplitude1}$, \footnote{This means that $\alpha_k$ is a solution of one form of the equation if and only if it is a solution of the other form.}  and thus of \eqref{eq:AmplitudeEq2} or \eqref{eq:AmplitudeEqk} .
%This equation, while much more intuitive than \eqref{eq:AmplitudeEq2}, makes it more difficult to determine the amplitude equation. 
%At first glance, \eqref{eq:DualityRelation2} does not seem to involve $\alpha_k$, but both $\begin{pmatrix}
%h^n_k\\
%k^n_k
%\end{pmatrix}$ and $\tilde{U}^n_{k+1,p}$
%depend on $U_k$ and thus on $U_{k,\alpha}$. 
%$h_k,k_k$ depend on  $l_s(U_{k})$ and $\tilde U_{k,p}$ contains an integral of $U_{k-1,P}$ 
%and therefore both terms contain $U_{k-1,\alpha}$. 
%\footnote{The algebra to get from \eqref{eq:DualityRelation2} to \eqref{eq:AmplitudeEq2} or \eqref{eq:AmplitudeEqk} is more difficult than going from \eqref{eq:Amplitude1}.}
%Thus, the amplitude equation is a necessary and sufficient condition to solve for $U_{k+1,h}$. Referring to \cite{Williams} and \cite{Hunter06}, we get the following proposition showing that the amplitude equation is well-posed.

To see first  the dependence of the terms in  \eqref{h2} on $U_k$, we can use the formula \eqref{eq:UkP}$_{k+1}$ for $U^{n}_{k+1,p}$ together with the formulas for $H^{n}_k$, $K^{n}_k$, $h^n_k$, $k^n_k$ given by \eqref{eq:CascadeIntRHS}, \eqref{eq:CascadeBdyRHS}.    The functions $H^{'}_k$, $K^{'}_k$, defined in \eqref{e4} and appearing in  \eqref{eq:UkP}$_{k+1}$,  are modifications of $H_k$, $K_k$ belonging to $S^*$.  Like $H_k$, $K_k$ these functions depend just on the $U_j$ for $j\leq k$.   For the purposes of this section the only other information we need is that the terms of $H^{'}_k$, $K^{'}_k$ in which $U_k$ appears are exactly the same as the terms of $H_k$, $K_k$ in which $U_k$ appears.  This observation allows us to use \eqref{eq:CascadeIntRHS}.   Thus, we can write
\begin{equation}\label{h3}
	\begin{split}
		&(a) \begin{pmatrix}H^{'n}_k\\K^{'n}_k\end{pmatrix}=-[L_{fs}(U_k)]^n+[A_{fff}(U_2,U_k)]^n+[A_{fff}(U_k,U_2)]^n+
		 N^n_1(U_2,...,U_{k-1})\\
		 &(b) \begin{pmatrix}h_k^n\\k_k^n\end{pmatrix}=- [l_s(U_k)]^n-[Q_2(\partial_{\f};\partial_{\f})(U_2,U_k)]^n-
		 [Q_2(\partial_{\f};\partial_{\f})(U_k,U_2)]^n+\\
		 &\qquad \qquad N^n_2(U_2,...,U_{k-1})+\begin{pmatrix}
		f_k^n\\
		g_k^n
		\end{pmatrix}
	\end{split}
	\end{equation}
	where the $N_j$'s are (known) nonlinear functions depending only on the profiles $U_2,...,U_{k-1}$, and the corresponding expression for $\begin{pmatrix}H^{n}_k\\K^{n}_k\end{pmatrix}$ differs from \eqref{h3}(a) only in the function $N_1$.
	
	Writing  $U_{k}=\underline{U_k}+U_k^{0*}+U_{k,\alpha_k}+U_{k,h}+U_{k,p}$ and using the bilinearity of the operators with arguments  $(U_2,U_k)$ or  $(U_k,U_2)$,  we claim we can modify the above to:
	\begin{equation}\label{h4}
	\begin{split}
		&\begin{pmatrix}H^{'n}_k\\K^{'n}_k\end{pmatrix}=-[L_{fs}(U_{k,\alpha})]^n+[A_{fff}(U_2,U_{k,\alpha})]^n+[A_{fff}(U_{k,\alpha},U_2)]^n+
		 N^n_3(U_2,...,U_{k-1},U_k^{0*},U_{k,h},U_{k,p} )\\
		 &\begin{pmatrix}h_k^n\\k_k^n\end{pmatrix}=- [l_s(U_{k,\alpha})]^n-[Q_2(\partial_{\f};\partial_{\f})(U_2,U_{k,\alpha})]^n-
		 [Q_2(\partial_{\f};\partial_{\f})(U_{k,\alpha}U_2)]^n+\\
		 &\qquad\qquad N^n_4(U_2,...,U_{k-1},,U_k^{0*},U_{k,h},U_{k,p} )+\begin{pmatrix}
		f_k^n\\
		g_k^n
		\end{pmatrix}.
	\end{split}
	\end{equation}
	This follows just from the fact that the right side of \eqref{h3} does not depend on $\underline{U_k}$.\footnote{The term  $[l_s(U_k)]^n$ is independent of $\underline{U_k}$ since $l_s$ is linear and $n\neq 0$.}
	At this point in the construction the arguments of $N_3$ and $N_4$ are known.  Thus, when \eqref{h4} is plugged into \eqref{h2} we obtain an equation in which the \emph{only} unknown is $\alpha_k$.  We refer to this as the \emph{amplitude equation} for $\alpha_k$.  
	
	Although it is  a lot of work to unravel the explicit form of the equation for $\alpha_k$, it is already clear from a glance at  \eqref{h2}, \eqref{eq:UkP}, and \eqref{h4} that the equation for $\alpha_2$ has a quadratic nonlinearity with a forcing term depending only on $(f_2,g_2)$, and that the equation for $\alpha_k$ is a linearized form of the equation for $\alpha_2$.
	%\footnote{Observe that since $U_j=0$ for $j\leq 1$, the formulas of section \ref{sec:Order} show immediately that $U_2=\underline{U_2}+U_{2,\alpha}$.  We will see in the next section that $\underline{U_2}=0$.} 
	%Amplitude equations for nonlinear elasticity, together with more or less similar amplitude equations arising in other areas,   
	
	Nonlocal amplitude equations involving bilinear Fourier multipliers arising in nonlinear elasticity and other areas have been studied by  a number of authors including 
	\cite{Lardner,Hunter06,B,Marcou, Benzoni-Gavage,Secchi, Williams}.   The next proposition describes the form of the equation that arises in isotropic hyperelastic nonlinear elasticity.

\begin{proposition}\label{prop:Amplitudes}
		
		a) The amplitude equation for $\alpha_2$ has the form
		\begin{equation}\label{eq:AmplitudeEq2}
		\partial_t\alpha_2+c\partial_x\alpha_2+\mathcal{H}(\mathcal{B}(\alpha_2,\alpha_2))=G_2(f_2,g_2)
	\end{equation}
	where $(-c,1)=\beta$,  $\mathcal{H}$ denotes the Hilbert transform with respect to $\theta$ ($\widehat{\mathcal{H}f}(k):=-i\sgn(k)\hat f(k)$), and $\mathcal{B}$ is the bilinear Fourier multiplier given by:
	\begin{equation}\label{eq:BilFourMult}
		\widehat{\mathcal{B}(\alpha_2,\alpha_2)}(n):=-\frac{1}{4\pi c_0}\sum_{n'\not=0}b(-n,n-n',n')\alpha_2(t,x,n-n')\alpha_2(t,x,n').
	\end{equation}
	 The kernel $b(n_1,n_2,n_3)$, determined in \cite{Hunter06,Williams},  is symmetric in its arguments and homogeneous of  degree two.
	The constant $c_0$ is determined in \cite{Williams}.
	 %\footnote{The kernel $b$	 is symmetric in its arguments and homogeneous of  degree two. See proposition \ref{propwellposed}.}
	
	b) For $k\geq3$ the amplitude equation has the form 
	\begin{equation}\label{eq:AmplitudeEqk}
		\partial_t\alpha_k+c\partial_x\alpha_k+2\mathcal{H}(\mathcal{B}(\alpha_2,\alpha_k))=G_k(U_2,\dots,U_{k-1}). 
	\end{equation}
	
\begin{rem}\label{h5}
The analogue of \eqref{eq:AmplitudeEq2} for space dimensions $d\geq 3$ is given in \cite{Williams}.  The only change is that the $c\partial_x$ operator is replaced in higher dimensions by $c\frac{\underline \eta}{|\underline \eta|}\cdot\nabla_{x'}$, where the Rayleigh frequency is now $(-c|\underline{\eta}|,\underline \eta)$, $\underline{\eta}\in\mathbb{R}^{d-1}\setminus \{0\}$, and $c$ is as before.
It is shown in chapter 2 of \cite{Williams} that the vector field $\partial_t+c\frac{\underline \eta}{|\underline \eta|}\cdot\nabla_{x'}$, which governs the  speed and direction of Rayleigh waves along the boundary,  is a characteristic vector field  of the Lopatinski determinant.  

\end{rem}

\end{proposition}

The well-posedness of the amplitude equation  \eqref{eq:AmplitudeEq2} has been studied in \cite{Hunter06,Williams}.\footnote{J. Hunter in \cite{Hunter06} studies a different, but equivalent, form of the equation which shares the essential feature that his kernel ``$b$"   is unbounded with positive homogeneity.   \cite{Secchi} studies a related equation with unbounded,  positively homogeneous kernel on a plasma-surface interface.  The kernels studied in \cite{B,Marcou} are bounded.}  Here we need a version of the result in which the time of existence of $\alpha_2$, even for $H^\infty$ solutions,   depends only on a \emph{fixed} low order of   regularity.   We state the result for the following problem 
in $d\geq 2$ space dimensions (so $x'\in\bbR^{d-1}$):
 \begin{equation}\label{ae2}
		\begin{split}
		&\partial_t\alpha+v\cdot\nabla_{x'}\alpha+\mathcal{H}(\mathcal{B}(\alpha,\alpha))=G(t,x',\theta)\\
		&\alpha(t,x',\theta)=0\text{ in }t<0, 
		\end{split}
	\end{equation}
where $\mathcal{H}$ and $\mathcal{B}$ are as in Proposition \ref{prop:Amplitudes}, and $v\in \bbR^{d-1}$ is a fixed vector.
	
\begin{proposition}\label{prop:Amplitudes2}
%	TODO:CHECK   There exists an integer $\bar{m}$ dependent only on the spatial dimension $d$ such that for every $m\in\mathbb{N}$ with $m\geq\bar{m}$ and every $R>0$ there exists a $T=T(m,R)$ such that if $||\alpha_0||_{H^m}<R$, then there exists a unique $\alpha\in\mathcal{C}([0,T];H^m(\mathbb{R}^{d-1}\times\mathbb{T};\mathbb{Z})$ to equations \eqref{eq:AmplitudeEq2} and \eqref{eq:AmplitudeEqk} satisfying $\alpha|_{t=0}=\alpha_0$.
Let $m\geq m_1>\frac{d}{2}+2$ and suppose $G(t,x',\theta)\in \mathcal{C}([0,T_0];H^m(\mathbb{R}^{d-1}\times\mathbb{T}))$ for some $T_0>0$.   For every $R>0$ there exists a $T=T(m_1,R)\leq T_0$ such that if  $|G|_{\mathcal{C}([0,T_0];H^{m_1}(\mathbb{R}^{d-1}\times\mathbb{T}))}<R$, then there exists a unique solution $\alpha\in\mathcal{C}([0,T];H^m(\mathbb{R}^{d-1}\times\mathbb{T}))\cap\mathcal{C}^1([0,T];H^{m-1}(\mathbb{R}^{d-1}\times\mathbb{T}))$ to the problem \eqref{ae2}.

\end{proposition}

\begin{rem}
1) This proposition follows directly from the main well-posedness result of \cite{Hunter06}.  The essential step for obtaining a time of existence depending on a fixed low order of regularity $H^{m_1}$ is to obtain a tame estimate of the form
\begin{align}\label{ae3}
\left|\dfrac{{\rm d}}{{\rm d}t} \| \alpha(t) \|_{H^m}^2\right|  \le C \, \| \alpha(t) \|_{H^m}^2 \|\alpha(t)\|_{H^{m_1}}
\end{align}
 for solutions to the Cauchy problem with zero interior forcing and nonzero initial data.  
 Although Hunter only \emph{uses} a weaker nontame estimate, namely, 
 \begin{align}\label{ae4}
\left|\dfrac{{\rm d}}{{\rm d}t} \| u(t) \|_{H^m}^2\right|  \le C \, \| u(t) \|_{H^m}^3, 
\end{align}
 to obtain a time of existence that shrinks with increasing $H^m$ regularity, he actually \emph{proves} the better estimate \eqref{ae3}.   The arguments using the tame estimate to get a time of existence depending just on $H^{m_1}$ regularity are standard and given, for example, in \cite{TaylorIII}, Chapter 16. 

2) Given $\alpha_2$ and $G$  as in Proposition \ref{prop:Amplitudes2}, one readily obtains (again using estimates contained in the proofs of \cite{Hunter06}) a solution with the same regularity on the same time interval to the \emph{linear} problem
\begin{equation}\label{ae5}
		\begin{split}
		&\partial_t\alpha+v\cdot\nabla_{x'}\alpha+2\mathcal{H}(\mathcal{B}(\alpha_2,\alpha))=G(t,x',\theta)\\
		&\alpha(t,x',\theta)=0\text{ in }t<0, 
		\end{split}
	\end{equation}
corresponding to  \eqref{eq:AmplitudeEqk}.

3) If $G\in H^\infty$ in \eqref{eq:AmplitudeEq2} then one can show $\alpha\in H^\infty$ by using the equation  to deduce increased regularity with respect to $t$.  If one has both $G$ and $\alpha_2$ in $H^\infty$, the same remark applies to the  solution  of \eqref{ae5}. 

4) In part \ref{tamewellp} we give a new proof of the tame estimate \eqref{ae3} that applies directly to the form of the amplitude equation given in \eqref{eq:AmplitudeEq2} 
%Our proof also applies directly to the particular form of the amplitude equation given in \eqref{eq:AmplitudeEqk} 
and also incorporates the slow tangential space variables.

\end{rem}

\section{Final steps in the construction of the  $U_k$}\label{sec:AnalysisU2}
	
	We now use the results of the previous two sections to complete the construction of the profiles.   In the construction of $U_2$ and $U_3$, we  ask to reader to accept for a moment that  the definition of $(H'_k,K'_k)$ \eqref{e4} implies that 
	\begin{align}\label{i1}
	(H_j,K_j)\in S^* \text{ for }j\leq k\Rightarrow (H_k,K_k)=(H'_k,K'_k)\in S^*.
	\end{align}
	This reflects the fact that in such cases there is no need to modify $(H_k,K_k)$.

	Recall that all profiles are required to vanish in $t\leq 0$. 
	
	\subsection{The profile $U_2$}
	The pieces of the leading profile $U_2$ are determined by the equations \eqref{g3}$_{2}$.    The only nonvanishing piece turns out to be $U_{2,\alpha}$.

	\begin{proposition}\label{prop:U2=U2alpha}
		The leading order profile $U_2$ is given by $U_2=U_{2,\alpha}(t,x,\theta,Y)\in S^*$.
	\end{proposition}
	\begin{proof}
	\textbf{1. }We follow the procedure outlined in section \ref{sec:Order}.    Since the profiles $U_j$, $j\leq 1$ are zero, we have $(H_j,K_j)=0$ for $j\leq 1$, so \eqref{i1}
	and the formulas \eqref{eq:Uk0*}, \eqref{eq:UkP},  and \eqref{g4}, \eqref{ukhc} imply immediately that $U_2^{0*}$, $U_{2,p}$, and $U_{2,h}$ are zero.  The term $U_{2,\alpha}$ is given by \eqref{g5}, where $\alpha_2$ is provided by Proposition \ref{prop:Amplitudes2}.

	%The next component of $U_2$ is $U_2^{0*}$, which can be calculated from equation \eqref{eq:LeadingInterior} with $n=0$:
	%\begin{equation}
	%\partial_{YY}
	%	\begin{pmatrix}
	%	u_2^{0*}\\
	%	v_2^{0*}
	%	\end{pmatrix}=0
	%\end{equation}\\
	%Since we want $U_2^{0*}$ to decay at infinity, this means we must choose $U_2^{0*}=0$, as the above equation only has linear functions as solutions.
	
	\textbf{2. } The only remaining component to find is $\underline{U_2}$. This piece is determined by line (e) of \eqref{g3}$_2$.  We now use the fact, which will be  clear from the definition \eqref{e4}, that 
	\begin{align}\label{i2}
	\begin{pmatrix}
	\underline{H'_k}\\
	\underline{K'_k}
	\end{pmatrix}=\begin{pmatrix}
	\underline{H_k}\\
	\underline{K_k}
	\end{pmatrix} \text{ for all }k,
	\end{align}
	so we can use \eqref{eq:CascadeIntRHS} to write  the interior equation as 
\begin{equation*}\label{eq:U2BarEqV1}
	\begin{split}
	&\begin{pmatrix}
	\underline{H_3}\\
	\underline{K_3}
	\end{pmatrix}=
	\underline{A_{ffs}(U_2,U_2)+A_{fff}(U_2,U_3)+A_{fff}(U_3,U_2)+B_{ffff}(U_2,U_2,U_2)-L_{ss}(U_2)-L_{fs}(U_3)}=0.
	\end{split}
	\end{equation*}
	Using Proposition \ref{d2a} it is easy to see that the terms involving fast derivatives vanish, so this equation reduces to 
	\begin{equation}\label{u2bar}
		L_{ss}(\underline{U}_2)=0 \text{ on }y>0.
	\end{equation}
	
	The boundary conditions for $\underline{u}_2,\underline{v}_2$  come from 
	\begin{equation}\label{eq:U2barBdy}
	\int_{0}^{\infty}\begin{pmatrix}
	H_2^{'0}\\
	K_2^{'0}
	\end{pmatrix}dY=\begin{pmatrix}
	h_2^0\\
	k_2^0
	\end{pmatrix}\text{ on }y=Y=0.
	\end{equation}
	Inspection of \eqref{eq:CascadeIntRHS} shows that $(H_2,K_2)\in S^*$, so $(H_2,K_2)=(H'_2,K'_2)$.\footnote{We use \eqref{i1} here.}  Thus \eqref{eq:U2barBdy} is 
		\begin{equation*}
	\int_0^\infty -[L_{fs}(U_2)]^0+\partial_Y[Q_2(\partial_{\f};\partial_{\f})(U_2,U_2)]^0dY=-[l_s(U_2)]^0-[Q_2(\partial_{\f};\partial_{\f})(U_2,U_2)]^0
	\end{equation*}
	Since $U_2^{0*}=0$ we have $[L_{fs}({U}_2)]^0=0$. Computing the integral  gives
	\begin{equation*}
	-[Q_2(\partial_{\f};\partial_{\f})(U_2,U_2)]^0=-[l_s(U_2)]^0-[Q_2(\partial_{\f};\partial_{\f})(U_2,U_2)]^0
	\end{equation*}
	which reduces to $l_s(\underline{U}_2)=0$ on $y=0$.\footnote{We use here the fact that $\underline G(t,x')=0$.}  With \eqref{u2bar} this gives $\underline{U_2}=0$.

	\end{proof}
	
	\subsection{The profile $U_3$}\label{sec:AnalysisU3}
	The profile $U_3$ is determined by the equations \eqref{g3}$_3$.  Knowing that $\underline{U_2}=0$, we see from the formulas \eqref{eq:CascadeIntRHS} that 
	\begin{align}
	\begin{pmatrix}H_3\\K_3\end{pmatrix}\in S^*\text{ and  }\begin{pmatrix}H_4\\K_4\end{pmatrix}\in S.
\end{align}

	\begin{proposition}
		There exists a profile  $U_3\in S$ satisfying the equations \eqref{g3}$_3$.
	\end{proposition}
	\begin{proof}

	\textbf{1. }By \eqref{i1} we have $(H_2,K_2)=(H'_2,K'_2)$, so we can use \eqref{eq:CascadeIntRHS} in  the formulas \eqref{eq:Uk0*}, \eqref{eq:UkP},  and \eqref{g4}, \eqref{ukhc} to construct the pieces 
	$U_3^{0*}$, $U_{3,p}$, and $U_{3,h}$.   The term $U_{3,\alpha}$ is given by \eqref{g5}, where $\alpha_3$ is provided by Proposition \ref{prop:Amplitudes2}.

	\textbf{2. } Finally, we construct $\underline{U}_3$, which is determined by line (e) of \eqref{g3}$_3$.    Using \eqref{i2} and the definitions of $H_4,K_4$ provided in \eqref{eq:CascadeIntRHS}, we obtain for the interior equation 
	\begin{equation}\label{eq:Uk3BarIntV2}
		\begin{split}\begin{pmatrix}
			\underline{H}_4\\
			\underline{K}_4
		\end{pmatrix}=&
			\underline{-L_{fs}(U_4)-L_{ss}(U_3)+A_{fss}(U_2,U_2)+A_{ffs}(U_2,U_3)+A_{ffs}(U_3,U_2)+A_{fff}(U_3,U_3)}\\
			&\underline{+B_{fffs}(U_2,U_2,U_2)+B_{ffff}(U_3,U_2,U_2)+B_{ffff}(U_2,U_3,U_2)+B_{ffff}(U_2,U_2,U_3)}=0.
	\end{split} 
	\end{equation}
	This simplifies to
	\begin{equation}\label{eq:Uk3BarIntFin}
		L_{ss}(\underline{U}_3)=0
	\end{equation}
	by an argument similar to that which gave \eqref{u2bar}.

	The boundary conditions for $\underline{U}_3$ come from the formula
	\begin{equation}\label{eq:Uk3BarBdyV1}
		\int_{0}^{\infty}\begin{pmatrix}
			H_3^{'0}\\
			K_3^{'0}
		\end{pmatrix}dY=\begin{pmatrix}
			h_3^0\\
			k_3^0
		\end{pmatrix}\text{ on }y=Y=0.
	\end{equation}
	Applying \eqref{i1} again, we can can use the definitions  of $H_3$ ,$K_3$, $h_3$ and $k_3$ in \eqref{eq:Uk3BarBdyV1} to obtain
	
	\begin{equation}\label{eq:Uk3BarBdyV2}
		\begin{split}
			\int_0^\infty&[-L_{fs}(U_3)+A_{fff}(U_2,U_3)+A_{fff}(U_3,U_2)+N_1(U_2)]^0dY=\\
			&[-l_s(U_3)-Q_2(\partial_{\f};\partial_{\f})(U_2,U_3)-Q_2(\partial_{\f};\partial_{\f})(U_3,U_2)+N_2(U_2)]^0
		\end{split}
	\end{equation}

	%\begin{equation}\label{eq:Uk3BarBdyV2}
	%	\begin{split}
	%		\int_0^\infty&[-L_{fs}(U_3)-L_{ss}(U_2)+A_{ffs}(U_2,U_2)+A_{fff}(U_2,U_3)+A_{fff}(U_3,U_2)+B_{ffff}(U_2,U_2,U_2)]^0dY=\\
	%		&[-l_s(U_3)-Q_2(\partial_{\f};\partial_{\f})(U_2,U_3)-Q_2(\partial_{\f};\partial_{\f})(U_3,U_2)-Q_2(\partial_{\f};\partial_{\s})(U_2,U_2)\\
	%		&-Q_2(\partial_{\s};\partial_{\f})(U_2,U_2)-C_2(\partial_{\f};\partial_{\f};\partial_{\f})(U_2,U_2,U_2)]^0
	%	\end{split}
	%\end{equation}
	To simplify this expression recall that the $A$'s are related to the $Q$'s as described in the  formulas \eqref{eq:AinTermsofQ}. The integral of  the term $[A_{fff}(U_3,U_2)]^0$ can be expanded 
	\begin{equation}\label{eq:Uk3CubicInt}
		\begin{split}
		&\int_0^\infty[A_{fff}(U_3,U_2)]^0dY=\int_0^\infty[\partial_\theta Q_1(\partial_{\f};\partial_{\f})(U_3,U_2)+\\
		&\qquad \partial_Y Q_2(\partial_{\f};\partial_{\f})(U_3,U_2)]^0dY=-[Q_2(\partial_{\f};\partial_{\f})(U_3,U_2)]^0|_{Y=0},
\end{split}
	\end{equation}
	since $[\partial_\theta Q_1(\partial_{\f};\partial_{\f})(U_3,U_2)]^0$ vanishes.	Notice that the right hand side  is a term appearing in $\begin{pmatrix}
		h_3^0\\
		k_3^0
	\end{pmatrix}$.  
	%Similarly, the $A_{fff}(U_3,U_2)$ term in $H_3,K_3$ cancels with the $Q_2(\partial_{\f};\partial_{\f})(U_3,U_2)$ in $h_3,k_3$.   After taking account of other terms that cancel or vanish,
	Doing the same for the other $A_{fff}$ term, 
	 we reduce  \eqref{eq:Uk3BarBdyV2} to the following:
	\begin{equation}\label{eq:Uk3BarBdyV3}
		\int_0^\infty[-L_{fs}(U_3)+N_1(U_2,U_2)]^0dY=[-l_s(U_3)+N_2(U_2)]^0
	\end{equation}
Using  $U_2^0=0$, which implies $[L_{ss}(U_2)]^0=0$, and  $L_{fs}(U_3)=L_{fs}(U_3^*)$, we get the final form of the boundary conditions:
	\begin{equation}\label{eq:U3BarBdyFin}
		l_s(\underline{U}_3)=-l_s(U_3^{0*})+\int_{0}^{\infty}[L_{fs}(U_3^*)-\partial_x Q_1(\partial_{\theta,Y};\partial_{\theta,Y})(U_2,U_2)]^0dY.
	\end{equation}
In view of  \eqref{eq:Uk3BarIntFin} and \eqref{eq:U3BarBdyFin}  we obtain  a unique solution  $\underline{U}_3\in \underline S$.     This  completes the construction of $U_3$.
	\end{proof}
	
	\begin{remark}\label{internal}
		One of the goals of Chapter 2 of \cite{Mar2} is to show that even though $\underline G(t,x)=0$ and $U_2\in S^*$,  it can happen that  $\underline{U}_3\not=0$. This conclusion is reached by showing  that the right side of \eqref{eq:U3BarBdyFin}, or rather its analogue in her simplified model,  is not $0$ and hence neither is $\underline{U}_3$.    This is an example of ``internal rectification''.
	The computation of \cite{Mar2} shows there is every reason to expect that the right side of \eqref{eq:U3BarBdyFin} is nonzero in the Saint Venant-Kirchhoff model as well, except for rare accidents.  When that happens,  the error analysis of part \ref{p2} shows that internal rectification is truly present in the exact solution.
		
		% see \cite[Marcou] for more details.
	\end{remark}
	
	\subsection{The profiles $U_k$, $k\geq 4$}\label{sec:AnalysisUk}    It remains to construct $U_k\in S$ satisfying \eqref{g3}$_k$, assuming that profiles $U_2,\dots,U_{k-1}$ in $S$ satisfying \eqref{g3}$_j$, $j\leq k-1$ have already been constructed.    By the construction of $U_2$ and $U_3$ we see that $(H_j,K_j)\in S^*$ for $j\leq 4$, so there was no need to modify $(H_j,K_j)$ for these $j$.\footnote{In \cite{Mar2} Marcou gave a construction of $U_2$ and part of $U_3$ which involved  only  $(H_j,K_j)$  for $j\leq 4$. Thus, she did not need to carry out a modification process like the one we describe below.} But for $j\geq 5$ we must expect $(H_j,K_j)$ to contain terms in $S^m$.  For instance, the term $\partial_x(\partial_y u_3\partial_\theta v_2)$ from $A_{fss}(U_3,U_2)\in S^m$, since $\underline{u}_3\not=0$ (normally), which implies that $H_5,K_5\not\in S^*$. \footnote{Recall that $H_j,K_j\in S^*$ for $j\leq 4$ since $\underline U_2$ turned out to be zero.    For these $j$ we have $H_j=H_j'$, $K_j=K_j'$.}   
In order to construct the higher profiles we must now define the $(H'_j,K'_j)$, $j\geq 5$.   We define the $(H'_j,K'_j)$ as elements of $S$, even though they will turn out by the choice of $\underline{U_{j-1}}$ to lie in $S^*$. 
	
	%In view of \eqref{d4}  we would like the general term $U_k$ to satisfy:
	%\begin{equation}\label{e1}
	%	\begin{split}
	%	&(a) L_{ff}(U_k)=\begin{pmatrix}
	%	H_{k-1}\\
	%	K_{k-1}
	%	\end{pmatrix} \text{ on }y,Y>0\\
	%	&(b) l_f(U_k)=\begin{pmatrix}
	%		h_{k-1}\\
	%		k_{k-1}
	%	\end{pmatrix}\text{ on }y,Y=0,
	%	\end{split}
	%\end{equation}
	%where  the functions $H_{k-1},K_{k-1}$ are defined in terms of $U_2,...,U_{k-1}$ by formula \eqref{eq:CascadeIntRHS}, and $h_{k-1},k_{k-1}$ are defined by \eqref{eq:CascadeBdyRHS}.  In Remark \ref{nec} we noted that a necessary condition for solving \eqref{e1}(a) with $U_k\in S$ is that the right side lie in $S^*$. 
%However, even though $\underline H_{k-1}, \underline K_{k-1}$ are both constructed to be zero,
%	the  functions  $H_{k-1},K_{k-1}$ are generically \textit{not} in $S^*$,  but rather in $S^*\oplus S^m$. So for instance, the term $\partial_x(\partial_y u_3\partial_\theta v_2)$ from $A_{fss}(U_3,U_2)$ is not in $S^*$ since $\underline{u}_3\not=0$ (normally), which implies that $H_5,K_5\not\in S^*$. \footnote{Recall that $H_j,K_j\in S^*$ for $j\leq 4$ since $\underline U_2$ turned out to be zero.    For these $j$ we have $H_j=H_j'$, $K_j=K_j'$.}   
	
For functions $f=\underline f+ f^*+f^m \in S^e$, we can define a modification $f^{mod}=\underline f+ f^*+f^{m,mod}$, where $f^{m,mod}$ is as in \eqref{d2a}.\footnote{We now suppress the subscript on $f^{m,mod}$; it will be clear from the context.}Applying this to the $H_{k-1}, K_{k-1}$  for $6\leq k\leq N$ we obtain a preliminary modification
	\begin{align}\label{e2}
	\begin{split}
	&\begin{pmatrix}H_{k-1}\\K_{k-1}\end{pmatrix}^{mod}=\underline{\begin{pmatrix}H_{k-1}\\K_{k-1}\end{pmatrix}}+\begin{pmatrix}H_{k-1}\\K_{k-1}\end{pmatrix}^*+M_{k-1,0}+\eps Y M_{k-1,1}+\dots+\\
	&\qquad \eps^{N-(k-2)-2}Y^{N-(k-2)-2}M_{k-1,N-(k-2)-2}+\eps^{N-(k-2)-1}R_{k-1,N-(k-2)-1},
	\end{split}
	\end{align}
	where the $M_{k-1,j}\in S^*$ are defined by
	\begin{align}\label{e3}
	M_{p,j}:=\partial^j_y\begin{pmatrix}H_{p}\\K_{p}\end{pmatrix}(t,x,0,\theta,Y)/j! \text{ for }p\geq 5,\; M_{p,j}:=0 \text{ for }p\leq 4.  
	\end{align}
	Noting that $L_{ff}U_k$ is part of the coefficient of $\eps^{k-2}$ in \eqref{d4},
	%and with a view toward replacing the right side of \eqref{e1}(a) by  $\begin{pmatrix}H'_{k-1}\\K'_{k-1}\end{pmatrix}\in S^*$, 
	we define\footnote{Here the term $YM_{k-2,1}$, for example, comes from the $\eps YM_{k-2,1}$ term in $\begin{pmatrix}H_{k-2}\\K_{k-2}\end{pmatrix}^{mod}$.}   
	%Although the construction of $\underline{U_{k-2}}$ makes $\underline{\begin{pmatrix}H_{k-1}\\K_{k-1}\end{pmatrix}}=0$, we do not want to \emph{assume} that term is zero in the definition \eqref{e4}.}
	\begin{align}\label{e4}
	\begin{pmatrix}H'_{k-1}\\K'_{k-1}\end{pmatrix}:=\underline{\begin{pmatrix}H_{k-1}\\K_{k-1}\end{pmatrix}}+\begin{pmatrix}H_{k-1}\\K_{k-1}\end{pmatrix}^*+M_{k-1,0}+YM_{k-2,1}+Y^2M_{k-3,2}+...+Y^{k-6}M_{5,k-6}.
	\end{align}

		\begin{proposition}\label{prop:UkExistUniq}
			For each $2\leq k\leq N$, there exists a profile $U_k\in S$ satisfying the equations \eqref{g3}$_k$.  
		\end{proposition}
	\begin{proof}
		\textbf{1. }The statement has been proved for $k=2,3$.  The $U_j$, $j\leq k-1$ are assumed to be known and to satisfy \eqref{g3}$_j$, $j\leq k-1$, so the $H'_j$, $K'_j$, $h_j$, $k_{j}$, $j\leq k-1$ are known, and we can use 
	 the formulas \eqref{eq:Uk0*}, \eqref{eq:UkP},  and \eqref{g4}, \eqref{ukhc} to construct the pieces 
	$U_k^{0*}$, $U_{k,p}$, and $U_{k,h}$.   The term $U_{k,\alpha}$ is given by \eqref{g5}, where $\alpha_k$ is provided by Proposition \ref{prop:Amplitudes2}.

		\textbf{2. }Finally, we construct $\underline{U_k}$, which is determined by line (e) of \eqref{g3}$_k$. Using \eqref{i2} we can use the definitions of $H_{k+1}$, $K_{k+1}$ to write the interior equation as

	%	The interior equation for $\underline{U}_k$ is derived from  $\underline{H}'_{k+1}=\underline{K}'_{k+1}=0$. Recall that $H'_{k+1}$ and $K'_{k+1}$ are derived from $H_{k+1}, K_{k+1}$ by replacing mixed terms with elements of $S^*$ and adding in corresponding terms from lower order mixed terms in the procedure discussed at the beginning of this section. Moreover, since both mixed terms and elements of $S^*$ limit to 0 as $Y\rightarrow\infty$, we have that:TODOmodify
	%	\begin{equation}
	%		\underline{H}_{k+1}=\underline{H}'_{k+1} \quad \underline{K}_{k+1}=\underline{K}'_{k+1}
	%	\end{equation}
		%We can partially expand the definition of $H_{k+1}$ and $K_{k+1}$ to get the following equation for $(\underline H_{k+1},\underline K_{k+1})$:
		\begin{equation}\label{eq:UkBarIntExpanded}
			\begin{split}
				&\underline{\begin{pmatrix}H'_{k+1}\\K'_{k+1}\end{pmatrix}}= \underline{-L_{ss}(U_k)-L_{fs}(U_{k+1})+A_{ffs}(U_k,U_2)+A_{ffs}(U_k,U_2)}+\\
				& \underline{B_{ffff}(U_k,U_2,U_2)+B_{ffff}(U_2,U_k,U_2)+B_{ffff}(U_2,U_2,U_k)+N_{k+1}(U_2,...,U_{k-1})}=0
			\end{split}
		\end{equation}
		where $N_{k+1}$ is a known nonlinear function.  As in earlier arguments this easily reduces to 
		\begin{equation}\label{eq:UkBarIntFin}
			L_{ss}(\underline{U}_k)=\underline{N_{k+1}(U_2,...,U_{k-1})}.
		\end{equation}
		%Observe that at this point, every term in the right hand side is known at this point. Moreover, 
		Starting at $H'_7,K'_7$ (or $k=6$), the function $\underline{N}_7$ is expected to be nonzero since $A_{sss}(U_3,U_3)$ contains terms  like $\partial_x[\partial_y \underline{u}_3\partial_x \underline{v}_3]\not=0$, which are normally nonzero.
		Observe that the function $N_{k+1}$ is different from the corresponding function in the expression for $(H_{k+1},K_{k+1})$, but 
	$\underline{N_{k+1}}$ is not different.

		Next, we look at the boundary conditions given by:
		\begin{equation}\label{eq:UkBarBdyV1}
			\int_{0}^{\infty}\begin{pmatrix}
				H'_k \\
				K'_k
			\end{pmatrix}^0 dY=\begin{pmatrix}
				h_k^0\\
				k_k^0
			\end{pmatrix}\text{ on }y=Y=0.
		\end{equation}
		Notice that here   the distinction between $H_k,K_k$ and $H'_k,K'_k$ has an effect,   because the terms $M_{k,j}$ in \eqref{e4} do not integrate to zero.      Observing (again) that  the terms of $H^{'}_k$, $K^{'}_k$ in which $U_k$ appears are exactly the same as the terms of $H_k$, $K_k$ in which $U_k$ appears, we can use \eqref{eq:CascadeIntRHS} to  write

		\begin{equation}
		\begin{split}
		&\int_{0}^{\infty}[-L_{fs}(U_k)+A_{fff}(U_k,U_2)+A_{fff}(U_2,U_k)+N_1(U_2,...,U_{k-1})]^0dY=\\
		&-[l_s(U_k)-Q_2(\partial_{\f};\partial_{\f})(U_k,U_2)-Q_2(\partial_{\f};\partial_{\f})(U_2,U_k)+N_2(U_2,...,U_{k-1})]^0
		\end{split}
		\end{equation}
		where the $N_j$'s are known nonlinear functions of the lower order profiles. An argument similar to the one  in section \ref{sec:AnalysisU3} allows us to simplify this to
		\begin{equation}\label{eq:UkBarBdyFin}
		l_s(\underline{U}_k)=-l_s(U_k^{0*})+\int_{0}^{\infty}L_{fs}(U_k^{0*})dY+[N(U_2,...,U_{k-1})]^0 \text{ on }y=Y=0.
		\end{equation}
		The equations  \eqref{eq:UkBarIntFin} and \eqref{eq:UkBarBdyFin} together with the initial condition $\underline{U}_k=0$ in $t\leq 0$ uniquely determine $\underline{U_k}$. This completes the construction of $U_k$ and the inductive step.
		\end{proof}

	\begin{theorem}\label{theo:Error}
	(a)   Assume $d=2$. 	Let $U_k$, $k=2,...,N$ be given by proposition \ref{prop:UkExistUniq}. Then the approximate solution $U_{a}^\varepsilon(t,x,y)=\sum_{k=2}^{N}\varepsilon^kU_k(t,x,y,\frac{x-ct}{\varepsilon},\frac{y}{\varepsilon})$ satisfies
			\begin{align}\label{eq:ModifiedInteriorError}
			\begin{split}
			&\partial_t^2U_{a}^\varepsilon+\nabla\cdot(L(\nabla U_{a}^\varepsilon)+Q(\nabla U_{a}^\varepsilon)+C(\nabla U_{a}^\varepsilon))=\varepsilon^{N-1} E'_N(t,x,y,\frac{x-ct}{\varepsilon},\frac{y}{\varepsilon})\text{ on }y>0\\
	&L_2(\nabla U_{a}^\varepsilon)+Q_2(\nabla U_{a}^\varepsilon)+C_2(\nabla U_{a}^\varepsilon)-\eps^2\begin{bmatrix}
				f\\
				g
			\end{bmatrix}=\varepsilon^Ne_N(t,x,0,\frac{x-ct}{\varepsilon},0)\text{ on }y=0,
		\end{split}
		\end{align}
		where  $E'_N, e_N\in S^e$.
	
	(b) Assume $d\geq 3$.  The same result holds, where now $U_k=U_k(t,x',x_d,\theta,Y)$ and 
	$$
	U^\eps_a= (\eps^2U_2+\dots+\eps^p U_p)|_{\theta=\frac{\beta\cdot (t,x')}{\eps},Y=\frac{x_d}{\eps}}.
	$$
Here $\beta=(-c|\underline{\eta}|,\underline{\eta})$ is a Rayleigh frequency as described in Remark \ref{h5}.

		\end{theorem}
	
	\begin{proof}
		\textbf{1. }First consider the interior equation.  Using \eqref{d4} we obtain\footnote{Here and below we suppress the evaluations $\theta=\frac{x-ct}{\varepsilon},Y=\frac{y}{\varepsilon}$.}
		\begin{equation}\label{eq:ErrorIntEq3}
		\begin{split}
			&\partial_t^2 U_{a}^\varepsilon+\nabla\cdot(L(\nabla U_{a}^\varepsilon)+Q(\nabla U_{a}^\varepsilon)+C(\nabla U_{a}^\varepsilon))=\\
			& \sum_{k=2}^{N} \varepsilon^{k-2}(L_{ff}(U_k)-\begin{pmatrix}
				H_{k-1}\\
				K_{k-1}
			\end{pmatrix})+\varepsilon^{N-1}E_N= \sum_{k=2}^{N}\eps^{k-2}(\begin{pmatrix}
				H'_{k-1}\\
				K'_{k-1}
			\end{pmatrix}) -\begin{pmatrix}
				H_{k-1}\\
				K_{k-1}
			\end{pmatrix})+\varepsilon^{N-1}E_N=\\
			& \sum_{k=6}^{N}\eps^{k-2}(\begin{pmatrix}
				H'_{k-1}\\
				K'_{k-1}
			\end{pmatrix}) -\begin{pmatrix}
				H^{mod}_{k-1}\\
				K^{mod}_{k-1}
			\end{pmatrix})+\varepsilon^{N-1}E_N,
		\end{split}
		\end{equation}
		where for the last equality we used \eqref{d3} and the fact that the $H_j, K_j$ are modified only for $j\geq 5$.    Substituting into the right side from \eqref{e2}, \eqref{e4} yields
		\begin{align}
		\begin{split}
		&\sum_{k=6}^N\sum_{l=0}^{k-6}\eps^{k-2}Y^lM_{k-1-l,l}-\sum_{k=6}^N\sum_{j=0}^{N-k}\eps^{k-2+j}Y^jM_{k-1,j}-\sum^N_{k=6}\eps^{N-1}R_{k-1,N-(k-2)-1}+\eps^{N-1}E_N=\\
		&\qquad \eps^{N-1}\left(E_N-\sum^N_{k=6}R_{k-1,N-k+1}\right):=\eps^{N-1}E_N'.
		\end{split}
		\end{align}

		\textbf{2. }On the boundary we have directly from \eqref{d5}
	       \begin{equation}
		\begin{split}
		&L_2(\nabla U_{a}^\varepsilon)+Q_2(\nabla U_{a}^\varepsilon)+C_2(\nabla U_{a}^\varepsilon)-\eps^2\begin{pmatrix}f\\g\end{pmatrix}=\\
		&\quad  \sum_{k=2}^{N} \varepsilon^{k-1}\left(l_{f}(U_k)-\begin{pmatrix}
				h_{k-1}\\
				k_{k-1}
			\end{pmatrix}\right)+\varepsilon^{N}e_N=\eps^Ne_N.
		\end{split}
		\end{equation}

		\textbf{3. }As already noted, the proof for $d\geq 3$ is just a repetition of that for $d=2$ with mainly notational changes.  For example,  in solving for $U_2$ one now uses the form of the amplitude equation given in Remark \ref{h5}.

		\end{proof}

Combining the results of Theorems \ref{main} and \ref{theo:Error} we obtain our main result for the SVK system:
 
 \begin{theorem}\label{maincor}
 Consider the traction problem in nonlinear elasticity \eqref{j0}, where $G(t,x',\theta)\in H^\infty([0,T_0]\times \mathbb{R}^{d-1}\times \mathbb{T})$, $d\geq 2$.   With 
 $\Omega:= (-\infty,T]\times \overline{R}^d_+$ 
 let $M>\frac{d}{2}+2$ 
and let 
\begin{align}\label{m1}
u^\eps_a(t,x)=\left(\eps^2U_2+\eps^3 U_3+\cdots+\eps^{M+1}U_{M+1}\right)|_{\theta=\frac{\beta\cdot (t,x')}{\eps},Y=\frac{x_d}{\eps}}\in H^\infty(\Omega)
\end{align}
be the approximate solution constructed in Theorem \ref{theo:Error} for $\eps\in (0,1]$ and some positive $T\leq T_0$.

(a) Suppose $s\geq [\frac{d}{2}]+6$.  There exist constants $\eps_0>0$ and $C>0$ such that for $\eps\in (0,\eps_0]$ the problem  \eqref{j0}  has a unique solution 
$u^\eps=u^\eps_a+v\in E^{s+2}(-\infty,T]$ such that $v^\eps$ satisfies the estimate
\begin{align}
|\eps^2 \overline{D}^2_x v^\eps(t)|_{s,\eps}+|\eps \overline{D} v^\eps(t)|_{s,\eps}+|\eps^2 \overline{D} \partial_t v^\eps(t) |_{s,\eps}\leq \eps^M C \text{ for }t\in [0,T]\text{ and }\eps\in (0,\eps_0].
\end{align}
In particular this implies
$|v^\eps|_{W^{1,\infty}(\Omega)}\leq C \eps^{M-\frac{d}{2}-1}$.

(b) With $s$ as in part (a), let $p$ be an integer $\geq 2$, choose $M$ such that $p<M-\frac{d}{2}-1$,  let $u^\eps_a$ be as in \eqref{m1},  and let $u^\eps=u^\eps_a+v$ be the exact solution as in part (a). Then we have
\begin{align}\label{bound2}
\left|u^\eps - (\eps^2U_2+\dots+\eps^p U_p)_{\theta=\frac{\beta\cdot (t,x')}{\eps},Y=\frac{x_d}{\eps}}\right|_{L^\infty(\Omega)} = o(\eps^p).
\end{align}
\end{theorem}

\begin{proof}
Using the definitions of the spaces $S$ and $S^e$ (section \ref{sec:Hyp}) and taking $R^\eps$, $r^\eps$ to be given by
\begin{align}
R^\eps(t,x)=-E'_{M+1}(t,x,\theta,Y)|_{\theta=\frac{\beta\cdot (t,x')}{\eps},Y=\frac{x_d}{\eps}},\;\;r^\eps(t,x')=-e_{M+1}(t,x',0,\theta,0)|_{\theta=\frac{\beta\cdot (t,x')}{\eps}},
\end{align}
where $E'_{M+1}$, $e_{M+1}$ are as in \eqref{eq:ModifiedInteriorError},
it is straightforward to check that Assumption \ref{ass1} is satisfied by $u^\eps_a$, $R^\eps$, $r^\eps$  for an appropriate choice of $A_1$, $A_2$.    Thus,  Remark \ref{NE} allows us to apply Theorem \ref{main} to prove part (a).  
We also have
\begin{align}
|u^\eps - (\eps^2U_2+\dots+\eps^p U_p)|_{L^\infty(\Omega)}\leq |\eps^{p+1}U_{p+1}+\dots+\eps^{M+1} U_{M+1})|_{L^\infty(\Omega)}+|v|_{L^\infty(\Omega)}
\end{align}
which gives \eqref{bound2}.

\end{proof}

%[\mathcal{E}^s_\eps(t)]^2\leq \int^t_0 e^{B_1(t-\sigma)}\left(\eps^2|R^\eps(\sigma)|^2_{s+1,\eps}+\langle r^\eps(\sigma)\rangle^2_

\part{Tame estimate for the amplitude equation}\label{tamewellp}

 In this section we give a new proof of  a  tame  a priori estimate for the amplitude equation \eqref{ae2}.   
This is the main step in obtaining a time of existence for very regular solutions that depends only on a \emph{fixed} low order of regularity.    The same proof provides a tame estimate for the pulse analogue of \eqref{ae2} considered by  \cite{Williams}.
% and we expect that the same methods may yield tame estimates for the problems considered in \cite{Hunter2006,Secchi}. 

\begin{prop}
\label{propwellposed}
Let $v \in \R^{d-1}$ be a fixed velocity vector, and assume that the kernel $b$ in \eqref{eq:BilFourMult} is 
symmetric with respect to its arguments and that there exists a constant $C>0$ such that
\begin{equation}
\label{hypotheseb}
\forall \, (n_1,n_2,n_3) \in \bbZ^3 \, ,\quad |b(n_1,n_2,n_3)|\leq C\min(|n_1n_2|, |n_2n_3|,|n_1n_3|) .
\end{equation}
Let $m\geq m_1>\frac{d}{2}+2$.    Then sufficiently smooth solutions of the Cauchy problem
%For every $R>0$ there exists a $T=T(m_1,R)$ such that if the initial 
%data $u_0 \in H^m(\R^{d-1}_y \times \bbT_\theta ;\R)$ satisfies $\| u_0 \|_{H^{m_1}} \le R$, then there exists a 
%unique $u \in {\mathcal C}([0,T];H^m(\R^{d-1}_y \times \bbT_\theta;\R))$ solution to the Cauchy problem
\begin{equation}
\label{cauchyamplitude}
\partial_t u +\sum_{j=1}^{d-1} v_j \, \partial_j u +{\mathcal H} \, \big( {\mathcal B}(u,u) \big) =0 \, ,\quad 
u|_{t=0}=u_0 \, .
\end{equation}
satisfy the estimate
\begin{align}\label{tame}
\left|\dfrac{{\rm d}}{{\rm d}t} \| u(t) \|_{H^m}^2\right|  \le C \, \| u(t) \|_{H^m}^2 \|u(t)\|_{H^{m_1}}.
\end{align}
\end{prop}

%When the initial condition for \eqref{cauchyamplitude} vanishes but the source term $g$  is nonzero, one 
%solves \eqref{cauchyamplitude} by using the Duhamel formula, thereby proving Proposition \ref{prop:Amplitudes2} . We omit the details of that case here and focus on the solvability of 
%the pure Cauchy problem for nonzero initial data and zero forcing term.
 In isotropic elastodynamics, the 
kernel $b$ that appears in \eqref{eq:BilFourMult}  satisfies the bound \eqref{hypotheseb} as can be seen immediately  by  inspection of the basic kernels written down in formulas (2.56)-(2.58) of \cite{Williams} (or formula (3.20) of \cite{Hunter06}).

The proof of proposition \ref{propwellposed} uses the following lemma from \cite{RR}.    Here and below,  integration with respect to $k$ or $l$ is summation over $\mathbb{Z}$. 
 \begin{lem}\label{f0a}
Suppose $G:(\mathbb{R}^{d-1}\times \mathbb{Z})\times (\mathbb{R}^{d-1}\times \mathbb{Z})\to\mathbb{C}$ is a locally integrable measurable function that can be decomposed into a finite sum 
\begin{align}
G(\xi,k,\eta,l)=\sum_{j=1}^K G_j(\xi,k,\eta,l)
\end{align}
such that for each $j$ we have either
\begin{align}
\sup_{\xi,k}\int |G_j(\xi,k,\eta,l)|^2 d(\eta ,l)< C\text{ or }\sup_{\eta,l}\int |G_j(\xi,k,\eta,l)|^2 d(\xi ,k)< C.
\end{align}
 Then
\begin{align}
(f,g)\to \int G(\xi,k,\eta,l) f(\xi-\eta,k-l)g(\eta,l)d(\eta ,l)
\end{align}
defines a continuous bilinear map of $L^2\times L^2\to L^2$, and 
\begin{align}
 |\int G(\xi,k,\eta,l) f(\xi-\eta,k-l)g(\eta,l)d(\eta, l)|_{L^2}\leq C |f|_{L^2} |g|_{L^2}.
\end{align}
\end{lem}

%\begin{proof}
\begin{proof}[Proof of Proposition \ref{propwellposed}]
\textbf{1. }
%The main step is to prove a tame a priori estimate for 
 %sufficiently smooth solutions to the Cauchy problem of the form 
%\begin{align}\label{tame}
%\left|\dfrac{{\rm d}}{{\rm d}t} \| u(t) \|_{H^m}^2\right|  \le C \, \| u(t) \|_{H^m}^2 \|u(t)\|_{H^{m_1}}.
%\end{align}
 %By standard arguments, see for instance \cite{TaylorIII,B}, a priori 
%estimates can be turned into a well-posedness result as stated in Proposition \ref{propwellposed} by 
%using convenient Fourier truncation approximations (recall here that the underlying space domain is 
%$\R^{d-1}_y \times \bbT_\theta$ so Fourier analysis is readily available).
We  consider a solution $u$ to the
%\in {\mathcal C}([0,T];H^m(\R^{d-1}_y \times \R_\theta;\R))$ to the 
Cauchy problem \eqref{cauchyamplitude} 
%and try to derive an estimate for the evolution of the $H^m$ 
%norm of $u$. We can assume that  $u$ 
that is sufficiently smooth for all manipulations below to be rigorous. 
Using the Fourier expression of the $H^m$ norm, we see that is enough to estimate just the $L^2$ 
norms of the functions $u$, $\partial_\theta^m u$, $\partial_y^\alpha u$ with $|\alpha|=m$ ($\alpha \in 
\bbN^{d-1}$). All other partial derivatives of $u$ can be dealt with by interpolating between such `extreme' 
cases. 
%Once again, we refer to \cite{TaylorIII} for the details of such arguments. 
Let us first prove the 
following bounds on the operator ${\mathcal B}$.\footnote{Lemma \ref{lembornesB} was proved in \cite{Williams} for the case of pulses.   The remainder of this proof differs from the argument in \cite{Williams}; the argument in \cite{Williams} gave only a nontame estimate for $\frac{d}{dt}\|u\|^2_{H^m}$.}

\begin{lem}
\label{lembornesB}
Under the assumptions of Proposition \ref{propwellposed}, the bilinear operator ${\mathcal B}$ is symmetric. 
It satisfies the Leibniz rule
$$
\partial_\theta {\mathcal B}(u,v) ={\mathcal B}(\partial_\theta u,v) +{\mathcal B}(u,\partial_\theta v) \, ,
$$
and more generally the Leibniz rule at any order of differentiation in $\theta$, as well as the 
bounds\footnote{The bounds obviously extend by continuity to functions in appropriate Sobolev spaces 
and are not restricted to functions in the Schwartz class.}
\begin{align*}
%\forall \, u,v,w \in {\mathcal S}(\R^{d-1} \times \R;\R) \, ,\quad 
%\left| \int_{\R^{d-1} \times \R} u \, {\mathcal H} \, \big( {\mathcal B}(v,w) \big) \, {\rm d}y \, {\rm d}\theta \right| 
%&\le C \, \| u \|_{L^2} \, \| v \|_{H^1} \, \| w \|_{H^{m_0}} \, ,\\
\forall \, u,v \in {\mathcal S}(\R^{d-1} \times \bbT;\R) \, ,\quad 
\left| \int_{\R^{d-1} \times \bbT} u \, {\mathcal H} \, \big( {\mathcal B}(u,v) \big)\, {\rm d}y \, {\rm d}\theta \right| 
&\le C \, \| v \|_{H^{m_1}} \, \| u \|_{L^2}^2 \, ,
\end{align*}
for a suitable constant $C$ and any integer $m_1$ satisfying $m_1 >\frac{d}{2} +2$. (The Sobolev norms refer 
to the space domain $\R^{d-1}_y \times \bbT_\theta$.)
\end{lem}

%\begin{proof}
The fact that ${\mathcal B}$ is symmetric comes from the symmetry of the kernel $b$ with respect to its 
three arguments. We now consider  functions $u,v$ in the Schwartz space ${\mathcal S}(\R^{d-1}_y 
\times \bbT_\theta;\R)$.
We will take advantage of some cancelation arising from the skew-symmetric 
operator ${\mathcal H}$. We compute
%\footnote{Here and below integrals $\int_\bbZ (...)dk$ should be interpreted as sums over $k$.}
\begin{multline*}
\int_{\R^{d-1} \times \bbT} u \, {\mathcal H} \, \big( {\mathcal B}(u,v) \big) \, {\rm d}y \, {\rm d}\theta 
=i\, \int_{\R^{d-1} \times \bbZ \times \bbZ} \overline{\widehat{u}(y,k)} \, \widehat{u}(y,k-l) \, \widehat{v}(y,l) \, 
\text{\rm sgn} (-k) \, b(-k,k-l,l) \, {\rm d}y \, {\rm d}k \, {\rm d}l \\
=i\, \int_{\R^{d-1} \times \bbZ \times \bbZ} \widehat{u}(y,-k) \, \widehat{u}(y,k-l) \, \widehat{v}(y,l) \, 
\text{\rm sgn} (-k) \, b(-k,k-l,l) \, {\rm d}y \, {\rm d}k \, {\rm d}l \\
=\dfrac{i}{2}\, \int_{\R^{d-1} \times \bbZ \times \bbZ} \widehat{u}(y,-k) \, \widehat{u}(y,k-l) \, \widehat{v}(y,l) \, 
\Big( \text{\rm sgn} (-k) +\text{\rm sgn} (k-l) \Big) \, b(-k,k-l,l) \, {\rm d}y \, {\rm d}k \, {\rm d}l\, ,
\end{multline*}
where we have used the fact that $u$ is real valued and the symmetry of $b$. Let us observe 
that if $-k$ and $k-l$ have opposite signs, then the quantity $\text{\rm sgn} (-k) +\text{\rm sgn} (k-l)$ 
vanishes. If $-k$ and $k-l$ have the same sign, then the sum of signs is either $2$ or $-2$, and there 
holds
$$
|k| \le |l| \, ,\quad \text{\rm and } \quad |k-l| \le |l| \, .
$$
With \eqref{hypotheseb} this yields
\begin{multline*}
\Big| \big( \text{\rm sgn} (-k) +\text{\rm sgn} (k-l) \big) \, b(-k,k-l,l) \Big| 
\le C  | \text{\rm sgn} (-k) +\text{\rm sgn} (k-l)| |k-l||l| \le C  |l|^2 \, .
\end{multline*}
Using the Cauchy-Schwarz and $L^2-L^1$ convolution inequalities, we derive the bound
\begin{align*}
\left| \int_{\R^{d-1} \times \bbT} u \, {\mathcal H} \, \big( {\mathcal B}(u,v) \big) \, {\rm d}y \, {\rm d}\theta \right| 
&\le C \, \int_{\R^{d-1} \times \bbZ \times \bbZ} | \widehat{u}(y,-k)| \, |\widehat{u}(y,k-l)| \, |l|^2 \, |\widehat{v}(y,l)| 
\, {\rm d}y \, {\rm d}k \, {\rm d}l\\
&\le C \, \int_{\R^{d-1}} \| \widehat{u}(y,\cdot) \|_{L^2} \, \| \widehat{u}(y,\cdot) \|_{L^2} \, 
\| k^2\, \widehat{v}(y,\cdot) \|_{L^1} \, {\rm d}y \\
&\le C \, \| u \|_{L^2}^2 \, \sup_{y \in \R^{d-1}} \| v(y,\cdot) \|_{H^q(\bbT)} \, \text{ for any }q>\frac{5}{2}.
\end{align*}
Applying the Sobolev imbedding Theorem completes the proof of Lemma \ref{lembornesB}.
%\end{proof}

\textbf{2. }Let us consider an integer $m \ge m_1$ with $m_1$ as in Lemma \ref{lembornesB}. We consider a 
sufficiently smooth solution $u$ to \eqref{cauchyamplitude} and compute (the transport terms with 
respect to the variables $y$ can be removed by a change of $(t,y)$ variables):
$$
\dfrac{{\rm d}}{{\rm d}t} \| u(t) \|_{L^2}^2 =-2 \, \int_{\R^{d-1} \times \R} u \, {\mathcal H} \, {\mathcal B}(u,u) 
\, {\rm d}y \, {\rm d}\theta \, .
$$
Applying the bound in Lemma \ref{lembornesB}, we get
\begin{equation}
\label{apriori1}
\left| \dfrac{{\rm d}}{{\rm d}t} \| u(t) \|^2_{L^2} \right| \le C \, \| u(t) \|^2_{L^2} \,  \| u(t) \|_{H^{m_1}} \leq  C\| u(t) \|^2_{H^m} \,  \| u(t) \|_{H^{m_1}}
\end{equation}

\textbf{3. }Let us now differentiate  \eqref{cauchyamplitude} $m$ times  with respect to $\theta$, and get
$$
\partial_t \partial_\theta^m u +\sum_{j=1}^{d-1} v_j \, \partial_j \partial_\theta^m u +2\, {\mathcal H} \, \big( 
{\mathcal B}(\partial_\theta^m u,u) \big) =-\sum_{m'=1}^{m-1} \binom {m} {m'} \, 
{\mathcal H} \, \big( {\mathcal B}(\partial_\theta^{m'} u,\partial_\theta^{m-m'} u) \big) \, .
$$
Taking the $L^2$ scalar product with $\partial_\theta^m u$, we get
\begin{align}\label{diff}
\dfrac{{\rm d}}{{\rm d}t} \| \partial_\theta^m u(t) \|_{L^2}^2 =&-4 \, \int_{\R^{d-1} \times \bbT} \partial_\theta^m u \, 
{\mathcal H} \, \big( {\mathcal B}(\partial_\theta^m u,u) \big) \, {\rm d}y \, {\rm d}\theta \\
&-2 \, \sum_{m'=1}^{m-1} \binom {m} {m'} \, \int_{\R^{d-1} \times \bbT} \partial_\theta^m u \, 
{\mathcal H} \, \big( {\mathcal B}(\partial_\theta^{m'} u,\partial_\theta^{m-m'} u) \big) \, {\rm d}y \, {\rm d}\theta \, .
\end{align}
For the first integral we apply the estimate of Lemma \ref{lembornesB} and get
$$
\left| \int_{\R^{d-1} \times \R} \partial_\theta^m u \, {\mathcal H} \, \big( {\mathcal B}(\partial_\theta^m u,u) 
\big) \, {\rm d}y \, {\rm d}\theta \right| \le C \, \| u(t) \|_{H^{m_1}} \, \| u(t) \|_{H^m}^2. 
$$

\textbf{4.  }Now we consider the remaining terms in \eqref{diff}. We let $\xi$ or $\eta$  denote Fourier variables dual to $y$.  
%The notation $\int ...dk$ denotes a sum over $k$ for wavetrains, an integral over $k$ for pulses. 
Assuming without loss of 
generality $m-1 \ge m' \ge m-m' \ge 1$,  and letting $\chi_A(k,l)$, $\chi_B(k,l)$ be the characteristic functions of $\{|k|\leq |l|\}$, $\{|l| < |k|\}$ respectively, we consider one of the remaining terms\begin{align}\label{b}
\begin{split}
&\left| \int_{\R^{d-1} \times \R} \partial_\theta^m u \, {\mathcal H} \, \big( {\mathcal B}(\partial_\theta^{m'} u, 
\partial_\theta^{m-m'} u) \big) \, {\rm d}y \, {\rm d}\theta \right|=\\
&\left | \int k^m\hat u(\xi,k)(k-l)^{m'}\hat u(\xi-\eta,k-l) l^{m-m'}\hat u(\eta,l)b(-k,k-l,l)[\chi_A(k,l)+\chi_B(k,l)]d(\xi,k)d(\eta,l)\right| \leq \\
&\qquad |\int F_A(\xi,k,\eta,l)d(\xi,k)d(\eta,l)|+|\int F_B(\xi,k,\eta,l)d(\xi,k)d(\eta,l)|:= \mathcal{A}+\mathcal{B},
\end{split}
\end{align}
%:=\left | \int k^m\hat u(\xi,k) H_1(\xi,k)d(\xi,k)\right|.
where $F_A$, $F_B$ have the obvious definitions.\footnote{We replace $\sgn{(-k)}$ by one in these estimates.}

\textbf{5. }We can estimate $\mathcal{A}$ by Cauchy-Schwarz after estimating the $L^2(\xi,k)$ norm of 
\begin{align}\label{c}
H_A(\xi,k):=\int(k-l)^{m'}\hat u(\xi-\eta,k-l) l^{m-m'}\hat u(\eta,l)b(-k,k-l,l)\chi_A(k,l)d(\eta,l).
\end{align}
For this we apply     
  lemma \ref{f0a} to the kernel 
\begin{align}
G_{A}(\xi,k,\eta,l):=\frac{|k-l|^{m'}|l|^{m-m'}|l||k-l|\chi_A}{\langle \eta,l\rangle^m\langle\xi-\eta,k-l\rangle^{m_1}}\lesssim \frac{1}{\langle\xi-\eta,k-l\rangle^{m_1-2}}.
\end{align}
Here we have used \eqref{hypotheseb}, the fact that $m'\geq 1$, and the fact that on supp $\chi_A$ we have $|k-l|\leq 2|l|$. Observe that this estimate does not work if $m'=0$. 
This gives
\begin{align}\label{d}
\mathcal{A}\lesssim \| u(t) \|_{H^m}^2 \|u(t)\|_{H^{m_1}}.
\end{align}

\textbf{6. }To estimate $\mathcal{B}$ we use $|k|^m\lesssim |l|^m+|k-l|^m$ to write
\begin{align} 
\begin{split}
&\mathcal{B}\leq \left | \int l^m\hat u(\xi,k)(k-l)^{m'}\hat u(\xi-\eta,k-l) l^{m-m'}\hat u(\eta,l)b(-k,k-l,l)\chi_B(k,l)d(\xi,k)d(\eta,l)\right|+\\
&\quad  \left | \int (k-l)^m\hat u(\xi,k)(k-l)^{m'}\hat u(\xi-\eta,k-l) l^{m-m'}\hat u(\eta,l)b(-k,k-l,l)\chi_B(k,l)d(\xi,k)d(\eta,l)\right|:=
 \mathcal{B}_1+\mathcal{B}_2.
\end{split}
\end{align}

\textbf{7. Estimate of $\mathcal{B}_1$. }Pairing $l^m$ with $\hat u(\eta,l)$, we can use Cauchy-Schwarz to estimate $\mathcal{B}_1$ after  estimating the $L^2(\eta,l)$ norm of 
\begin{align}\label{e}
H_{B1}(\eta,l):=\int(k-l)^{m'}\hat u(\xi-\eta,k-l) l^{m-m'}\hat u(\xi,k)b(-k,k-l,l)\chi_B(k,l)d(\xi,k).
\end{align}
To do this we apply     
  lemma \ref{f0a} to the kernel 
\begin{align}
G_{B1}(\xi,k,\eta,l):=\frac{|k-l|^{m'}|l|^{m-m'}|l||k-l|\chi_B}{\langle \xi,k\rangle^m\langle\xi-\eta,k-l\rangle^{m_1}}\lesssim \frac{1}{\langle\xi-\eta,k-l\rangle^{m_1-2}}.
\end{align}
Here we have used \eqref{hypotheseb}, the fact that $m'\geq 1$, and the fact that on supp $\chi_B$ we have $|k-l|\leq 2|k|$, $|l|\leq |k|$.   This gives
\begin{align}\label{f}
\mathcal{B}_1\lesssim \| u(t) \|_{H^m}^2 \|u(t)\|_{H^{m_1}}.
\end{align}

\textbf{8. Estimate of $\mathcal{B}_2$. }In the integral that defines $\mathcal{B}_2$ make the change of variables
\begin{align}\label{g}
(\xi,k,\eta, l)\to (\xi,k,\alpha,p) \text{ where }\alpha=\xi-\eta, p=k-l
\end{align}
to obtain
\begin{align}
\left | \int p^m\hat u(\xi,k)p^{m'}\hat u(\alpha,p) (k-p)^{m-m'}\hat u(\xi-\alpha,k-p)b(-k,p,k-p)\chi_B(k,k-p)d(\xi,k)d(\alpha,p)\right|
\end{align}
Pairing $p^m$ with $\hat u(\alpha,p)$, we can use Cauchy-Schwarz to estimate $\mathcal{B}_2$ after estimating the $L^2(\alpha,p)$ norm of 
\begin{align}\label{h}
H_{B2}(\alpha,p):=\int p^{m'}\hat u(\xi-\alpha,k-p) (k-p)^{m-m'}\hat u(\xi,k)b(-k,p,k-p)\chi_B(k,k-p)d(\xi,k).
\end{align}
For this we apply     
 lemma \ref{f0a} to the kernel 
\begin{align}
G_{B2}(\xi,k,\alpha,p):=\frac{|k|^{m'}|k-p|^{m-m'}|k||k-p|\chi_B}{\langle \xi,k\rangle^m\langle\xi-\alpha,k-p\rangle^{m_1}}\lesssim \frac{1}{\langle\xi-\alpha,k-p\rangle^{m_1-2}}.
\end{align}
Here we have used \eqref{hypotheseb}, the fact that $m-m'\geq 1$, and the fact that on supp $\chi_B$ we have $|k-p|\leq |k|$, $|p|\leq 2|k|$. \footnote{This estimate does not work if $m'=m$.}  This gives
\begin{align}\label{i}
\mathcal{B}_2\lesssim \| u(t) \|_{H^m}^2 \|u(t)\|_{H^{m_1}}.
\end{align}

\textbf{9. }  These estimates go through unchanged if factors like $(k-l)^{m'}$, $l^{m-m'}$ are replaced by $\langle k-l\rangle^{m'}$, $\langle l\rangle^{m-m'}$.

The $y$-partial derivatives of $u$ are estimated in an entirely similar way. The factors $k^m$, $(k-l)^{m'}$, and $l^{m-m'}$ in \eqref{b} are now replaced by 
$\langle \xi\rangle^{m}$, $\langle \xi-\eta\rangle^{m'}$, $\langle \eta\rangle^{m-m'}$.     The dichotomy $\{|k|\leq |l|\}$, $\{|l| < |k|\}$ is replaced by $\{|\xi|\leq |\eta|\}$, $\{|\eta| < |\xi|\}$ with respective characteristic functions $\chi_A(\xi,\eta)$, $\chi_B(\xi,\eta)$.

%3.)For the estimates in steps 2-6 to work we need $m_1>\frac{d}{2}+2$.  For the estimate \eqref{ab}  you needed $m_0+1>\frac{d}{2}+\frac{5}{2}$.    Thus, we require 
%$m_1>\frac{d}{2}+\frac{5}{2}$.

Putting these estimates together gives the estimate \eqref{tame} for $m\geq m_1>\frac{d}{2}+2$.

\end{proof}

\bibliographystyle{alpha}
\bibliography{RayleighWT}

\end{document}